\documentclass[a4paper,11pt]{article}
\usepackage{adjustbox}
\usepackage{aligned-overset}
\usepackage{amsmath,amsthm}
\usepackage{authblk}
\usepackage[style=numeric-comp,maxbibnames=99,bibencoding=utf8,giveninits=true]{biblatex}
\usepackage{booktabs}
\usepackage[hmargin={26mm,26mm},vmargin={30mm,35mm}]{geometry}
\usepackage{graphicx}
\usepackage{subcaption}
\graphicspath{ {./figures/} }
\usepackage{cases}
\usepackage{enumitem}
\usepackage{hyperref}
\usepackage{newtxtext}
\usepackage{newtxmath}
\usepackage{nicefrac}
\usepackage{xcolor}

\newcommand{\Real}{\mathbb{R}}
\newcommand{\Natural}{\mathbb{N}}

\newcommand{\scalarPoly}[1]{\mathcal{P}^{#1}}
\newcommand{\vectPoly}[1]{\boldsymbol{\mathcal{P}}^{#1}}

\newcommand{\scalarSobolev}[3]{W^{{#1},{#2}}({#3})}
\newcommand{\vectSobolev}[3]{\boldsymbol{W}^{{#1},{#2}}({#3})}

\newcommand{\Hdiv}[1]{\boldsymbol{H}(\operatorname{div};#1)}

\newcommand{\Vh}[1]{\boldsymbol{V}_{h}^{#1}}
\newcommand{\Zh}[1]{\boldsymbol{Z}_{h}^{#1}}
\newcommand{\Qh}[1]{Q_h^{#1}}

\newcommand{\tF}{t_{\rm F}}

\newcommand{\elements}[1]{\mathcal{T}_{#1}}
\newcommand{\faces}[1]{\mathcal{F}_{#1}}
\newcommand{\Th}{\elements{h}}
\newcommand{\Thd}{\elements{h}^{{\rm d}, n}}
\newcommand{\Thc}{\elements{h}^{{\rm a}, n}}
\newcommand{\Fh}{\faces{h}}
\newcommand{\Fhd}{\faces{h}^{{\rm d}, n}}
\newcommand{\Fhc}{\faces{h}^{{\rm a}, n}}
\newcommand{\Fhi}{\faces{h}^{\rm i}}
\newcommand{\TT}{\elements{T}}
\newcommand{\TF}{\elements{F}}
\newcommand{\FT}{\faces{T}}
\newcommand{\normal}{\boldsymbol{n}}
\newcommand{\normalF}{\boldsymbol{n}_F}
\newcommand{\tstep}{\Delta t}

\newcommand{\RT}{\boldsymbol{\mathcal{RT}}}
\newcommand{\BDM}{\boldsymbol{\mathcal{BDM}}}
\newcommand{\NE}{\boldsymbol{\mathcal{N}}}

\newcommand{\jump}[1]{\llbracket #1\rrbracket}
\newcommand{\avg}[1]{\{\!\!\{ #1 \}\!\!\}}
\newcommand{\rF}{\boldsymbol{r}_F^k}
\newcommand{\Rh}{\boldsymbol{R}_h^{k}}
\newcommand{\Gh}{\boldsymbol{G}_h^{k}}
\newcommand{\projRT}[1]{\boldsymbol{I}_{\RT,h}^{#1}}
\newcommand{\projRTEle}[1]{\boldsymbol{I}_{\RT,T}^{#1}}
\newcommand{\projL}[1]{\boldsymbol{\pi}_{0,h}^{#1}}
\newcommand{\projLEle}[1]{\boldsymbol{\pi}_{0,T}^{#1}}
\newcommand{\restr}[2]{#1|_{#2}}

\newcommand{\norm}[2]{\|#2\|_{#1}}
\newcommand{\seminorm}[2]{|#2|_{#1}}
\newcommand{\Norm}[2]{\left\|#2\right\|_{#1}}

\newcommand{\rbracket}[1]{\left( #1 \right)}

\newcommand{\cbracket}[1]{\left\{ #1 \right\}}

\newcommand{\timeTerm}[1]{\mathcal{T}^{#1}}
\newcommand{\diffTerm}[1]{\mathcal{D}^{#1}}
\newcommand{\convTerm}[1]{\mathcal{C}^{#1}}
\newcommand{\upwTerm}[1]{\mathcal{J}^{#1}}

\newcommand{\timeErr}{\mathcal{E}_{\mathcal{T}}^n}
\newcommand{\diffErr}{\mathcal{E}_{\mathcal{D}}^n}
\newcommand{\convErr}{\mathcal{E}_{\mathcal{C}}^n}
\newcommand{\upwErr}{\mathcal{E}_{\mathcal{J}}^n}

\newcommand{\Err}{\mathcal{E}^n}

\newcommand{\st}{\, : \,}
\newcommand{\email}[1]{\href{mailto:#1}{#1}}

\newtheorem{theorem}{Theorem}
\newtheorem{proposition}[theorem]{Proposition}
\newtheorem{lemma}[theorem]{Lemma}

\theoremstyle{remark}
\newtheorem{remark}[theorem]{Remark}
\theoremstyle{definition}
\newtheorem{assumption}[theorem]{Assumption}

\bibliography{pnsrt}

\begin{document}

\title{A Reynolds-semi-robust $\boldsymbol{H}(\operatorname{div})$-conforming method for unsteady incompressible power-law flows}

\author[1,2]{Louren\c{c}o Beir\~{a}o da Veiga}
\affil[1]{Dipartimento di Matematica e Applicazioni, Universit\`{a} di Milano Bicocca,
  \email{lourenco.beirao@unimib.it}, \ \email{k.haile@campus.unimib.it}}
\affil[2]{IMATI-PV, CNR, Pavia, Italy}
\author[3]{Daniele A. Di Pietro}
\affil[3]{IMAG, Univ. Montpellier, CNRS, Montpellier, France, \email{daniele.di-pietro@umontpellier.fr}}
\author[1]{Kirubell B. Haile}
\date{}
\maketitle

\begin{abstract}
  In this work, we prove what appear to be the first Reynolds-semi-robust and pressure-robust velocity error estimates for an $\boldsymbol{H}(\operatorname{div})$-conforming approximation of unsteady incompressible flows of power-law type fluids.
  The proposed method hinges on a discontinuous Galerkin approximation of the viscous term and a reinforced upwind-type stabilization of the convective term.
  The derived velocity error estimates account for pre-asymptotic orders of convergence observed in convection-dominated flows through regime-dependent estimates of the error contributions.
  A complete set of numerical results validate the theoretical findings.
  \smallskip\\
  \textbf{Keywords:}
  Navier–Stokes equations,
  non-Newtonian fluids,
  $\boldsymbol{H}(\operatorname{div})$-conforming methods,
  Reynolds-semi-robust estimates,
  pre-asymptotic convergence rates
  \smallskip\\
  \textbf{MSC2010 classification:} 76A05, 76D05, 65N30, 65N08, 65N12
\end{abstract}

\section{Introduction}\label{sec:Intro}

In this work we prove what appear to be the first Reynolds-semi-robust and pressure-robust error estimates for an $\boldsymbol{H}(\operatorname{div})$-conforming finite element approximation of the $p$-Navier--Stokes equations.

Fluids with nonlinear rheologies are encountered in several fields, ranging from geosciences \cite{Isaac.Stadler.ea:15,Ahlkrona.Elfverson:21,Schubert.Turcotte.ea:01} to chemical engineering \cite{Ko.Pustejovska.ea:18} and biomechanics \cite{Lai.Kuei.ea:78,Galdi.Rannacher.ea:08}.
The mathematical study of the corresponding system of equations is considered, e.g., in \cite{Ladyzhenskaya:69,Malek.Rajagopal:05,Ruzicka.Diening:07,Diening.Ettwein:08,Beirao-da-Veiga:09,Berselli.Ruzicka:20}.
Several methods have also been developed for its numerical solution.
Finite element methods for creeping flows of non-Newtonian fluids have been considered in \cite{Barrett.Liu:94,Hirn:13,Belenki.Berselli.ea:12}.
Both finite difference and finite element methods for the full non-Newtonian Navier--Stokes equations are considered in the monograph \cite{Crochet.Davies.ea:12}; see also~\cite{Crochet.Walters:83}.
Other contributions worth mentioning here include:
\cite{Diening.Kreuzer.ea:13,Kreuzer.Suli:16} on finite element methods with implicit power-law-like rheologies;
\cite{Ko.Pustejovska.ea:18,Ko.Suli:18} on generalized Newtonian fluids with space variable and concentration-dependent power-law index;
\cite{Kroner.Ruzicka.ea:14} as well as the paper series~\cite{Kaltenbach.Ruzicka:23,Kaltenbach.Ruzicka:23*1,Kaltenbach.Ruzicka:23*2} concerning various types of discontinuous Galerkin approximations;
\cite{Janecka.Malek.ea:19} on transient flows of non-Newtonian fluids.
\smallskip

The numerical solution of the $p$-Navier--Stokes problem involves several challenges.
The first is obviously related to the non-linearity of the viscous term.
The discretization considered here is inspired by discontinuous Galerkin techniques; see, e.g., \cite{Burman.Ern:08,Del-Pezzo.Lombardi.ea:12,Diening.Koner.ea:14,Malkmus.Ruzicka.ea:18}.
We also borrow ideas from the analysis of the Hybrid High-Order approximations of the $p$-Navier-Stokes considered in \cite{Botti.Castanon-Quiroz.ea:21,Castanon-Quiroz.Di-Pietro.ea:23}, based on the approach developed in \cite{Di-Pietro.Droniou:17,Di-Pietro.Droniou:17*1} for the scalar case, and from the theoretical developments for the VEM method in \cite{antonietti2024virtual}.

A second challenge is related to robustness in the convection-dominated regime.
It is classically known that, in this case, non-dissipative approximations of the convection term can lead to a loss of stability (see for instance \cite{Girault.Raviart:86,John:16,john2018finite}).
Remedies to this problem include, but are not limited to:
the streamline upwind Petrov--Galerkin method and its variants \cite{Franca.Freym:82,Brooks.Hughes:82,Tobiska.Verfurt:96,Beirao-da-Veiga.Dassi.ea:23};
the continuous interior penalty method \cite{Burman.Fernandez.ea:06,Burman.Fernandez:07};
grad-div stabilizations \cite{Olshanskii.Lube.ea:09,De-Frutos.Garcia-Archilla.ea:19}.
A scheme is referred to as \emph{Reynolds quasi-robust} if, assuming sufficient regularity on the solution, it admits velocity error estimates that do not depend on inverse powers of the viscosity coefficient. See \cite{volker2021} for a survey on Reynolds quasi-robust finite element methods for the classical incompressible Navier–Stokes equations.

A third challenging aspect, originally pointed out in~\cite{Linke:14},
is \emph{pressure-robustness}.
At the continuous level, irrotational body forces do not affect the velocity field.
Pressure-robust schemes preserve this property at the discrete level.
From a mathematical standpoint, this translates into velocity error estimates that are independent of the pressure.
Achieving pressure robustness requires to embed a form of compatibility with the de Rham complex in the discretization of the forcing term, either by using $\boldsymbol{H}(\operatorname{div})$-conforming (projections of) the test function as, e.g., in \cite{Linke:14,John.Linke.ea:17,Di-Pietro.Ern.ea:16,Castanon-Quiroz.Di-Pietro:20,li2023ema,Castanon-Quiroz.Di-Pietro:24}, or by interpolating in a way compatible with the Helmholtz--Hodge decomposition \cite{Beirao-da-Veiga.Dassi.ea:22,Di-Pietro.Droniou.ea:24}.
\smallskip

Recent works \cite{Schroeder.Lube:18,Barrenechea.Burman.ea:20,Han.Hou:21,Beirao-da-Veiga.Di-Pietro.ea:25} have pointed out the possibility to achieve both Reynolds-semi-robustess and pressure-robustness through the use of $\boldsymbol{H}(\operatorname{div})$-conforming approximations of the velocity for incompressible flows of Newtonian fluids, with developments also for more advanced problems such as magnetohydrodynamics \cite{beirao2024robust,da2024pressure}.
Furthermore, in \cite{Beirao-da-Veiga.Di-Pietro.ea:24} the authors developed regime-dependent estimates based on (local) dimensionless numbers for a robust Discontinuous Galerkin (DG) discretization of $p$-type diffusion with advection, underlying for the first time the delicate interplay among these two physics.

In this work, we combine these recent advances to establish what are, to the best of our knowledge, the first Reynolds-semi-robust and pressure-robust velocity error estimates for an approximation of the $p$-Navier--Stokes problem based on Brezzi--Douglas--Marini velocity and discontinuous pressure.
These estimates additionally account for pre-asymptotic convergence rates depending on the power-law exponent and on the convection- or diffusion-dominated nature of each element and face.
\smallskip

The rest of this work is organized as follows.
In Section~\ref{sec:Discrete setting} we establish the continuous and discrete setting.
The fully discrete problem is stated in Section~\ref{sec: Discrete problem}, which also contains a study of its stability.
The Reynolds-semi-robust and pressure-robust velocity error estimate is proved in Section~\ref{sec:Velocity error analysis}.
Finally, a complete numerical validation is provided in Section~\ref{sec:Numerical tests}, where we consider both manufactured solutions to validate the theoretical findings and tests meant to assess the behaviour of the scheme in more physical situations.

\section{Setting} \label{sec:Discrete setting}

\subsection{Continuous problem} \label{sec:p-Navier-Stokes-Problem}

Let $\Omega\subset\Real^d$, with $d\in \{2, 3 \}$, denote an open, bounded, connected polygonal or polyhedral domain, and let $\tF > 0$ denote a final time. We consider the evolutionary incompressible Navier--Stokes flow of a non-Newtonian fluid:
\smallskip\\
Find the velocity field $ \boldsymbol{u}:[0,\tF] \times \Omega \to \mathbb{R}^d $ and the pressure field $ p:(0,\tF\rbrack \times \Omega \to \mathbb{R} $  such that
\begin{subequations}\label{eq:continuous.problem}
  \begin{alignat}{2}
    \partial_t \boldsymbol{u}
    - \boldsymbol{\nabla} \cdot \boldsymbol{\sigma}(\boldsymbol{\nabla u})
    +(\boldsymbol{u} \cdot \boldsymbol{\nabla}) \boldsymbol{u}
    + \boldsymbol{\nabla} p
    &= \boldsymbol{f}   &\qquad &\text{in $(0,\tF] \times \Omega$}\, \label{eq:continuous.problem.momentum},
      \\
      \boldsymbol{\nabla} \cdot \boldsymbol{u}
      &= 0 &\qquad &\text{in $(0,\tF] \times \Omega$}\, \label{eq:continuous.problem.mass},
        \\
        \boldsymbol{u}
        &= \boldsymbol{0}          &\qquad &\text{on $(0,\tF] \times \partial\Omega$}\, \label{eq:continuous.problem.boundary.condition},
          \\
          \boldsymbol{u}(0,\cdot) &= \boldsymbol{u}^0 &\qquad &\text{in $\Omega$}\, \label{eq:continuous.problem.initial.condition},
  \end{alignat}
\end{subequations}
where $\boldsymbol{f}:(0,\tF\rbrack \times \Omega \to \Real^d$ represents an external body force and $\boldsymbol{u}^0 : \Omega \to \Real^d$ a given divergence-free initial velocity with vanishing trace on $\partial \Omega$, while $\boldsymbol{\sigma}$ represents the diffusive flux function detailed below.

Given two real numbers $\nu >0$ and $r\in (1,\infty)$, let
\begin{equation}\label{eq:sobolev:indexes}
  {r}' \coloneqq \frac{r}{r-1} \in (1,\infty), \quad  \overline{r} \coloneqq \max\{r,2\} \in [2,\infty), \quad \underline{r} \coloneqq \min \{ r, 2\} \in (1,2].
\end{equation}
For $\mathbb{T} \in \{ \Real^d, \Real^{d\times d} \}$, the diffusive flux function is defined as follows:
\begin{equation}\label{def:sigma}
  \boldsymbol{\sigma} : \mathbb{T} \ni \boldsymbol{A} \mapsto \nu \, |\boldsymbol{A}|^{r-2} \, \boldsymbol{A} \in \mathbb{T}.
\end{equation}
This function satisfies the following properties (see \cite[Corollary A.3]{Botti.Castanon-Quiroz.ea:21}): For all $\boldsymbol{A},\boldsymbol{B} \in \mathbb{T}$,
\begin{gather}
  |\boldsymbol{\sigma}(\boldsymbol{A}) - \boldsymbol{\sigma}(\boldsymbol{B})|
  \lesssim
  \nu \left(
  |\boldsymbol{A}|^r + |\boldsymbol{B}|^r
  \right)^\frac{r - \underline{r}}{r}
  |\boldsymbol{A} - \boldsymbol{B}|^{\underline{r}-1} \label{eq:sigma:holder:cont},
  \\
  \left(
  \boldsymbol{\sigma}(\boldsymbol{A}) - \boldsymbol{\sigma}(\boldsymbol{B})
  \right) \star \left(\boldsymbol{A} - \boldsymbol{B}\right)
  \gtrsim \nu \left( |\boldsymbol{A}|^r + |\boldsymbol{B}|^r \right)^\frac{\underline{r}-2}{r} |\boldsymbol{A} - \boldsymbol{B}|^{\overline{r}}, \label{eq:sigma:monotonicity}
\end{gather}
where $\star$ denotes the appropriate tensor contraction (dot product if $\boldsymbol{A} ,\, \boldsymbol{B} \in \Real^d$ or the Frobenius product if  $\boldsymbol{A} ,\, \boldsymbol{B} \in \Real^{d\times d}$).

\begin{remark}[Choice of model problem]
  Equation~\eqref{eq:continuous.problem} has to be intended as a simplified problem that retains all the difficulties related to regime-dependent behaviours of numerical approximations.
  A more physically accurate version of the diffusion term would be obtained replacing $\boldsymbol{\sigma}( \boldsymbol{\nabla} \boldsymbol{u} )$ with $\boldsymbol{\sigma}( \boldsymbol{\varepsilon} (\boldsymbol{u}) )$, where $\boldsymbol{\varepsilon}(\boldsymbol{u}) \coloneqq \frac12\left(\boldsymbol{\nabla} \boldsymbol{u} + \boldsymbol{\nabla} \boldsymbol{u}^\top\right)$ denotes the rate-of-strain tensor.
  This would only require minor changes in the analysis since, also due to the presence of DG-like jumps in the formulation, Korn-type inequalities hold for the presented scheme; see, e.g., \cite[Lemma~7.23]{Di-Pietro.Droniou:20} for the case $r = 2$ (the extension to $r \neq 2$ is possible in the spirit of \cite[Section~6.2.1]{Di-Pietro.Droniou:20}).
\end{remark}

\subsection{Mesh and inequalities up to a constant}
Let $\{\Th\}_h$ be a family of conforming simplicial meshes of $\Omega$ in the sense of \cite[Chapter 2]{Ciarlet:02}.
For any $T\in\Th$, we denote by $\omega_T$ the union of the mesh elements sharing at least one face with $T$, which are collected in the set $\elements{T}$.
We moreover denote by $\Fh$ the set of mesh faces and by $\Fhi \subset \Fh$ the subset of interfaces contained in $\Omega$.
Given a face $F\in\Fh$, we denote by $\TF$ the set of mesh elements sharing $F$ and by $\omega_F$ their union.
For any mesh element or face $Y \in \Th \cup \Fh$, we denote by $h_Y$ its diameter and set $h \coloneq \max_{T \in \Th} h_T$.

To avoid naming generic constants, from this point on we will use the notation $a\lesssim b$ (resp., $a \gtrsim b$) to express the inequality $a \le C b$ (resp., $a \ge C b$), with $C$ independent of the mesh size and time step, the viscosity parameter $\nu$, and the solution, but possibly depending on other quantities including the domain, the ambient dimension $d$, the mesh regularity parameter, and the Sobolev index $r$.
We will write $a \simeq b$ in lieu of ``$a \lesssim b$ and $b \lesssim a$''.
Furthermore, we will use the symbol $c(\delta)$ for a
generic constant, possibly different at each occurrence, which depends on the parameter $\delta$ but is independent of the mesh size, the problem data and
solution.

\subsection{Function spaces}

Let $Y \in \Th \cup \Fh \cup \{\Omega\}$. For an integer $m \geq 0 $ and a real number $1 \leq q \leq \infty$, we denote by $\scalarSobolev{m}{q}{Y}$ the usual Sobolev space on $Y$ equipped with the standard norm $\norm{\scalarSobolev{m}{q}{Y}}{\cdot}$ and seminorm $\seminorm{\scalarSobolev{m}{q}{Y}}{\cdot}$. For $m=0$, we obtain the Lebesgue space $L^q(Y) \coloneqq \scalarSobolev{0}{q}{Y}$ while, for $q=2$, we have the Hilbert space $H^m(Y) \coloneqq \scalarSobolev{m}{2}{Y}$.
We define the following broken Sobolev space on $\Th$:
\[
\scalarSobolev{m}{q}{\Th} \coloneqq
\left\{v\in L^q(\Omega) \st \restr{v}{T} \in \scalarSobolev{m}{q}{T} \text{ for all } T \in \Th \right\}.
\]
We will also need broken spaces defined on local sets $\mathcal{T}_Y$, $Y \in \Th \cup \Fh$, defined in a similar way. We furthermore introduce the broken gradient $\boldsymbol{\nabla}_h : \scalarSobolev{1}{1}{\Th} \to L^1(\Th)^d$ such that $\restr{(\boldsymbol{\nabla}_h \varphi)}{T} \coloneqq \boldsymbol{\nabla}\restr{\varphi}{T}$ for all $T \in \Th$.

Given a polynomial degree $m \geq 0$, we define the broken polynomial space
\[
\scalarPoly{m}(\Th) \coloneqq \cbracket{v_h \in L^1(\Omega) \st \restr{v_h}{T} \in \scalarPoly{m}(T) \text{ for all } T \in \Th},
\]
where $\scalarPoly{m}(Y)$, $Y \in \Th \cup \Fh$, denotes the space of polynomials of total degree up to $m$ on $Y$.
We also introduce the convention $\scalarPoly{-1}(Y) \coloneq \{ 0\}$ and, for $m \ge -1$, the local $L^2$-orthogonal projector on $\mathcal{P}^m(Y)$ is denoted by $\pi_{0,Y}^m$.
The global $L^2$-orthogonal projector on $\scalarPoly{m}(\Th)$ is denoted by $\pi_{0,h}^m$, and is defined setting, for all $\varphi \in L^1(\Omega)$,
\[
\restr{(\pi_{0,h}^m \varphi )}{T}
\coloneqq \pi_{0,T}^m \restr{\varphi}{T}
\qquad \forall T \in \Th.
\]

Spaces of vector- or tensor-valued functions, as well as their elements, are denoted in boldface font.
For example, $ \boldsymbol{\varphi} \in \vectSobolev{m}{q}{Y}$ denotes a vector- or tensor-valued function, with the distinction between vector or tensor being clear from the context.
Finally, in the following we denote by $\Hdiv{\Omega}$ the space of functions in $\boldsymbol{L}^2(\Omega)$ with divergence in $L^2(\Omega)$.

\subsection{Raviart--Thomas interpolation}

Given an integer $k \ge 0$, for all $T \in \Th$, we denote by $\RT^k(T) \coloneqq \vectPoly{k}(T) + \boldsymbol{x} \scalarPoly{k}(T)$ the usual Raviart--Thomas space \cite{Boffi.Brezzi.ea:13} of order $k$.
The corresponding interpolator $\projRTEle{k}: \boldsymbol{H}^1(T) \to \RT^k(T)$ is uniquely defined by the following conditions:
For all $\boldsymbol{v} \in \boldsymbol{H}^1(T)$,
\begin{equation}\label{eq:projRTEle}
  \projLEle{k-1} \, \projRTEle{k} \boldsymbol{v} = \projLEle{k-1} \boldsymbol{v}
  \text{ and }
  (\projRTEle{k} \boldsymbol{v})\cdot \normal_F = \pi_{0,F}^k (\boldsymbol{v} \cdot \normal_F)
  \quad \forall F \in \FT.
\end{equation}
We have the following approximation properties (see, for instance, \cite[Theorem 16.4]{Ern.Guermond:21}):
For all $q \in [1,\infty]$ and all integers $s = 1, \ldots, k+1$ and $j=0\ldots,s$, it holds, for any $T\in \Th$,
\begin{equation}\label{eq:approx:RT:ele}
  \seminorm{\vectSobolev{j}{q}{T}}{\boldsymbol{v} - \projRTEle{k} \boldsymbol{v}} \lesssim h_T^{s-j} \seminorm{\vectSobolev{s}{q}{T}}{\boldsymbol{v}} \quad \forall \boldsymbol{v} \in \vectSobolev{s}{q}{T},
\end{equation}
and, for any $F\in \FT$,
\begin{equation}\label{eq:approx:RT:face}
  \norm{\boldsymbol{L}^q(F)}{\boldsymbol{v} - \projRTEle{k} \boldsymbol{v}}
  \lesssim h_T^{s-\frac{1}{q}} \seminorm{\vectSobolev{s}{q}{T}}{\boldsymbol{v}} \quad \forall \boldsymbol{v} \in \vectSobolev{s}{q}{T}.
\end{equation}
The global interpolator on $\RT^k(\Th) \coloneq \left\{ \boldsymbol{v}_h \in \Hdiv{\Omega} \;:\; \text{$\restr{\boldsymbol{v}_h}{T} \in \RT^k(T)$ for all $T \in \Th$}\right\}$, obtained enforcing the conditions \eqref{eq:projRTEle} for all $T \in \Th$ and all $F \in \Fh$, is denoted by $\projRT{k}$.

\subsection{Jumps, liftings and discrete gradient}

For each interface $F\in\Fhi$, we fix once and for all an orientation for the unit normal vector $\normal_F$.
Denoting by $T_1$ and $T_2$ the elements sharing $F$ ordered so that $\normal_F$ points out of $T_1$, we define the jump and average operators such that, for any vector-valued function $\boldsymbol{v} \in \vectSobolev{1}{1}{\Th}$,
\[
\jump{\boldsymbol{v}} \coloneqq \restr{\boldsymbol{v}}{T_1} - \restr{\boldsymbol{v}}{T_2},\qquad
\avg{\boldsymbol{v}} \coloneqq \frac12 \rbracket{\restr{\boldsymbol{v}}{T_1} + \restr{\boldsymbol{v}}{T_2}}.
\]
The above operators are extended to boundary faces $F\in \Fh \setminus \Fhi$ setting
$\jump{\boldsymbol{v}} \coloneqq \boldsymbol{v}$ and
$\avg{\boldsymbol{v}} \coloneqq \boldsymbol{v}$.
Given an integer $k\ge 0$, for each $F\in\Fh$ we define the tensor-valued local trace lifting
$\rF:\boldsymbol{L}^1(F) \to \vectPoly{k}(\Th)$
such that, for all vector-valued $\boldsymbol{w} \in \boldsymbol{L}^1(F)$,
\begin{equation}\label{eq:rF}
  \int_\Omega \rF ( \boldsymbol{w}  ) : \boldsymbol{\tau}_h
  = \int_F \boldsymbol{w} \cdot \rbracket{\avg{\boldsymbol{\tau}_h}\normal_F}
  \qquad \forall \boldsymbol{\tau}_h\in \vectPoly{k}(\Th)
\end{equation}
and we let $\Rh: \vectSobolev{1}{1}{\Th} \to \vectPoly{k}(\Th)$ be the global face jumps lifting such that, for any $\boldsymbol{v} \in \vectSobolev{1}{1}{\Th}$,
\begin{equation}\label{eq:Rh}
  \Rh \boldsymbol{v} \coloneqq \sum_{F\in\Fh}\rF(\jump{\boldsymbol{v}}).
\end{equation}
Finally, we define the discrete gradient $\Gh: \vectSobolev{1}{1}{\Th}\to \boldsymbol{L}^1(\Omega)^d$ setting
\begin{equation}\label{eq:Gh}
  \Gh \boldsymbol{v} \coloneqq \boldsymbol{\nabla}_h \boldsymbol{v} - \Rh\boldsymbol{v} .
\end{equation}

The following result is a straightforward consequence of \cite[Lemma 3.2 and Remark 3.2]{Beirao-da-Veiga.Di-Pietro.ea:24}, together with the approximation property of $\projRT{k}$ discussed in the previous section.
\begin{lemma}[Boundedness and approximation properties of the discrete gradient]\label{lem:approx:Gh}
  Let $q\in [1,\infty]$, $m\in \{ 0,\ldots,k \}$, and $\boldsymbol{w} \in \boldsymbol{W}^{1,q}_0(\Omega) \cap \boldsymbol{W}^{m+1,q}(\Th)$.
  Then, the following holds for any $T\in \Th$:
  \begin{alignat}{2}
    \label{eq:Gh:stab}
    \norm{\boldsymbol{L}^q(T)}{\Gh \projRT{k} \boldsymbol{w}} &\lesssim \seminorm{\boldsymbol{W}^{1,q}(\TT)}{\boldsymbol{w}},
    \\ \label{eq:Gh:approx}
    \norm{\boldsymbol{L}^q(T)}{\Gh \projRT{k} \boldsymbol{w} - \boldsymbol{\nabla w}} &\lesssim h_T^{m} \seminorm{\boldsymbol{W}^{m+1,q}(\TT)}{\boldsymbol{w}}.
  \end{alignat}
\end{lemma}

\section[An H(div)-conforming scheme]{An $\boldsymbol{H}(\operatorname{div})$-conforming scheme}\label{sec: Discrete problem}

From this point on, we let a polynomial degree $k\ge 1$ be fixed.

\subsection{Discrete problem}

We denote by $\BDM^k(\Th)$ the Brezzi--Douglas--Marini space of degree $k$ over $\Th$ defined as (see, e.g., \cite{Boffi.Brezzi.ea:13})
\[
\BDM^k(\Th) \coloneq \vectPoly{k}(\Th) \cap \Hdiv{\Omega} \, ,
\]
and set
\[
\Vh{k}
\coloneqq \cbracket{ \boldsymbol{v}_h \in \BDM^k(\Th) \st \restr{( \boldsymbol{v}_h \cdot \normal)}{\partial \Omega } = 0}.
\]
Furthermore, we define
\[
\Qh{k-1} \coloneqq \left\{
q_h \in \scalarPoly{k-1}(\Th) \st \int_\Omega q_h = 0
\right\}.
\]

In order to state the fully discretized version of problem \eqref{eq:continuous.problem}, we introduce a sequence of time steps $t^n \coloneqq n \, \tstep,$ with  $n = 0, 1, 2, \ldots, N$, where $\tstep \coloneqq \frac{\tF}{N}$ is the time step length.
For the sake of brevity, at a fixed time $t^n$, for any smooth enough time-dependent function $\varphi$ we introduce the abridged notation $\varphi^n \coloneq \varphi(t^n)$, with the understanding that $\varphi^n$ is a function of space only if $\varphi$ depends on both space and time.
Assuming $\boldsymbol{f} \in C^0((0,\tF]; \boldsymbol{L}^2(\Omega))$ and defining the discrete initial condition as $\boldsymbol{u}_h^0 \coloneqq \projRT{k} \boldsymbol{u}^0$, we consider the following approximation of \eqref{eq:continuous.problem}:
\\
For $n=1,\ldots,N$, find $\rbracket{\boldsymbol{u}_h^n,p_h^n} \in \Vh{k} \times \Qh{k-1}$ such that, for all $\rbracket{\boldsymbol{v}_h, q_h} \in \Vh{k} \times \Qh{k-1}$,
\begin{subequations}\label{eq:full.discrete.problem}
  \begin{align}\label{eq:full.discrete.problem:momentum}
    \int_\Omega \frac{\boldsymbol{u}_h^n - \boldsymbol{u}_h^{n-1}}{\tstep} \cdot \boldsymbol{v}_h
    + a_h(\boldsymbol{u}_h^n,\boldsymbol{v}_h)
    + b_h(\boldsymbol{v}_h,p_h^n)
    + c_h(\boldsymbol{u}_h^{n},\boldsymbol{u}_h^{n},\boldsymbol{v}_h)
    + j_h(\boldsymbol{u}_h^{n};\boldsymbol{u}_h^{n},\boldsymbol{v}_h)
    &= \int_\Omega \boldsymbol{f}^n \cdot \boldsymbol{v}_h \, ,
    \\ \label{eq:full.discrete.problem:mass}
    b_h(\boldsymbol{u}_h^n,q_h) &= 0 \, ,
  \end{align}
\end{subequations}
where
\begin{align}\label{eq:ah}
  a_h(\boldsymbol{w}_h,\boldsymbol{v}_h) &\coloneqq
  \int_\Omega \boldsymbol{\sigma}(\Gh \boldsymbol{w}_h) : \Gh \boldsymbol{v}_h
  + \sum_{F \in \mathcal{F}_h} h_F^{1-r} \int_F \boldsymbol{\sigma} (\jump{\boldsymbol{w}_h}) \cdot \jump{\boldsymbol{v}_h}, \\ \nonumber 
  b_h(\boldsymbol{v}_h,q_h) &\coloneqq
  -\int_\Omega q_h \, (\boldsymbol{\nabla} \cdot \boldsymbol{v}_h), \\ \label{eq:ch}
  c_h(\boldsymbol{z}_h,\boldsymbol{w}_h,\boldsymbol{v}_h) &\coloneqq
  \int_\Omega (\boldsymbol{z}_h \cdot \boldsymbol{\nabla}_h)\boldsymbol{w}_h \cdot \boldsymbol{v}_h
  -\sum_{F \in \mathcal{F}_h^i} \int_F (\boldsymbol{z}_h \cdot \normalF)\jump{\boldsymbol{w}_h}\cdot\avg{\boldsymbol{v}_h}, \\ \label{eq:jh}
  j_h(\boldsymbol{z}_h;\boldsymbol{w}_h,\boldsymbol{v}_h) & \coloneqq \sum_{F \in \mathcal{F}_h^i} \gamma_F(\boldsymbol{z}_h)\int_F  \jump{\boldsymbol{w}_h}\cdot \jump{\boldsymbol{v}_h}.
\end{align}
In \eqref{eq:jh}, the stabilization parameter is defined as
\begin{equation}\label{eq:gammaF}
  \gamma_F(\boldsymbol{z}_h) \coloneqq \max \left\{ \norm{L^\infty(F)}{\boldsymbol{z}_h \cdot \normalF}, C_F \right\},
\end{equation}
where $C_F\in {\mathbb R}$ is a fixed positive ``safeguard'' parameter. 
The presence of $C_F$ in \eqref{eq:gammaF} slightly differentiates $j_h$ from standard upwinding in the NS literature. Such parameter, which in practice can be taken very small, is introduced to guarantee some convection control also for vanishing {\it discrete} velocities. At the theoretical level, a justification is given by \eqref{eq:gamma-u-bound}.
Although the above discrete forms are defined on the discrete spaces $\Vh{k}$ and $\Qh{k}$, we intend them as extended to any sufficiently regular function whenever needed.
\begin{remark}[Time stepping scheme]
  Other choices of time discretizations, such as the semi-implicit Euler scheme or a higher order time stepping (i.e., the midpoint rule) could be adopted.
  In order to fix the ideas, here we focus on the implicit Euler scheme.
\end{remark}

\subsection{Summary of the theoretical results}

The rest of this section and the following one are devoted to the stability and convergence analysis of problem~\eqref{eq:full.discrete.problem}.
In particular, in Lemma~\ref{eq:stability} below we prove the existence of a solution at every time step as well as an a priori bound on the discrete solution. A Reynolds-semi-robust and pressure-robust error estimate for the velocity is proved in Theorem~\ref{thm:error.estimate} (for $ r \geq 2$) and Theorem~\ref{thm:error.estimate.rs2} (for $1<r<2$) of Section \ref{sec:Velocity error analysis}, see also Table \ref{tab:convergence.rates}. This error estimate accounts for the diffusion- or convection-dominated nature of the flow inside each element and on each face.
Assuming that the velocity and the rate-of-stress tensor have sufficient regularity to fully exploit the polynomial degree, the natural error measure (including the discrete $\boldsymbol{C}^0(\boldsymbol{L}^2)$-norm squared plus the $\boldsymbol{L}^{\overline{r}}(\boldsymbol{W}^{1,r})$-norm to the $\overline{r}$-th power) converges as follows:
\begin{itemize}
\item For diffusion-dominated flows:
  $$
  \begin{aligned}
    & \textrm{If } r < 2 \ : \quad \tstep^2 + h^{rk}, \\
     & \textrm{If } r = 2 \ : \quad \tstep^2 + h^{2k},  \\
    & \textrm{If } r > 2 \ : \quad
    \tstep^2 + \begin{cases}
      h^{r'(k+1)} & \text{if $k \ge \frac{r'}{2-r'}$},
      \\
      h^{2k} & \text{otherwise} \, ;
    \end{cases}
  \end{aligned}
  $$
\item For convection-dominated flows, as $\tstep^2 + h^{2k+1}$ irrespectively of the value of $r$.
\end{itemize}

\subsection{Divergence-free discrete velocity space}\label{sec:add-not}

We introduce the divergence-free subspace of $\Vh{k}$:
\[
\Zh{k} \coloneqq \cbracket{\boldsymbol{v}_h \in \Vh{k} \st b_h(\boldsymbol{v}_h, q_h)=0 \quad \forall q_h \in \Qh{k-1}}.
\]%
\begin{remark}[Interpolates of divergence-free functions with zero trace]\label{rem:int.u.Zh}
  It can be checked that, for all $\boldsymbol{v} \in \boldsymbol{H}^1_0(\Omega)$
  such that $\boldsymbol{\nabla} \cdot \boldsymbol{v} = 0$, $\projRT{k}\boldsymbol{v} \in \Zh{k}$.
\end{remark}
Observing that $\boldsymbol{\nabla} \cdot \Vh{k} = \Qh{k-1}$, it follows that condition $b_h(\boldsymbol{v}_h, q_h)=0$ for all $q_h \in \Qh{k-1}$ is equivalent to $\boldsymbol{\nabla} \cdot \boldsymbol{v}_h = 0$, so that $\Zh{k} = \left\{\boldsymbol{v}_h \in \Vh{k} \st \boldsymbol{\nabla} \cdot \boldsymbol{v}_h =0 \right\}$.
Notice that the discrete velocity solution of \eqref{eq:full.discrete.problem} belongs to $\Zh{k}$, therefore $\boldsymbol{u}_h^n$ is pointwise divergence-free.
The following property follows from \cite[Lemma 6.39]{Di-Pietro.Ern:12} noticing that the first two terms of \cite[Eq. (6.57)]{Di-Pietro.Ern:12} coincide with the expression \eqref{eq:ch} of $c_h$ and the last two are zero since $\boldsymbol{w}_h \in \Zh{k}$:
\begin{equation}\label{eq:ch:non-dissipativity}
  c_h (\boldsymbol{z}_h,\boldsymbol{v}_h, \boldsymbol{v}_h) = 0 \qquad \forall (\boldsymbol{z}_h, \boldsymbol{v}_h) \in \Zh{k} \times \Vh{k} .
\end{equation}
The following anti-symmetry property follows from \eqref{eq:ch:non-dissipativity} and linearity:
\begin{equation}\label{eq:ch:anti-symmetry}
  c_h(\boldsymbol{z}_h, \boldsymbol{w}_h, \boldsymbol{v}_h)
  = -c_h(\boldsymbol{z}_h, \boldsymbol{v}_h, \boldsymbol{w}_h)
  \qquad
  \forall (\boldsymbol{z}_h, \boldsymbol{w}_h, \boldsymbol{v}_h) \in \Zh{k} \times \Vh{k} \times \Vh{k}.
\end{equation}

\subsection{Discrete norms and seminorms}

We define the following broken norm for all $\boldsymbol{v} \in \vectSobolev{1}{q}{\Th}$ with $q \in [1,\infty]$:
\[
\norm{1,q,h}{\boldsymbol{v}}
\coloneqq
\begin{cases}
  \left(
  \norm{\boldsymbol{L}^q(\Omega)}{\boldsymbol{\nabla}_h \boldsymbol{v}}^q
  + \sum_{F \in \Fh} h_F^{1-q} \norm{\boldsymbol{L}^q(F)}{\jump{\boldsymbol{v}}}^q
  \right)^{\frac1q}
  &\text{if $q < \infty$},
  \\
  \norm{\boldsymbol{L}^\infty(\Omega)}{\boldsymbol{\nabla}_h \boldsymbol{v}}
  + \max_{F \in \Fh} h_F^{-1} \norm{\boldsymbol{L}^\infty(F)}{\jump{\boldsymbol{v}}}
  &\text{if $q = \infty$.}
\end{cases}
\]
For a fixed $\boldsymbol{w}_h \in \Vh{k}$, we additionally define the following semi-norm induced by $j_h$:
\begin{equation}\label{eq:convective.seminorm}
  \seminorm{\boldsymbol{w}_h}{\boldsymbol{v}}^2
  \coloneqq
  j_h(\boldsymbol{w}_h;\boldsymbol{v},\boldsymbol{v})
  = \sum_{F \in \Fhi} \gamma_F(\boldsymbol{w}_h) \norm{\boldsymbol{L}^2(F)}{\jump{\boldsymbol{v}}}^2.
\end{equation}

\begin{proposition}[Norm of the global jump lifting]
  For all $q \in [1,\infty)$, it holds
    \begin{equation}\label{eq:norm.Rh}
      \norm{\boldsymbol{L}^q(\Omega)}{\boldsymbol{R}_h^k \boldsymbol{v}_h}^q
      \simeq
      \sum_{F \in \Fh} h_F^{1-q} \norm{\boldsymbol{L}^q(F)}{\jump{\boldsymbol{v}_h}}^q
      \qquad \forall \boldsymbol{v}_h \in \Vh{k}.
    \end{equation}
\end{proposition}

\begin{proof}
  The fact that $\norm{\boldsymbol{L}^q(\Omega)}{\boldsymbol{R}_h^k \boldsymbol{v}_h}^q \lesssim \sum_{F \in \Fh} h_F^{1-q} \norm{\boldsymbol{L}^q(F)}{\jump{\boldsymbol{v}_h}}^q$ can be proved using similar arguments as for \cite[Eq.~(4.42)]{Di-Pietro.Ern:12} together with the bound \cite[Eq.~(3.9)]{Beirao-da-Veiga.Di-Pietro.ea:24}.
  
  Let us prove briefly the converse inequality.
  In what follows, given a space $\boldsymbol{W}$ of vector-valued functions, we denote by $\boldsymbol{W} \otimes \Real^d$ the space of tensor-valued functions with rows in $\boldsymbol{W}$.
  Starting with $q = 2$, for a fixed $\boldsymbol{v}_h \in \Vh{k}$, we let $\boldsymbol{\tau}_h \in \BDM^k(\Th) \otimes \Real^d$ be such that $\boldsymbol{\tau}_h \normal_F = h_F^{-1} \jump{\boldsymbol{v}_h}$ for all $F \in \Fh$ and, if $k \ge 2$, $\int_T \boldsymbol{\tau}_h : \boldsymbol{\upsilon} = 0$ for all $\boldsymbol{\upsilon} \in \NE^{k-2}(T) \otimes \Real^d$, where $\NE^{k-2}(T)$ denotes the N\'ed\'elec space of degree $k-2$ on $T$.
  By unisolvence of the degrees of freedom together with a scaling argument and $h_T \simeq h_F$ for all $F \in \FT$, we have that $\norm{\boldsymbol{L}^2(T)}{\boldsymbol{w}}^2 \simeq \sum_{F \in \FT} h_F \norm{L^2(F)}{\boldsymbol{w} \cdot \normal_F}^2$ for all $T \in \Th$ and all $\boldsymbol{w} \in \BDM^k(T)$ such that $\int_T \boldsymbol{w} \cdot \boldsymbol{z} = 0$ for all $\boldsymbol{z} \in \NE^{k-2}(T)$ if $k \ge 2$.
  Applying this results to $(\boldsymbol{\tau}_h)|_T$, summing over $T \in \Th$, and using the fact that each face appears at most twice in the right-hand side, we conclude that
  \begin{equation}\label{eq:norm.tau.h}
    \norm{\boldsymbol{L}^2(\Omega)}{\boldsymbol{\tau}_h}^2
    \simeq \sum_{F \in \Fh} h_F^{-1} \norm{\boldsymbol{L}^2(F)}{\jump{\boldsymbol{v}_h}}^2.
  \end{equation}
  Thus, we can write
  \[
  \begin{aligned}
    \sum_{F \in \Fh} h_F^{-1} \norm{\boldsymbol{L}^2(F)}{\jump{\boldsymbol{v}_h}}^2
    \overset{\eqref{eq:rF},\,\eqref{eq:Rh}}&= \int_\Omega \Rh \boldsymbol{v}_h : \boldsymbol{\tau}_h
    \le \norm{\boldsymbol{L}^2(\Omega)}{\Rh \boldsymbol{v}_h}
    \norm{\boldsymbol{L}^2(\Omega)}{\boldsymbol{\tau}_h}
    \\
    &\overset{\eqref{eq:norm.tau.h}}\lesssim \norm{\boldsymbol{L}^2(\Omega)}{\Rh \boldsymbol{v}_h}
    \left(
    \sum_{F \in \Fh} h_F^{-1} \norm{\boldsymbol{L}^2(F)}{\jump{\boldsymbol{v}_h}}^2
    \right)^{\frac12},
  \end{aligned}
  \]
  which, after simplification, gives \eqref{eq:norm.Rh} for $q = 2$.
  The case $q \neq 2$ can be deduced from the case $q = 2$ using direct and inverse Lebesgue embeddings (see \cite[Lemma~1.25]{Di-Pietro.Droniou:20}) along with mesh regularity.
\end{proof}
\subsection{Properties of the discrete diffusive form}

\begin{lemma}[Norm equivalence]\label{lem:ah:norm:equivalence}
  For all $\boldsymbol{v}_h \in \Vh{k}$,
  \begin{equation}\label{eq:ah:norm:equivalence}
    a_h (\boldsymbol{v}_h,\boldsymbol{v}_h) \simeq \nu \norm{1,r,h}{\boldsymbol{v}_h}^r.
  \end{equation}
\end{lemma}
\begin{proof}
  From the definition \eqref{eq:ah} of $a_h$ it follows that 
  $\nu^{-1} a_h (\boldsymbol{v}_h,\boldsymbol{v}_h) = \norm{\boldsymbol{L}^r(\Omega)}{\Gh \boldsymbol{v}_h}^r + \sum_{F \in \Fh} h_F^{1-r} \norm{\boldsymbol{L}^r(F)}{\jump{\boldsymbol{v}_h}}^r$.
  Using \eqref{eq:Gh} followed by the fact that $(\alpha + \beta)^r \lesssim \alpha^r + \beta^r$ for all positive real numbers $\alpha$ and $\beta$, we conclude that $\nu^{-1} a_h (\boldsymbol{v}_h,\boldsymbol{v}_h) \lesssim \norm{\boldsymbol{L}^r(\Omega)}{\boldsymbol{\nabla}_h \boldsymbol{v}_h}^r + \norm{\boldsymbol{L}^r(\Omega)}{\Rh \boldsymbol{v}_h}^r + \sum_{F \in \Fh} h_F^{1-r} \norm{\boldsymbol{L}^r(F)}{\jump{\boldsymbol{v}_h}}^r \lesssim \norm{1,r,h}{\boldsymbol{v}_h}^r$, where the last inequality follows from \eqref{eq:norm.Rh} with $q=r$. 
  To prove the converse inequality, we start from $\norm{1,r,h}{\boldsymbol{v}_h}^r = \norm{\boldsymbol{L}^r(\Omega)}{\boldsymbol{\nabla}_h \boldsymbol{v}_h}^r + \sum_{F \in \Fh} h_F^{1-r} \norm{\boldsymbol{L}^r(F)}{\jump{\boldsymbol{v}_h}}^r$ and use 
  $\norm{\boldsymbol{L}^r(\Omega)}{\boldsymbol{\nabla}_h \boldsymbol{v}_h}^r
  \overset{\eqref{eq:Gh}}\le
  \norm{\boldsymbol{L}^r(\Omega)}{\Gh \boldsymbol{v}_h}^r
  + \norm{\boldsymbol{L}^r(\Omega)}{\Rh \boldsymbol{v}_h}^r$
  and \eqref{eq:norm.Rh} with $q = r$ to bound the last term in the right-hand side.
\end{proof}
By appropriately combining \eqref{eq:sigma:holder:cont}, \eqref{eq:sigma:monotonicity}, and \eqref{eq:ah:norm:equivalence}, one can deduce the following result:
\begin{lemma}[Boundedness and monotonicity of $a_h$]
  For all $\boldsymbol{w}_h,\boldsymbol{v}_h,\boldsymbol{z}_h \in \Vh{k}$, defining $\boldsymbol{e}_h \coloneqq \boldsymbol{w}_h - \boldsymbol{z}_h$ and recalling the definition \eqref{eq:sobolev:indexes} of $\overline{r}$ and $\underline{r}$, the following holds:
  \begin{alignat}{2}
    a_h (\boldsymbol{w}_h,\boldsymbol{v}_h) - a_h(\boldsymbol{z}_h,\boldsymbol{v}_h)
    &\lesssim
    \nu
    \left(
    \norm{1,r,h}{\boldsymbol{w}_h}^r
    + \norm{1,r,h}{\boldsymbol{z}_h}^r
    \right)^{\frac{r - \underline{r}}{r}}
    \norm{1,r,h}{\boldsymbol{e}_h}^{\underline{r}-1}
    \norm{1,r,h}{\boldsymbol{v}_h}, \label{eq:ah:continuity}\\
    a_h (\boldsymbol{w}_h,\boldsymbol{e}_h) - a_h(\boldsymbol{z}_h,\boldsymbol{e}_h)
    &\gtrsim
    \nu
    \left(
    \norm{1,r,h}{\boldsymbol{w}_h}^r
    + \norm{1,r,h}{\boldsymbol{z}_h}^r
    \right)^{\frac{\underline{r}-2}{r}}
    \norm{1,r,h}{\boldsymbol{e}_h}^{\overline{r}}. \label{eq:ah:coercivity}
  \end{alignat}
\end{lemma}

\begin{remark}[Quasi-norms]
  The monotonicity and continuity results here above are directly written in terms of classical (DG-type) Sobolev norms since, also for the sake of readability, the error estimates in the present contribution are developed in such norms. An alternative approach would be to bridge more explicitly with suitable DG-type quasi-norms (in the spirit, e.g., of \cite{Barrett.Liu:94} for the Stokes problem with conforming elements) with the intent of obtaining quasi-optimal error bounds in those quasi-norms as an intermediate result.
  For the specific case of lowest-order (piecewise linear) DG-type approximation, estimates of order one, matching the convergence rate of the approximation error, have been obtained in~\cite{Diening.Koner.ea:14,Malkmus.Ruzicka.ea:17}
\end{remark}

\subsection{Stability}

\begin{lemma}[Existence of a discrete solution and a priori estimate]\label{eq:stability}
  Let $\boldsymbol{f} \in \boldsymbol{C}^{0}\left((0,\tF];\boldsymbol{L}^2(\Omega)\right)$.
  Then, problem \eqref{eq:full.discrete.problem} admits at least one solution.
  Furthermore, the discrete velocity satisfies the following stability estimate for all $n \geq 1$:
  \[
  \norm{\boldsymbol{L}^2(\Omega)}{\boldsymbol{u}_h^n}^2
  + \sum_{\ell=1}^n \tstep \left(
  \nu \norm{1,r,h}{\boldsymbol{u}_h^\ell}^r
  + \seminorm{\boldsymbol{u}_h^\ell}{\boldsymbol{u}_h^\ell}^2
  \right)
  \lesssim
  \rbracket{\sum_{\ell=1}^n\tstep \norm{\boldsymbol{L}^2(\Omega)}{\boldsymbol{f}^\ell}}^2
  + \norm{\boldsymbol{L}^2(\Omega)}{\boldsymbol{u}_h^0}^2 \, ,
  \]
  with hidden constant depending only on $k$ and the mesh regularity parameter.
\end{lemma}

\begin{proof}
  We first show that problem \eqref{eq:full.discrete.problem} admits at least one solution, and then the stability estimate.
  \medskip\\
  \underline{\textbf{Existence of a solution.}}
  We consider the finite dimensional Hilbert space $\Zh{k}$, endowed with an arbitrary scalar product $(\cdot, \cdot)_{\Zh{k}}$ and the induced norm $\norm{\Zh{k}}{\cdot}$. We define, for fixed $\boldsymbol{z}_h \in \Zh{k}$ and $\boldsymbol{\phi} \in \boldsymbol{L}^2(\Omega)$,
  the function $\Psi_{\boldsymbol{z}_h,\boldsymbol{\phi}} : \Zh{k} \to \Zh{k}$ such that
  \[
  (\Psi_{\boldsymbol{z}_h,\boldsymbol{\phi}} \boldsymbol{w}_h , \boldsymbol{v}_h)_{\Zh{k}} \coloneqq
  \int_\Omega \frac{\boldsymbol{w}_h - \boldsymbol{z}_h}{\tstep} \cdot \boldsymbol{v}_h
  + a_h(\boldsymbol{w}_h,\boldsymbol{v}_h)
  + c_h(\boldsymbol{w}_h,\boldsymbol{w}_h,\boldsymbol{v}_h)
  + j_h(\boldsymbol{w}_h;\boldsymbol{w}_h,\boldsymbol{v}_h)
  - \int_\Omega \boldsymbol{\phi} \cdot \boldsymbol{v}_h.
  \]
  At each time step $n$, the velocity solution $\boldsymbol{u}_h^n$ of \eqref{eq:full.discrete.problem} is characterized by
  \[
  (\Psi_{\boldsymbol{u}_{h}^{n-1},\boldsymbol{f}^n} \boldsymbol{u}_h^n,\boldsymbol{v}_h)_{\Zh{k}} = 0
  \qquad \forall \boldsymbol{v}_h \in \Zh{k}.
  \]
  The existence of a solution to this problem can be proved using \cite[Chapter IV Corollary 1.1]{Girault.Raviart:86}.
  In particular, we need to show that, for fixed $\boldsymbol{z}_h \in \Zh{k}$ and $\boldsymbol{\phi} \in \boldsymbol{L}^2(\Omega) $,
  \begin{itemize}
  \item $\Psi_{\boldsymbol{z}_h,\boldsymbol{\phi}}$ is continuous;
  \item $\Psi_{\boldsymbol{z}_h,\boldsymbol{\phi}}$ is coercive, i.e., there exists $\mu > 0$ such that
    $(\Psi_{\boldsymbol{z}_h,\boldsymbol{\phi}} \boldsymbol{w}_h , \boldsymbol{w}_h)_{\Zh{k}} \geq 0$
    for all $ \boldsymbol{w}_h \in \Zh{k}$ with
    $\norm{\Zh{k}}{\boldsymbol{w}_h} = \mu$.
  \end{itemize}
   Since continuity is a standard result, we only provide a brief justification. We proceed by taking a sequence $(\boldsymbol{w}_h^{j} \in \Zh{k})_{j \in \Natural}$ such that $\lim_{j \to \infty} \norm{\Zh{k}}{\boldsymbol{w}_h - \boldsymbol{w}_h^{j}} = 0$ and showing that
  \begin{equation}\label{eq:Psi.continuity}
    \lim_{j \to \infty} \norm{\Zh{k}}{ \Psi_{\boldsymbol{z}_h,\boldsymbol{\phi}} \boldsymbol{w}_h - \Psi_{\boldsymbol{z}_h,\boldsymbol{\phi}} \boldsymbol{w}_h^{j}} = 0.
  \end{equation}
  We first write $
  (\Psi_{\boldsymbol{z}_h,\boldsymbol{\phi}} \boldsymbol{w}_h - \Psi_{\boldsymbol{z}_h,\boldsymbol{\phi}} \boldsymbol{w}_h^{j},\boldsymbol{v}_h)_{\Zh{k}} = \mathfrak{T}_1 + \mathfrak{T}_2 + \mathfrak{T}_3 + \mathfrak{T}_4,
  $
  where 
  $\mathfrak{T}_1 \coloneqq \int_\Omega \frac{\boldsymbol{w}_h - \boldsymbol{w}_h^{j}}{\tstep} \cdot \boldsymbol{v}_h$, 
  $\mathfrak{T}_2 \coloneqq a_h(\boldsymbol{w}_h,\boldsymbol{v}_h) - a_h(\boldsymbol{w}_h^{j},\boldsymbol{v}_h) $, 
  $\mathfrak{T}_3 \coloneqq c_h(\boldsymbol{w}_h,\boldsymbol{w}_h,\boldsymbol{v}_h) - c_h(\boldsymbol{w}_h^{j},\boldsymbol{w}_h^{j},\boldsymbol{v}_h)$ and 
  $\mathfrak{T}_4 \coloneqq j_h(\boldsymbol{w}_h;\boldsymbol{w}_h,\boldsymbol{v}_h) - j_h(\boldsymbol{w}_h^{j};\boldsymbol{w}_h^{j},\boldsymbol{v}_h)$.

  We next bound the above terms. For the first term, we use a Cauchy--Schwarz inequality to write
  $
  \mathfrak{T}_1 \leq \frac{1}{\tstep} \norm{\boldsymbol{L}^2(\Omega)}{\boldsymbol{w}_h - \boldsymbol{w}_h^{j}} \norm{\boldsymbol{L}^2(\Omega)}{\boldsymbol{v}_h}$.
  For the second term, using the boundedness \eqref{eq:ah:continuity} of $a_h$, we have that
  $
  \mathfrak{T}_2 \lesssim \left( \norm{1,r,h}{\boldsymbol{w}_h}^r + \norm{1,r,h}{\boldsymbol{w}_h^{j}}^r  \right)^{\frac{r-\underline{r}}{r}} \norm{1,r,h}{\boldsymbol{w}_h - \boldsymbol{w}_h^{j} }^{\underline{r}-1} \norm{1,r,h}{\boldsymbol{v}_h}$.
  For the third term, after adding and subtracting $c_h(\boldsymbol{w}_h^{j},\boldsymbol{w}_h,\boldsymbol{v}_h)$, using the linearity of $c_h$ in its arguments, and concluding with $(2,2,\infty)$-H\"{o}lder inequalities, we can show that
  $
  \mathfrak{T}_3 \lesssim \left(\norm{\boldsymbol{L}^2(\Omega)}{\boldsymbol{w}_h} + \norm{\boldsymbol{L}^2(\Omega)}{\boldsymbol{w}_h^{j}}\right)\norm{\boldsymbol{L}^2(\Omega)}{\boldsymbol{w}_h - \boldsymbol{w}_h^{j}} \norm{1,\infty,h}{\boldsymbol{v}_h}$.
  For the last term, we add and subtract $j_h(\boldsymbol{w}_h;\boldsymbol{w}_h^{j},\boldsymbol{v}_h)$ and, after noticing that
  $\left|\gamma_F(\boldsymbol{w}_h) - \gamma_F(\boldsymbol{w}_h^{j})\right| \leq \norm{\boldsymbol{L}^\infty(F)}{(\boldsymbol{w}_h - \boldsymbol{w}_h^{j})\cdot \normalF}$ and using a discrete trace inequality on $\norm{\boldsymbol{L}^2(F)}{\jump{\boldsymbol{w}_h^{j}}}$ and $\norm{\boldsymbol{L}^2(F)}{\jump{\boldsymbol{v}_h}}$, we obtain
  \begin{equation*}
    \begin{aligned}
      \mathfrak{T}_4
      \lesssim \seminorm{\boldsymbol{w}_h}{\boldsymbol{w}_h - \boldsymbol{w}_h^{j}} \seminorm{\boldsymbol{w}_h}{\boldsymbol{v}_h}
      + \left(
      \max_{F \in \Fh } h_F^{-1}\norm{\boldsymbol{L}^\infty(\omega_F)}{\boldsymbol{w}_h - \boldsymbol{w}_h^{j} }  \right)\norm{\boldsymbol{L}^2(\Omega)}{\boldsymbol{w}_h^{j}} \norm{\boldsymbol{L}^2(\Omega)}{\boldsymbol{v}_h} \, .
    \end{aligned}
  \end{equation*}
  Combining the above estimates, recalling that $\underline{r} - 1 \overset{\eqref{eq:sobolev:indexes}}\geq 0$,
  that all norms on a finite dimensional space are equivalent,
  and that all the norms applied to $\boldsymbol{w}_h$, $\boldsymbol{w}_h^{j}$, $\boldsymbol{v}_h$ are finite gives \eqref{eq:Psi.continuity}.
  \smallskip

  For the coercivity, we notice that, using Lemma~\ref{lem:ah:norm:equivalence}, the non dissipativity \eqref{eq:ch:non-dissipativity} of $c_h$, and the definition~\eqref{eq:convective.seminorm} of the seminorm $\seminorm{\boldsymbol{z}_h}{\cdot}$, we can write
  \[
  (\Psi_{\boldsymbol{z}_h,\boldsymbol{\phi}} \boldsymbol{w}_h,\boldsymbol{w}_h )_{\Zh{k}}
  \gtrsim
  \int_\Omega \frac{\boldsymbol{w}_h - \boldsymbol{z}_h}{\tstep} \cdot \boldsymbol{w}_h + \nu \norm{1,r,h}{\boldsymbol{w}_h}^r + \seminorm{\boldsymbol{w}_h}{\boldsymbol{w}_h}^2-\int_\Omega \boldsymbol{\phi} \cdot \boldsymbol{w}_h.
  \]
  Using Cauchy-Schwarz inequalities on the integrals followed by the Young-type inequalty $ab \le a^2 + \frac14 b^2$, we get
  \begin{equation*}
    \begin{aligned}
      (\Psi_{\boldsymbol{z}_h,\boldsymbol{\phi}} \boldsymbol{w}_h,\boldsymbol{w}_h )_{\Zh{k}}
      &\gtrsim
      \frac{1}{\tstep}\norm{\boldsymbol{L}^2(\Omega)}{\boldsymbol{w}_h}^2 -  \frac{1}{4\tstep}\norm{\boldsymbol{L}^2(\Omega)}{\boldsymbol{w}_h}^2 - \frac{1}{\tstep}\norm{\boldsymbol{L}^2(\Omega)}{\boldsymbol{z}_h}^2 + \nu \norm{1,r,h}{\boldsymbol{w}_h}^r \\
      &\quad+ \seminorm{\boldsymbol{w}_h}{\boldsymbol{w}_h}^2
      -\tstep \norm{\boldsymbol{L}^2(\Omega)}{\boldsymbol{\phi}}^2 - \frac{1}{4\tstep} \norm{\boldsymbol{L}^2(\Omega)}{\boldsymbol{w}_h}^2 \\
      & \geq \frac{1}{\tstep}
      \left[
        \frac{1}{2} \left(
        \norm{\boldsymbol{L}^2(\Omega)}{\boldsymbol{w}_h}^2
        + \tstep \big(\nu \norm{1,r,h}{\boldsymbol{w}_h}^r + \seminorm{\boldsymbol{w}_h}{\boldsymbol{w}_h}^2\big)
        \right) - \big(
        \norm{\boldsymbol{L}^2(\Omega)}{\boldsymbol{z}_h}^2 + \tstep^2 \norm{\boldsymbol{L}^2(\Omega)}{\boldsymbol{\phi}}^2
        \big)
        \right] \\
      &\gtrsim
      \frac{1}{\tstep}
      \left[
        \frac{1}{2} \norm{\Zh{k}}{\boldsymbol{w}_h}^2
        - \big(
        \norm{\boldsymbol{L}^2(\Omega)}{\boldsymbol{z}_h}^2 + \tstep^2\norm{\boldsymbol{L}^2(\Omega)}{\boldsymbol{\phi}}^2
        \big)
        \right].
    \end{aligned}
  \end{equation*}
  If we choose $\mu$ equal to the square root of $2\big(\norm{\boldsymbol{L}^2(\Omega)}{\boldsymbol{z}_h}^2 + \tstep^2\norm{\boldsymbol{L}^2(\Omega)}{\boldsymbol{\phi}}^2\big)$, it follows $(\Psi_{\boldsymbol{z}_h,\boldsymbol{\phi}} \boldsymbol{w}_h , \boldsymbol{w}_h)_{\Zh{k}} \geq 0$ for all $ \boldsymbol{w}_h \in \Zh{k}$ with $\norm{\Zh{k}}{\boldsymbol{w}_h} = \mu$.

  The existence of the pressure component follows from the discrete inf-sup condition
  \begin{equation}\label{eq:inf-sup}
  \norm{L^{r'}(\Omega)}{q_h}
  \lesssim
  \sup_{\boldsymbol{v}_h \in \Vh{k} \setminus \{\boldsymbol{0}\}} \frac{b_h(\boldsymbol{v}_h,q_h)}{\norm{1,r,h}{\boldsymbol{v}_h}}
  \qquad \forall q_h \in \Qh{k-1} \, ,
  \end{equation}
  combined with a standard argument hinging on the fact that $\Psi_{\boldsymbol{u}_h^{n-1},\boldsymbol{f}^n} \boldsymbol{u}_h^n \in \mathrm{Im}(\mathcal{B}_h^{*})$, where $\mathcal{B}_h^{*} : \Qh{k-1} \to (\Vh{k})^{*}$ is defined as $\langle \mathcal{B}_h^{*} q_h , \boldsymbol{v}_h \rangle \coloneqq -b_h(\boldsymbol{v}_h,q_h)$; see for instance \cite[Chapter 5]{Boffi.Brezzi.ea:13}.
  To prove \eqref{eq:inf-sup}, one can proceed by exhibiting a Fortin operator \cite{Boffi.Brezzi.ea:13} (an example in $L^r$ norms can be found in \cite[Lemma 7.4]{Botti.Castanon-Quiroz.ea:21}). The only point to check is the boundedness of the BDM interpolator, which can be inferred as in \cite[Remark 2.2]{Beirao-da-Veiga.Di-Pietro.ea:25}, where the Raviart--Thomas--Nédélec space is considered, together with trace inequalities.
  \medskip\\
  \underline{\textbf{Stability estimate.}}
  Taking $(\boldsymbol{v}_h, q_h) = (\boldsymbol{u}_h^n,p_h^n)$ in \eqref{eq:full.discrete.problem}, we obtain
  \begin{equation}\label{eq:stab.basic}
    \int_\Omega \frac{\boldsymbol{u}_h^n - \boldsymbol{u}_h^{n-1}}{\tstep} \cdot \boldsymbol{u}_h^n
    + a_h(\boldsymbol{u}_h^n,\boldsymbol{u}_h^n)
    + c_h(\boldsymbol{u}_h^{n},\boldsymbol{u}_h^{n},\boldsymbol{u}_h^n)
    + j_h(\boldsymbol{u}_h^{n};\boldsymbol{u}_h^{n},\boldsymbol{u}_h^n)
    = \int_\Omega \boldsymbol{f}^n \cdot \boldsymbol{u}_h^n.
  \end{equation}
  Since $\boldsymbol{\nabla} \cdot \boldsymbol{u}_h^n = 0$ (cf. Section \ref{sec:add-not}), $c_h(\boldsymbol{u}_h^{n},\boldsymbol{u}_h^{n},\boldsymbol{u}_h^n)=0$ by \eqref{eq:ch:non-dissipativity}.
  Moreover, by \eqref{eq:convective.seminorm}, $j_h(\boldsymbol{u}_h^{n};\boldsymbol{u}_h^{n},\boldsymbol{u}_h^n) = \seminorm{\boldsymbol{u}_h^n}{\boldsymbol{u}_h^n}^2$.
  Finally, from \eqref{eq:ah:norm:equivalence}, we infer $a_h(\boldsymbol{u}_h^n,\boldsymbol{u}_h^n) \gtrsim \nu \norm{1,r,h}{\boldsymbol{u}_h^n}^r \ge 0$.
  Using these results and applying the Cauchy--Schwarz inequality to the right-hand side of the previous equation, we obtain
  \begin{equation}\label{eq:stab.initial}
  \int_\Omega \frac{\boldsymbol{u}_h^n - \boldsymbol{u}_h^{n-1}}{\tstep} \cdot \boldsymbol{u}_h^n \le
    \norm{\boldsymbol{L}^2(\Omega)}{\boldsymbol{f}^n} \norm{\boldsymbol{L}^2(\Omega)}{\boldsymbol{u}_h^n}.
  \end{equation}
  After multiplying by $\tstep$ and rearranging, we get
  \[
  \norm{\boldsymbol{L}^2(\Omega)}{\boldsymbol{u}_h^n}^2
  \leq \tstep \int_\Omega \boldsymbol{f}^n \cdot \boldsymbol{u}_h^n
  + \int_\Omega \boldsymbol{u}_h^{n-1} \cdot \boldsymbol{u}_h^n.
  \]
  Applying a Cauchy--Schwarz inequality to the integrals in the right-hand side and simplifying, we arrive at
  \[
  \norm{\boldsymbol{L}^2(\Omega)}{\boldsymbol{u}_h^n}
  \leq
  \tstep \norm{\boldsymbol{L}^2(\Omega)}{\boldsymbol{f}^n} + \norm{\boldsymbol{L}^2(\Omega)}{\boldsymbol{u}_h^{n-1}}  .
  \]
  Proceeding recursively, we obtain the following bound for the $\boldsymbol{L}^2$-norm of $\boldsymbol{u}_h^n$:
  \begin{equation}\label{eq:stab.uhn.l2}
    \norm{\boldsymbol{L}^2(\Omega)}{\boldsymbol{u}_h^n}
    \leq
    \sum_{i=1}^n\tstep \norm{\boldsymbol{L}^2(\Omega)}{\boldsymbol{f}^i}
    + \norm{\boldsymbol{L}^2(\Omega)}{\boldsymbol{u}_h^0} .
  \end{equation}
  Going back to \eqref{eq:stab.basic} with $n$ replaced by $\ell$, using the stability properties of $a_h$, $c_h$, and $j_h$ as above, applying the relation
  \begin{equation}\label{eq:relation.unsteady}
    (\alpha - \beta) \alpha
    = \frac12[ \alpha^2 + (\alpha - \beta)^2 - \beta^2 ]
    \ge \frac12(\alpha^2 - \beta^2)
  \end{equation}
  for the first term in the left-hand side, and bounding the right-hand side using \eqref{eq:stab.uhn.l2}, we infer
  \begin{multline*}
    \frac{1}{2 \, \tstep} \left(\norm{\boldsymbol{L}^2(\Omega)}{\boldsymbol{u}_h^{\ell}}^2
      - \norm{\boldsymbol{L}^2(\Omega)}{\boldsymbol{u}_h^{\ell-1}}^2 \right)
    + \nu \norm{1,r,h}{\boldsymbol{u}_h^{\ell}}^r
    + \seminorm{\boldsymbol{u}_h^{\ell}}{\boldsymbol{u}_h^{\ell}}^2
    \\
    \lesssim
    \norm{\boldsymbol{L}^2(\Omega)}{\boldsymbol{f}^{\ell}}\sum_{i=1}^{\ell}\tstep \norm{\boldsymbol{L}^2(\Omega)}{\boldsymbol{f}^i}
    +\norm{\boldsymbol{L}^2(\Omega)}{\boldsymbol{f}^{\ell}}\norm{\boldsymbol{L}^2(\Omega)}{\boldsymbol{u}_h^0}.
  \end{multline*}
  Observing that, whenever $\ell \le n$, it holds $\sum_{i=1}^{\ell}\tstep \norm{\boldsymbol{L}^2(\Omega)}{\boldsymbol{f}^i} \leq \sum_{i=1}^n\tstep \norm{\boldsymbol{L}^2(\Omega)}{\boldsymbol{f}^i}$, and summing for $\ell=1,\ldots,n$ the equation here above, after multiplying by $\tstep$ we get:
  \begin{equation*}
    \begin{aligned}
      &\frac{1}{2} \sum_{\ell=1}^n \rbracket{\norm{\boldsymbol{L}^2(\Omega)}{\boldsymbol{u}_h^{\ell}}^2
        - \norm{\boldsymbol{L}^2(\Omega)}{\boldsymbol{u}_h^{\ell-1}}^2}
      + \sum_{\ell=1}^n \tstep \left(
      \nu \norm{1,r,h}{\boldsymbol{u}_h^{\ell}}^r
      + \seminorm{\boldsymbol{u}_h^\ell}{\boldsymbol{u}_h^{\ell}}^2
      \right)
      \\
      &\qquad
      \lesssim
      \sum_{\ell=1}^n \tstep \norm{\boldsymbol{L}^2(\Omega)}{\boldsymbol{f}^{\ell}}\sum_{i=1}^n\tstep \norm{\boldsymbol{L}^2(\Omega)}{\boldsymbol{f}^i}
      + \sum_{\ell=1}^n\tstep \norm{\boldsymbol{L}^2(\Omega)}{\boldsymbol{f}^{\ell}}\norm{\boldsymbol{L}^2(\Omega)}{\boldsymbol{u}_h^0} ,
      \\
      &\qquad
      \lesssim
      \left(
      \sum_{\ell=1}^n \tstep \norm{\boldsymbol{L}^2(\Omega)}{\boldsymbol{f}^{\ell}}
      \right)^2
      + \norm{\boldsymbol{L}^2(\Omega)}{\boldsymbol{u}_h^0}^2,
    \end{aligned}
  \end{equation*}
  where the conclusion follows applying the Young's inequality to the second term in the right-hand side.
  Noticing that the first sum on the left-hand side is a telescoping sum, and rearranging, we obtain the desired result.
\end{proof}

\section{Velocity error analysis} \label{sec:Velocity error analysis}

\subsection{Regularity assumptions}
For the convergence analysis, we require the following regularity on the exact solution (Assumption \ref{reg:ass:u}).
\begin{assumption}[Regularity of the continuous solution]\label{reg:ass:u}
  There is $m \in \{1,\ldots,k\}$ such that the solution $\boldsymbol{u}$ of the continuous problem \eqref{eq:continuous.problem} satisfies
  \[
  \boldsymbol{u} \in \boldsymbol{H}^2(0,\tF; \boldsymbol{L}^2(\Omega)) \cap \boldsymbol{H}^1(0,\tF; \boldsymbol{H}^{m+1}(\Th))
  \cap \boldsymbol{C}^0([0,\tF]; \Xi) \, ,
  \]
  where
  \[
  \Xi \coloneqq \boldsymbol{W}_{0}^{1,r}(\Omega)
  \cap \boldsymbol{W}^{1,\infty}(\Th)
  \cap \boldsymbol{W}^{m+1,\bar{r}}(\Th) \, .
  \]
  Furthermore, the stress is assumed to satisfy
  \[
  \nu^{-\frac12} \boldsymbol{\sigma}(\boldsymbol{\nabla u})
  = \nu^{\frac12} |\boldsymbol{\nabla u}|^{r-2} \boldsymbol{\nabla u}
  \in \boldsymbol{C}^0([0,\tF];\boldsymbol{W}^{m+1,r'}(\Th) \cap \boldsymbol{H}^{m+\frac12}(\Th))\, .
  \]
\end{assumption}

\begin{remark}[Regularity assumptions]
    In the present context, we are interested in uncovering two key features of the method: its robustness in convection-dominated situations and the orders of convergence achievable for smooth solutions.
    In particular, the assumption $\boldsymbol{u} \in \boldsymbol{C}^0([0,\tF]; \boldsymbol{W}^{1,\infty}(\Omega))$ is required to make the Reynolds-semi-robust estimate of the convective term possible; see Lemma~\ref{lem:convective:error}.
    
    The assumptions involving the exponent $m$, on the other hand, make it possible to show to which extent it is possible to exploit the selected polynomial degree.
    The problem of identifying assumptions on the domain and on the right-hand side under which this regularity can be proved a priori is an altogether different one, and lies out of the scope of the present work.
    We notice, however, this is by no means specific to the non-linear problem considered here, in the sense that finite element error estimates are usually formulated by assuming (and not proving) sufficient regularity requirements for the exact solution.
    
    For the sake of completeness, we notice that a convergence of the scheme to minimal regularity solutions could in all likelihood be carried out using compactness arguments in the spirit of~\cite{prohl2008,Di-Pietro.Ern:10,Burman.Ern:08}.
    Such approach is, however, unsuitable to show the Reynolds-semi-robust nature of the scheme, which is precisely the purpose of the present work.
\end{remark}

\subsection{Error equation}

Let $\boldsymbol{u}$ and $ \{ \boldsymbol{u}_h^i \}_{i=1,\ldots,N} $ be the velocity fields solutions of the weak formulation of problem \eqref{eq:continuous.problem} and of the discrete problem \eqref{eq:full.discrete.problem}, respectively, and set
\begin{equation}\label{eq:ehn}
  \boldsymbol{e}_h^n \coloneqq \boldsymbol{u}_h^n - \projRT{k} \boldsymbol{u}^n.
\end{equation}
Notice that $\boldsymbol{e}_h^n \in \Zh{k}$, since $\boldsymbol{u}_h^n \in \Zh{k} $ (cf. Section \eqref{sec:add-not}) and also  $\projRT{k} \boldsymbol{u}^n \in \Zh{k}$ (see Remark~\ref{rem:int.u.Zh}).

\begin{proposition}[Error equation]\label{prop:EE}
  Let Assumption \ref{reg:ass:u} hold.
  The following error equation holds: For all $\boldsymbol{v}_h \in \Zh{k}$,
  \begin{multline}\label{eq:error:equation}
    \int_{\Omega} \frac{\boldsymbol{e}_h^n - \boldsymbol{e}_h^{n-1}}{\tstep} \cdot \boldsymbol{v}_h
    + \rbracket{a_h(\boldsymbol{u}_h^n,\boldsymbol{v}_h) - a_h(\projRT{k}\boldsymbol{u}^n,\boldsymbol{v}_h)}
    + c_h(\boldsymbol{u}_h^n,\boldsymbol{e}_h^n,\boldsymbol{v}_h)
    + j_h(\boldsymbol{u}_h^n;\boldsymbol{e}_h^n,\boldsymbol{v}_h)
    \\
    = \timeTerm{n} (\boldsymbol{v}_h)
    + \diffTerm{n} (\boldsymbol{v}_h)
    + \convTerm{n} (\boldsymbol{v}_h)
    + \upwTerm{n} (\boldsymbol{v}_h),
  \end{multline}
  with
  \begin{alignat}{2}
    \timeTerm{n} (\boldsymbol{v}_h)
    &\coloneqq
    \int_{\Omega} \partial_t \boldsymbol{u}^n \cdot \boldsymbol{v}_h
    - \int_{\Omega} \frac{\projRT{k} \boldsymbol{u}^n - \projRT{k} \boldsymbol{u}^{n-1}}{\tstep} \cdot \boldsymbol{v}_h \, , \label{eq:consistency:error:time}\\
    \diffTerm{n} (\boldsymbol{v}_h)
    &\coloneqq -\int_{\Omega} \boldsymbol{\nabla} \cdot \boldsymbol{\sigma}(\boldsymbol{\nabla} \boldsymbol{u}^n) \cdot \boldsymbol{v}_h
    - a_h(\projRT{k} \boldsymbol{u}^n,\boldsymbol{v}_h) \, , \label{eq:consistency:error:diffusive}\\
    \convTerm{n} (\boldsymbol{v}_h)
    &\coloneqq \int_{\Omega} (\boldsymbol{u}^n \cdot \boldsymbol{\nabla}) \boldsymbol{u}^n \cdot \boldsymbol{v}_h
    - c_h (\boldsymbol{u}_h^n,\projRT{k} \boldsymbol{u}^n,\boldsymbol{v}_h) \, , \label{eq:consistency:error:convective} \\
    \upwTerm{n} (\boldsymbol{v}_h)
    &\coloneqq
    -j_h(\boldsymbol{u}_h^n;\projRT{k} \boldsymbol{u}^n,\boldsymbol{v}_h) \, . \label{eq:consistency:error:upwind}
  \end{alignat}
\end{proposition}

\begin{proof}
  Inferring $\boldsymbol{f}^n$ from \eqref{eq:continuous.problem.momentum} evaluated at time $t^n$, we have
  \[
  \int_{\Omega} \boldsymbol{f}^n \cdot \boldsymbol{v}_h =
  \int_{\Omega} \left(
  \partial_t \boldsymbol{u}^n - \boldsymbol{\nabla} \cdot \boldsymbol{\sigma}(\boldsymbol{\nabla} \boldsymbol{u}^n)+(\boldsymbol{u}^n \cdot \boldsymbol{\nabla}) \boldsymbol{u}^n
  \right)\cdot \boldsymbol{v}_h,
  \]
  where the term involving the pressure vanishes due to an integration by parts followed by $\boldsymbol{\nabla} \cdot \boldsymbol{v}_h=0$ and $\restr{(\boldsymbol{v}_h \cdot \normal)}{\partial\Omega}=0$. Substituting the above expression into the discrete momentum equation \eqref{eq:full.discrete.problem:momentum} and noticing that $b_h(\boldsymbol{v}_h,p_h^n) = 0$ since $\boldsymbol{v}_h \in \Zh{k}$, we infer
  \begin{multline*}
    \int_{\Omega} \left( \partial_t \boldsymbol{u}^n - \boldsymbol{\nabla} \cdot \boldsymbol{\sigma}(\boldsymbol{\nabla} \boldsymbol{u}^n)+(\boldsymbol{u}^n \cdot \boldsymbol{\nabla}) \boldsymbol{u}^n \right) \cdot \boldsymbol{v}_h
    =
    \\
    \int_\Omega \frac{\boldsymbol{u}_h^n - \boldsymbol{u}_h^{n-1}}{\tstep} \cdot \boldsymbol{v}_h
    + a_h(\boldsymbol{u}_h^n,\boldsymbol{v}_h)
    + c_h(\boldsymbol{u}_h^{n},\boldsymbol{u}_h^{n},\boldsymbol{v}_h)
    + j_h(\boldsymbol{u}_h^{n};\boldsymbol{u}_h^{n},\boldsymbol{v}_h).
  \end{multline*}
  Subtracting from both sides of the above equation the quantity
  \[\int_{\Omega} \frac{\projRT{k} \boldsymbol{u}^n - \projRT{k} \boldsymbol{u}^{n-1}}{\tstep} \cdot \boldsymbol{v}_h
  + a_h(\projRT{k} \boldsymbol{u}^n,\boldsymbol{v}_h)
  +c_h (\boldsymbol{u}_h^n,\projRT{k} \boldsymbol{u}^n,\boldsymbol{v}_h)
  + j_h(\boldsymbol{u}_h^n;\projRT{k} \boldsymbol{u}^n,\boldsymbol{v}_h),
  \]
  recalling the definition of $\boldsymbol{e}_h^n$, and rearranging leads to the desired conclusion.
\end{proof}

\subsection{Convection- and diffusion-dominated elements and faces}

For all $T \in \Th$ and all $F \in \Fh$, let $\hat{K}_T^n \in [0,\infty]$ and $\hat{K}_F^n \in [0,\infty]$ be defined by 
\begin{equation}\label{eq:hat.K.T.F}
  \hat{K}_T^n \coloneqq \norm{\boldsymbol{L}^\infty(\omega_T)}{| \boldsymbol{\nabla} \boldsymbol{u}^n |^{r-2}},\quad
  \hat{K}_F^n \coloneqq \max\left\{ \norm{\boldsymbol{L}^\infty(F)}{| \boldsymbol{\nabla} \boldsymbol{u}^n |^{r-2}}, h_F^{2-r} \norm{\boldsymbol{L}^\infty(F)}{ | \jump{\projRT{k} \boldsymbol{u}^n} |^{r-2}} \right\},
\end{equation}
with the convention that $\hat{K}_T^n = \infty$ if the restriction of $| \boldsymbol{\nabla} \boldsymbol{u}^n |^{r-2}$ is not in $\boldsymbol{L}^\infty(\omega_T)$, and that $\hat{K}_F^n = \infty$
whenever the functions involved in its definition are not in $\boldsymbol{L}^\infty(F)$.
The quantities $\nu \hat{K}_T^n$ and $\nu \hat{K}_F^n$ are representative of the magnitude of the (velocity-dependent) diffusion coefficient in the constitutive law \eqref{def:sigma} at time $t^n$ on the element $T$ and the face $F$, respectively.
Notice that, when considering faces, we also include jumps since these represent an important part of the discrete gradient for discontinuous Galerkin methods.
We introduce the partitions $\Th = \Thc \sqcup \Thd$ and $\Fh = \Fhc \sqcup \Fhd$ of the sets of elements and faces into convection- and diffusion-dominated such that
\begin{gather} \label{eq:ReyT>1new}
  \nu \hat{K}_T^n < \norm{\boldsymbol{L}^\infty(T)}{\boldsymbol{u}^n} h_T
  \qquad \forall T \in \Thc,
  \\ \label{eq:ReyT<=1new}
  \norm{\boldsymbol{L}^\infty(T)}{\boldsymbol{u}^n} \le \nu \hat{K}_T^n h_T^{-1}
  \qquad \forall T \in  \Thd,
\end{gather}
and
\begin{gather} \label{eq:ReyF>1new}
  \nu \hat{K}_F^n < \norm{L^{\infty}(F)}{\boldsymbol{u}^n\cdot\normalF} h_F
  \qquad \forall F \in \Fhc,
  \\ \label{eq:ReyF<=1new}
  \norm{L^{\infty}(F)}{\boldsymbol{u}^n\cdot\normalF} \le \nu \hat{K}_F^n h_F^{-1}          \qquad \forall F \in \Fhd.
\end{gather}
We also note that
\begin{equation}\label{eq:ReyF>1new.gamma-u}
  \nu \hat{K}_F^n \lesssim \gamma_F(\boldsymbol{u}_h^n) h_F
  \qquad \forall F \in \Fhc,
\end{equation}
which follows by combining \eqref{eq:ReyF>1new} and the bound
\begin{equation}\label{eq:gamma-u-bound}
  \norm{L^{\infty}(F)}{\boldsymbol{u}^n\cdot\normalF} \lesssim \gamma_F(\boldsymbol{u}_h^n).
\end{equation}
This bound can be rigorously justified (at least globally) by the fact that $\norm{L^{\infty}(\Omega)}{\boldsymbol{u}}$ is bounded by assumption at all times and $\gamma_F(\boldsymbol{u}_h^n) \ge C_F > 0$ (see \eqref{eq:gammaF}).

\subsection{Time-derivative error}

\begin{lemma}[Estimate of the time-derivative error]\label{lem:time:derivative:error}
  Recalling the definition~\eqref{eq:consistency:error:time} of $\timeTerm{n}(\cdot)$, under Assumption \ref{reg:ass:u}, for $n=1,\ldots,N$ and any $\delta >0 $, it holds
  \begin{equation}\label{eq:timeTerm:estimate}
    \timeTerm{n} (\boldsymbol{e}_h^n)
    \le
    \delta \norm{\boldsymbol{L}^2(\Omega)}{\boldsymbol{e}_h^n}^2
    + c(\delta) \timeErr,
  \end{equation}
  where
  \begin{equation}\label{eq:timeErr}
    \timeErr \coloneq
    \tstep \norm{\boldsymbol{L}^2(t^{n-1}, t^n; \boldsymbol{L}^2(\Omega))}{\partial_{tt}^2 \boldsymbol{u}}^2
    + \frac{1}{\tstep}\sum_{T \in \Th} h_T^{2(m+1)} \norm{\boldsymbol{L}^2(t^{n-1},t^n;\boldsymbol{H}^{m+1}(T))}{\partial_t \boldsymbol{u}}^2.
  \end{equation}
\end{lemma}

\begin{proof}
  After adding and subtracting $\frac{\boldsymbol{u}^n - \boldsymbol{u}^{n-1}}{\tstep}$ and using the Cauchy--Schwarz inequality along with a generalized Young inequality, we split and estimate $\timeTerm{n} (\boldsymbol{e}_h^n)$ as follows:
  \begin{equation}\label{eq:err:timeTerm}
    \timeTerm{n} (\boldsymbol{e}_h^n) \le \delta \norm{\boldsymbol{L}^2(\Omega)}{\boldsymbol{e}_h^n}^2
    +c(\delta)\left( \timeTerm{n}_1+ \timeTerm{n}_2 \right),
  \end{equation}
  where
  \[
  \timeTerm{n}_1
  \coloneqq
  \Norm{\boldsymbol{L}^2(\Omega)}{\partial_t \boldsymbol{u}^n - \frac{\boldsymbol{u}^n-\boldsymbol{u}^{n-1}}{\tstep}}^2 ,
  \qquad
  \timeTerm{n}_2
  \coloneqq
  \Norm{\boldsymbol{L}^2(\Omega)}{\frac{\boldsymbol{u}^n-\boldsymbol{u}^{n-1}}{\tstep} - \frac{\projRT{k}\boldsymbol{u}^n-\projRT{k}\boldsymbol{u}^{n-1}}{\tstep}}^2.
  \]
  Using $\boldsymbol{u}^n = \boldsymbol{u}^{n-1} + \tstep \, \partial_t \boldsymbol{u}^n - \int_{t^{n-1}}^{t^n} (t-t^{n-1}) \, \partial_{tt}^2 \boldsymbol{u}(t,\cdot) \, dt$ and Jensen's inequality, we obtain that
  \begin{equation}\label{eq:estimate:timeTerm.1}
    \timeTerm{n}_1
    \lesssim \tstep \norm{\boldsymbol{L}^2(t^{n-1},t^n;\boldsymbol{L}^2(\Omega))}{\partial_{tt}^2 \boldsymbol{u}}^2.
  \end{equation}
  In a similar way, since $\partial_t \projRT{k}\boldsymbol{u}(t,\cdot) =\projRT{k} \partial_t \boldsymbol{u}(t,\cdot) $ by linearity, Jensen's inequality and the approximation property \eqref{eq:approx:RT:ele} of $\projRT{k}$ with $(j,q,s)=(0,2,m)$  give
  \begin{equation}\label{eq:estimate:timeTerm.2}
    \begin{alignedat}{2}
      \timeTerm{n}_2
      &=\Norm{\boldsymbol{L}^2(\Omega)}{\frac{1}{\tstep}\int_{t^{n-1}}^{t^n} ( \partial_t \boldsymbol{u}(t,\cdot) - \projRT{k}\partial_t \boldsymbol{u}(t,\cdot) ) dt }^2
      \leq
      \frac{1}{\tstep} \int_{t^{n-1}}^{t^n} \norm{\boldsymbol{L}^2(\Omega)}{\partial_t \boldsymbol{u}(t,\cdot) - \projRT{k}\partial_t \boldsymbol{u}(t,\cdot)}^2 dt \\
      &\lesssim
      \frac{1}{\tstep} \int_{t^{n-1}}^{t^n} \sum_{T \in \Th} h_T^{2(m+1)} \seminorm{\boldsymbol{H}^{m+1}(T)}{\partial_t \boldsymbol{u}(t,\cdot)}^2 \leq
      \frac{1}{\tstep} \sum_{T \in \Th} h_T^{2(m+1)} \norm{\boldsymbol{L}^2(t^{n-1},t^n;\boldsymbol{H}^{m+1}(T))}{\partial_t \boldsymbol{u}}^2.
    \end{alignedat}
  \end{equation}
  The conclusion follows plugging the estimates \eqref{eq:estimate:timeTerm.1} and \eqref{eq:estimate:timeTerm.2} into \eqref{eq:err:timeTerm} with $c(\delta)$ including the hidden constants in the estimates of $\timeTerm{n}_1$ and $\timeTerm{n}_2$.
\end{proof}

\subsection{Diffusive error}

The following bound for the diffusive error distinguishes between diffusion and convection dominated elements/faces in order to emphasize the different expected error reduction rates also at the local level.

\begin{lemma}[Regime-dependent estimate of the diffusive error]\label{lem:diffusive:error}
  Recalling the definition~\eqref{eq:consistency:error:diffusive} of $\diffTerm{n}(\cdot)$, under Assumption \ref{reg:ass:u}, for $n=1,\ldots,N$ and any $\delta >0 $,  it holds
  \begin{equation}\label{eq:diffTerm:estimate}
    \diffTerm{n}(\boldsymbol{e}_h^n)
    \le
    \delta\left[
      a_h(\boldsymbol{u}_h^n,\boldsymbol{e}_h^n)
      - a_h(\projRT{k}\boldsymbol{u}^n,\boldsymbol{e}_h^n)
      + \nu \norm{1,r,h}{\boldsymbol{e}_h^n}^{\overline{r}}
      + \seminorm{\boldsymbol{u}_h^n}{\boldsymbol{e}_h^n}^2
      \right]
    + c(\delta) \diffErr,
  \end{equation}
  where
  \begin{equation}\label{eq:diffErr}
    \begin{aligned}
      \diffErr
      &\coloneqq
      \sum_{T \in \Thd}  h_T^{\underline{r}m} \nu
      \norm{\boldsymbol{L}^r(\Omega)}{\boldsymbol{\nabla} \boldsymbol{u}^n}^{\bar{r}-2}
      \seminorm{\boldsymbol{W}^{m+1,r}(\TT)}{\boldsymbol{u}^n}^{\underline{r}}
      + \sum_{T \in \Thc} h_T^{2 m +1} \norm{\boldsymbol{L}^\infty(T)}{\boldsymbol{u}^n} \seminorm{\boldsymbol{H}^m(T)}{\boldsymbol{u}^n}^2
      \\
      &\quad
      +  \sum_{F \in \Fhd} h_F^{r m} \nu \seminorm{\boldsymbol{W}^{m+1,r}(\TF)}{\boldsymbol{u}^n}^r
      + \left[
        \sum_{F \in \Fhd} h_F^{r'(m+1)} \left(\nu^{-\frac{1}{\overline{r}}}\seminorm{\boldsymbol{W}^{m+1,r'}(\TF)}{\boldsymbol{\sigma}(\boldsymbol{\nabla} \boldsymbol{u}^n)}\right)^{r'}
        \right]^{\frac{\overline{r}'}{r'}}
      \\
      &\quad
      + \sum_{F \in \Fhc} h_F^{2m+1}
      (\nu \hat{K}_F^n)^{-1} \seminorm{\boldsymbol{H}^{m+\frac12}(\TF)}{\boldsymbol{\sigma}(\boldsymbol{\nabla} \boldsymbol{u}^n)}^2
      + \sum_{F \in \Fhc} h_F^{2m+1} \norm{L^\infty(F)}{\boldsymbol{u}^n \cdot \normalF} \seminorm{\boldsymbol{H}^{m+1}(\TF)}{\boldsymbol{u}^n}^2.
    \end{aligned}
  \end{equation}
\end{lemma}
\begin{proof}
  Following the same argument as in \cite[Eq. (5.4)]{Beirao-da-Veiga.Di-Pietro.ea:24}, we obtain
  \begin{equation} \label{eq:estimate:D}
    \begin{aligned}
      \diffTerm{n} (\boldsymbol{e}_h^n) &=
      \int_\Omega \left(
      \boldsymbol{\sigma} ( \boldsymbol{\nabla} \boldsymbol{u}^n)- \boldsymbol{\sigma} ( \Gh \projRT{k} \boldsymbol{u}^n)
      \right) : \Gh \boldsymbol{e}_h^n \\
      &\qquad+\sum_{F \in \Fh} \int_F \left(
      \avg{\projL{k} \boldsymbol{\sigma}(\boldsymbol{\nabla} \boldsymbol{u}^n) - \boldsymbol{\sigma}(\boldsymbol{\nabla} \boldsymbol{u}^n) } \normalF
      \right) \cdot \jump{\boldsymbol{e}_h^n} \\
      &\qquad -\sum_{F \in \Fh} h_F^{1-r}\int_F \boldsymbol{\sigma}(\jump{\projRT{k} \boldsymbol{u}^n} ) \cdot \jump{\boldsymbol{e}_h^n}\eqcolon \diffTerm{n}_1 + \diffTerm{n}_2 + \diffTerm{n}_3.
    \end{aligned}
  \end{equation}
  \\
  \underline{\textbf{Estimate of $\diffTerm{n}_1$}.}
  For the first term, we write $\diffTerm{n}_1 = \sum_{T\in\Th}\diffTerm{n}_{1,T}$ and restrict our attention to a single element $T \in \Th$.
  By \cite[Lemma 2.1]{Beirao-da-Veiga.Di-Pietro.ea:24}, for any $\delta > 0$ it holds
  \begin{equation}\label{eq:estimate:D1.T}
    \diffTerm{n}_{1,T}
    \leq \delta \left[ \int_T \left( \boldsymbol{\sigma} (\Gh \boldsymbol{u}_h^n) - \boldsymbol{\sigma} (\Gh \projRT{k} \boldsymbol{u}^n) \right): \Gh \boldsymbol{e}_h^n \right] + c(\delta) \diffTerm{n}_{1,T,\mathrm{err}},
  \end{equation}
  where
  \begin{equation*}
    \diffTerm{n}_{1,T,\mathrm{err}} \coloneqq \nu \int_T \left( | \boldsymbol{\nabla} \boldsymbol{u}^n | + | \Gh \projRT{k} \boldsymbol{u}^n | \right)^{r-2} |\boldsymbol{\nabla} \boldsymbol{u}^n-\Gh \projRT{k} \boldsymbol{u}^n |^2.
  \end{equation*}
  For $T\in \Thd $, using the approximation and boundedness properties of $\Gh \circ \projRT{k}$ stated in Lemma~\ref{lem:approx:Gh}, we can apply the same argument that leads to \cite[Eq. (5.6) and Eq. (5.7)]{Beirao-da-Veiga.Di-Pietro.ea:24} to obtain
  \begin{equation}\label{eq:estimate:D1.T.err:diffusive}
    \begin{aligned}
      \diffTerm{n}_{1,T,\mathrm{err}}
      &\lesssim
      \nu h_T^{\underline{r}m} \norm{\boldsymbol{L}^r(\Omega)}{\boldsymbol{\nabla} \boldsymbol{u}^n}^{\bar{r}-2} \seminorm{\boldsymbol{W}^{m+1,r}(\TT)}{\boldsymbol{u}^n}^{\underline{r}} \, ,
    \end{aligned}
  \end{equation}
  where, in the case $r>2$, we also used the fact that $\nu^{\frac1r}\norm{\boldsymbol{L}^r(\Omega)}{\boldsymbol{\nabla} \boldsymbol{u}^n}$, $n=1,\ldots,N$, is uniformly bounded due to Assumption~\ref{reg:ass:u}.
  If $T\in \Thc$ and $r<2$, we have that
  \begin{equation} \label{eq:estimate:D1.T.err:advective:r<2}
    \begin{aligned}
      \diffTerm{n}_{1,T,\mathrm{err}}
      &\leq \nu \norm{L^\infty(T)}{|\boldsymbol{\nabla} \boldsymbol{u}^n |^{r-2}} \norm{\boldsymbol{L}^2(T)}{\boldsymbol{\nabla} \boldsymbol{u}^n-\Gh \projRT{k} \boldsymbol{u}^n }^2
      \\
      \overset{\eqref{eq:Gh:approx},\, \eqref{eq:hat.K.T.F},\,\eqref{eq:ReyT>1new}}&\lesssim
      h_T^{2 m +1} \norm{\boldsymbol{L}^\infty(T)}{\boldsymbol{u}^n} \seminorm{\boldsymbol{H}^{m+1}(T)}{\boldsymbol{u}^n}^2,
    \end{aligned}
  \end{equation}
  where the first inequality follows from the fact that $( | \boldsymbol{\nabla} \boldsymbol{u}^n | + | \Gh \projRT{k} \boldsymbol{u}^n | )^{r-2} \leq | \boldsymbol{\nabla} \boldsymbol{u}^n |^{r-2}$ since $r<2$.
  If $r\geq 2$, the same identical bound \eqref{eq:estimate:D1.T.err:advective:r<2} follows, but with a slightly different reasoning: for the first inequality, we use an $(\infty,1)$-H\"{o}lder inequality followed by a triangle inequality to write
  $\norm{L^\infty(T)}{(| \boldsymbol{\nabla} \boldsymbol{u}^n | + | \Gh \projRT{k} \boldsymbol{u}^n |)^{r-2}}
  \le\norm{L^\infty(T)}{| \boldsymbol{\nabla} \boldsymbol{u}^n |^{r-2}}
  + \norm{L^\infty(T)}{| \Gh \projRT{k} \boldsymbol{u}^n |^{r-2}}
  \lesssim \norm{L^\infty(\TT)}{| \boldsymbol{\nabla} \boldsymbol{u}^n |}^{r-2}$,
  where the conclusion follows noticing that $\Real^+ \ni x \mapsto x^{r-2} \in \Real$ is strictly increasing if $r > 2$ to write $\norm{L^\infty(T)}{|\cdot|^{r-2}} = \norm{L^\infty(T)}{\cdot}^{r-2}$ and invoking the $L^\infty$-boundedness \eqref{eq:Gh:stab} of $\Gh \circ \projRT{k}$.
  We conclude plugging \eqref{eq:estimate:D1.T.err:diffusive} and \eqref{eq:estimate:D1.T.err:advective:r<2} into \eqref{eq:estimate:D1.T}, thus obtaining
  \begin{equation}\label{eq:estimate:D1:final}
    \diffTerm{n}_1
    \le \delta \int_\Omega \left( \boldsymbol{\sigma} (\Gh \boldsymbol{u}_h^n) - \boldsymbol{\sigma} (\Gh \projRT{k} \boldsymbol{u}^n) \right): \Gh \boldsymbol{e}_h^n
    + c(\delta) \diffErr.
  \end{equation}
  \noindent\underline{\textbf{Estimate of $\diffTerm{n}_2$}.}
  We decompose the second term as $\diffTerm{n}_2 = \sum_{F \in \Fh} \diffTerm{n}_{2,F}$.
  For $F \in \Fhd$, we can apply verbatim the argument that leads to \cite[Eq.~(5.11)]{Beirao-da-Veiga.Di-Pietro.ea:24} and write
  \begin{equation}\label{eq:estimate:D2.F:diffusive}
    \diffTerm{n}_{2,F}
    \lesssim
    h_F^{m+1} \nu^{-\frac{1}{\overline{r}}} \seminorm{\boldsymbol{W}^{m+1,r'}(\TF)}{\boldsymbol{\sigma}(\boldsymbol{\nabla} \boldsymbol{u}^n)}
    ~ \nu^{\frac{1}{\overline{r}}}\big( h_F^{1-r}\norm{\boldsymbol{L}^r(F)}{\jump{\boldsymbol{e}_h^n}}^r \big)^{\frac{1}{r}}.
  \end{equation}
  For $F \in \Fhc$, first applying a $(2,\infty,2)$-H\"{o}lder inequality,
  then using $\norm{\boldsymbol{L}^{\infty}(F)}{\normalF} \leq 1$,
  and finally multiplying and dividing by $\nu^{\frac12}\gamma_F(\boldsymbol{u}_h^n)^{-\frac12}$ the resulting inequality, we get
  \[
  \begin{aligned}
    \diffTerm{n}_{2,F}
    &\lesssim
    \nu^{-\frac12} \nu^{\frac12}\gamma_F(\boldsymbol{u}_h^n)^{-\frac12}
    \norm{\boldsymbol{L}^{2}(F)}{\avg{\projL{k}\boldsymbol{\sigma}(\boldsymbol{\nabla} \boldsymbol{u}^n) - \boldsymbol{\sigma}(\boldsymbol{\nabla} \boldsymbol{u}^n)}}
    ~ \gamma_F(\boldsymbol{u}_h^n)^{\frac12}\norm{\boldsymbol{L}^2(F)}{\jump{\boldsymbol{e}_h^n}} \\
    \overset{\eqref{eq:ReyF>1new.gamma-u}}&\lesssim
    h_F^{m+\frac12} (\nu \hat{K}_F^n)^{-\frac12} \seminorm{\boldsymbol{H}^{m+\frac12}(\TF)}{\boldsymbol{\sigma}(\boldsymbol{\nabla} \boldsymbol{u}^n)}
    ~ \left( \gamma_F(\boldsymbol{u}_h^n) \norm{\boldsymbol{L}^2(F)}{\jump{\boldsymbol{e}_h^n}}^2 \right)^{\frac12},
  \end{aligned}
  \]
  where, in the last inequality, we have used also the approximation properties of the $L^2$-orthogonal projector $\projL{k}$.
  Using a $(r,r')$-H\"older inequality on the sum over $F\in\Fhd$, a Cauchy--Schwarz inequality on the sum over $F \in \Fhc$, and then applying generalized $(\overline{r},\overline{r}')$- and $(2,2)$-Young inequalities on $\Fhd$ and  on $\Fhc$, respectively, we obtain, for any $\delta > 0$,
  \begin{equation}\label{eq:estimate:D2:final}
    \diffTerm{n}_2
    \le \delta \left[ \nu \norm{1,r,h}{\boldsymbol{e}_h^n}^{\overline{r}} + \seminorm{\boldsymbol{u}_h^n}{\boldsymbol{e}_h^n}^2\right]
    + c(\delta) \diffErr.
  \end{equation}
  \noindent\underline{\textbf{Estimate of $\diffTerm{n}_3$}.} For the last term, observing that $\jump{\boldsymbol{u}^n}=0$ and then applying \cite[Lemma 2.1]{Beirao-da-Veiga.Di-Pietro.ea:24}, we can write, for any $\delta > 0$,
  \begin{equation}\label{eq:estimate:D3}
    \diffTerm{n}_3
    \le \delta
    \left[
      \sum_{F \in \Fh} h_F^{1-r}\int_ F \left( \boldsymbol{\sigma} (\jump{\boldsymbol{u}_h^n}) - \boldsymbol{\sigma}(\jump{\projRT{k} \boldsymbol{u}^n})  \right)\cdot \jump{\boldsymbol{e}_h^n}
      \right]  + c(\delta) \sum_{F \in \Fh} \diffTerm{n}_{3,F,\mathrm{err}},
  \end{equation}
  where
  \[
  \diffTerm{n}_{3,F,\mathrm{err}}
  \coloneqq \nu h_F^{1-r} \int_F \big| \jump{\projRT{k} \boldsymbol{u}^n} \big|^{r-2}
  \big| \jump{\projRT{k} \boldsymbol{u}^n-\boldsymbol{u}^n} \big|^2.
  \]
  For $F\in\Fhd$, using the fact that $\jump{\boldsymbol{u}^n} = \boldsymbol{0}$ to write $\jump{\projRT{k} \boldsymbol{u}^n} = \jump{\projRT{k} \boldsymbol{u}^n - \boldsymbol{u}^n}$, we have
  \begin{equation}\label{eq:estimate:D3.F.diffusive}
    \begin{aligned}
      \diffTerm{n}_{3,F,\mathrm{err}}
      = \nu h_F^{1-r} \norm{\boldsymbol{L}^r(F)}{\jump{\projRT{k} \boldsymbol{u}^n-\boldsymbol{u}^n}}^r
      \overset{\eqref{eq:approx:RT:face}}\lesssim
      h_F^{r m} \nu \seminorm{\boldsymbol{W}^{m+1,r}(\TF)}{\boldsymbol{u}^n}^r.
    \end{aligned}
  \end{equation}
  For $F \in \Fhc$, using an $(\infty,1)$-H\"older inequality, we get
  \begin{equation}\label{eq:estimate:D3.F.convective}
    \begin{aligned}
      \diffTerm{n}_{3,F,\mathrm{err}}
      &\leq \nu h_F^{1-r} \norm{\boldsymbol{L}^\infty(F)}{|\jump{\projRT{k} \boldsymbol{u}^n}|^{r-2}} \norm{\boldsymbol{L}^2(F)}{\jump{\projRT{k} \boldsymbol{u}^n -\boldsymbol{u}^n }}^2
      \\
      \overset{\eqref{eq:approx:RT:face}}&\lesssim
      h_F^{2m +1} \norm{L^\infty(F)}{\boldsymbol{u}^n \cdot \normalF} \seminorm{\boldsymbol{H}^{m+1}(\TF)}{\boldsymbol{u}^n}^2,
    \end{aligned}
  \end{equation}
  where we have written $\nu h_F^{1-r} \norm{\boldsymbol{L}^\infty(F)}{|\jump{\projRT{k} \boldsymbol{u}^n}|^{r-2}} \overset{\eqref{eq:hat.K.T.F}}\leq \nu h_F^{-1} \hat{K}_F^n \overset{\eqref{eq:ReyF>1new}}\lesssim \norm{L^\infty(F)}{\boldsymbol{u}^n \cdot \normalF} $ to obtain the second inequality.
  Gathering \eqref{eq:estimate:D3.F.diffusive} and \eqref{eq:estimate:D3.F.convective}, from \eqref{eq:estimate:D3} we deduce that
  \begin{equation}\label{eq:estimate:D3:final}
    \diffTerm{n}_3 \le \delta
    \sum_{F \in \Fh} h_F^{1-r}\int_ F \left( \boldsymbol{\sigma} (\jump{\boldsymbol{u}_h^n}) - \boldsymbol{\sigma}(\jump{\projRT{k} \boldsymbol{u}^n})  \right)\cdot \jump{\boldsymbol{e}_h^n}
    + c(\delta) \diffErr.
  \end{equation}
  \underline{\textbf{Conclusion}.} Plug \eqref{eq:estimate:D1:final}, \eqref{eq:estimate:D2:final} and \eqref{eq:estimate:D3:final} into \eqref{eq:estimate:D}.
\end{proof}

\begin{remark}[Non-degeneracy of the bound \eqref{eq:diffTerm:estimate}]
  The term $(\nu \hat{K}_F^n)^{-1}$ appearing in \eqref{eq:diffErr} does not indicate a potential degeneration of the estimate \eqref{eq:diffTerm:estimate} in the convection-dominated case.
  Indeed, that same term could be handled also using \eqref{eq:estimate:D2.F:diffusive} (as in the diffusive case), thus guaranteeing its boundedness. The presence of the factor $(\hat{K}_F^n)^{-1}$ can indeed be interpreted as a ``scaling'' associated to the higher order ``stress'' norm $\seminorm{\boldsymbol{H}^{m+\frac12}(\TF)}{\boldsymbol{\sigma}(\boldsymbol{\nabla} \boldsymbol{u}^n)}^2$, while the term $\nu^{-1}$ is balanced by the definition of $\boldsymbol{\sigma}(\boldsymbol{\nabla} \boldsymbol{u})$ which includes a factor $\nu$, see \eqref{def:sigma}. An analogous observation holds for the fourth addendum in \eqref{eq:diffErr}.
\end{remark}

\subsection{Convective error}

In Lemmas \ref{lem:convective:error} and \ref{lem:upw:stab:error} a positive constant $C_{\boldsymbol{u}}^\star$ will appear, which is of the form
\begin{equation}\label{eq:Custar}
  C_{\boldsymbol{u}}^\star = C
  \norm{\boldsymbol{L}^\infty(0,\tF; \boldsymbol{L}^\infty(\Omega))}{\boldsymbol{\nabla}\boldsymbol{u}}\, ,
\end{equation}
with $C$ only depending on the polynomial degree $k$ and the mesh regularity parameter. We take particular care in tracking this constant since it will appear in the (Gronwall related) exponential term of our final estimate.

\begin{lemma}[Estimate of the convective error]\label{lem:convective:error}
  Recalling the definition~\eqref{eq:consistency:error:convective} of $\convTerm{n}(\cdot)$, under Assumption \ref{reg:ass:u}, for $n=1,\ldots,N$ and any $\delta >0 $, it holds
  \begin{equation}\label{eq:convTerm:estimate}
    \convTerm{n}(\boldsymbol{e}_h^n)
    \le
    C_{\boldsymbol{u}}^\star \norm{\boldsymbol{L}^2(\Omega)}{\boldsymbol{e}_h^n}^2
    + \delta \left[\norm{\boldsymbol{L}^2(\Omega)}{\boldsymbol{e}_h^n}^2 
    + \seminorm{\boldsymbol{u}_h^n}{\boldsymbol{e}_h^n}^2
      \right]
    + c(\delta) \convErr,
  \end{equation}
  where the constant $C_{\boldsymbol{u}}^\star$ was introduced in \eqref{eq:Custar},
  while
  \begin{equation}\label{eq:convErr}
    \begin{aligned}
      \convErr
      &\coloneqq
      \sum_{T \in \Th} h_T^{2(m+1)} \norm{\boldsymbol{L}^\infty(T)}{\boldsymbol{\nabla} \boldsymbol{u}^n}^2  \seminorm{\boldsymbol{H}^{m+1}(T)}{\boldsymbol{u}^n}^2
      +
      \sum_{F \in \Fh} h_F^{2(m+1)} \norm{\boldsymbol{L}^\infty(\omega_F)}{\boldsymbol{\nabla} \boldsymbol{u}^n}^2 \seminorm{\boldsymbol{H}^{m+1}(\TF)}{\boldsymbol{u}^n}^2
      \\
      &\quad   +
      \sum_{F \in \Fh} h_F^{2m+1}\norm{L^\infty(F)}{\boldsymbol{u}^n \cdot \normalF}  \seminorm{\boldsymbol{H}^{m+1}(\TF)}{\boldsymbol{u}^n}^2.
    \end{aligned}
  \end{equation}
\end{lemma}
\begin{proof}
  Noticing that $\jump{\boldsymbol{u}^n} = 0$ by the regularity assumptions and recalling the definition \eqref{eq:ch} of $c_h$, the first term in the expression \eqref{eq:consistency:error:convective} of $\convTerm{n} (\boldsymbol{e}_h^n)$ is equal to $c_h(\boldsymbol{u}^n, \boldsymbol{u}^n, \boldsymbol{e}_h^n)$.
  Accounting for this fact,
  adding and subtracting $c_h(\boldsymbol{u}^n,\projRT{k} \boldsymbol{u}^n, \boldsymbol{e}_h^n)$,
  and using \eqref{eq:ch:anti-symmetry} to write $c_h(\boldsymbol{u}^n, \boldsymbol{u}^n-\projRT{k}\boldsymbol{u}^n , \boldsymbol{e}_h^n )= - c_h(\boldsymbol{u}^n, \boldsymbol{e}_h^n , \boldsymbol{u}^n-\projRT{k}\boldsymbol{u}^n)$ after recalling that $\boldsymbol{\nabla } \cdot \boldsymbol{u}^n = 0$ by \eqref{eq:continuous.problem.mass},
  we split $\convTerm{n} (\boldsymbol{e}_h^n)$ as follows:
  \begin{equation}\label{eq:estimate:C}
    \begin{aligned}
      \convTerm{n}(\boldsymbol{e}_h^n)
      &=
      - c_h(\boldsymbol{u}^n, \boldsymbol{e}_h^n , \boldsymbol{u}^n-\projRT{k}\boldsymbol{u}^n)
      - c_h(\boldsymbol{e}_h^n, \projRT{k}\boldsymbol{u}^n ,\boldsymbol{e}_h^n )
      + c_h(\boldsymbol{u}^n-\projRT{k}\boldsymbol{u}^n, \projRT{k}\boldsymbol{u}^n,\boldsymbol{e}_h^n ) \\
      &\eqqcolon \convTerm{n}_1 + \convTerm{n}_2 + \convTerm{n}_3.
    \end{aligned}
  \end{equation}
  \smallskip
  \underline{\textbf{Estimate of $\convTerm{n}_1$}.}
  Recalling \eqref{eq:ch}, we write $\convTerm{n}_1 = \sum_{T\in\Th} \convTerm{n}_{1,T} + \sum_{F \in \Fhi} \convTerm{n}_{1,F}$
  with $\convTerm{n}_{1,T} \coloneqq \int_T (\boldsymbol{u}^n \cdot \boldsymbol{\nabla}) \boldsymbol{e}_h^n \cdot (\boldsymbol{u}^n-\projRT{k}\boldsymbol{u}^n)$ and
  $\convTerm{n}_{1,F} \coloneqq \int_F (\boldsymbol{u}^n \cdot \normalF) \jump{\boldsymbol{e}_h^n} \cdot \avg{\boldsymbol{u}^n-\projRT{k}\boldsymbol{u}^n}$.

  For $\convTerm{n}_{1,T}$, following the argument in \cite[Lemma 5.1]{Han.Hou:21}, after defining $\bar{\boldsymbol{u}}_T^n \coloneqq \frac{1}{|T|} \int_T \boldsymbol{u}^n$, we obtain
  \[
  \begin{aligned}
    \convTerm{n}_{1,T} &= \int_T ((\boldsymbol{u}^n - \bar{\boldsymbol{u}}_T^n )  \cdot \boldsymbol{\nabla}) \boldsymbol{e}_h^n \cdot (\boldsymbol{u}^n-\projRTEle{k}\boldsymbol{u}^n )
    \\
    &\leq
    \norm{\boldsymbol{L}^\infty(T)}{\boldsymbol{u}^n - \bar{\boldsymbol{u}}_T^n}
    \norm{\boldsymbol{L}^2(T)}{\boldsymbol{\nabla }\boldsymbol{e}_h^n}
    \norm{\boldsymbol{L}^2(T)}{\boldsymbol{u}^n-\projRTEle{k}\boldsymbol{u}^n} \\
    &\lesssim
    h_T^{m+1} \norm{\boldsymbol{L}^\infty(T)}{\boldsymbol{\nabla} \boldsymbol{u}^n}
    \seminorm{\boldsymbol{H}^{m+1}(T)}{\boldsymbol{u}^n} \norm{\boldsymbol{L}^2(T)}{\boldsymbol{e}_h^n},
  \end{aligned}
  \]
  where, in the last inequality, we have used the approximation property of $\bar{\boldsymbol{u}}_T^n$ (see, e.g., \cite[Eq.~(5.2)]{Han.Hou:21}),
  an inverse estimate on $\norm{\boldsymbol{L}^2(T)}{\boldsymbol{\nabla }\boldsymbol{e}_h^n}$, and the approximation property \eqref{eq:approx:RT:ele} of $\projRTEle{k}$ with $(j,q,s) = (0,2,m+1)$.

  Moving to $\convTerm{n}_{1,F}$, we use a $(\infty,2,2)$-H\"older inequality and \eqref{eq:gamma-u-bound} to write
  \begin{equation}\label{eq:C1F}
    \begin{aligned}
      \convTerm{n}_{1,F} &\leq \norm{L^\infty(F)}{\boldsymbol{u}^n \cdot \normalF}
      \norm{\boldsymbol{L}^{2}(F)}{\avg{\boldsymbol{u}^n-\projRT{k}\boldsymbol{u}^n }}
      \norm{\boldsymbol{L}^2(F)}{\jump{\boldsymbol{e}_h^n}} \\
      &\lesssim
      h_F^{m+\frac12}
      \norm{L^\infty(F)}{\boldsymbol{u}^n \cdot \normalF}^{\frac12}  \seminorm{\boldsymbol{H}^{m+1}(\TF)}{\boldsymbol{u}^n} \left(
      \gamma_F(\boldsymbol{u}_h^n) \norm{\boldsymbol{L}^2(F)}{\jump{\boldsymbol{e}_h^n}}^2
      \right)^{\frac{1}{2}}.
    \end{aligned}
  \end{equation}
  Gathering the above estimate, we use first Cauchy–Schwarz inequalities on the sums over faces and elements, then we apply a generalized $(2, 2)$-Young's inequality to get, for any $\delta >0$ the following bound for $\convTerm{n}_1$:
  \[
  \convTerm{n}_1
  \le
  \delta \left[ \norm{\boldsymbol{L}^2(\Omega)}{\boldsymbol{e}_h^n}^2 
  + \seminorm{\boldsymbol{u}_h^n}{\boldsymbol{e}_h^n}^2 \right]
  + c(\delta) \convErr.
  \]
  \\
  \underline{\textbf{Estimate of $\convTerm{n}_2$}.} After splitting $\convTerm{n}_2$ into the sum of element and face contributions, we apply a $(2,\infty,2)$-H\"older inequality, the $\boldsymbol{W}^{1,\infty}$-continuity and the trace approximation property \eqref{eq:approx:RT:face} of $\projRTEle{k}$ with $(q,s)=(\infty,1)$, and the discrete trace inequality to write
  \[
  \begin{alignedat}{2}
    \convTerm{n}_{2,T}
    &\leq \norm{\boldsymbol{L}^\infty(T)}{\boldsymbol{\nabla} \projRTEle{k} \boldsymbol{u}^n}\norm{\boldsymbol{L}^2(T)}{\boldsymbol{e}_h^n}^2
    \lesssim \norm{\boldsymbol{L}^\infty(T)}{\boldsymbol{\nabla} \boldsymbol{u}^n}\norm{\boldsymbol{L}^2(T)}{\boldsymbol{e}_h^n}^2
    &\qquad& \forall T \in \Th,
    \\
    \convTerm{n}_{2,F} &\lesssim \norm{\boldsymbol{L}^{\infty}(F)}{\jump{\boldsymbol{u}^n-\projRT{k}\boldsymbol{u}^n }}
    h_F^{-1}\norm{\boldsymbol{L}^2(\TF)}{\boldsymbol{e}_h^n}^2 \lesssim
    \norm{\boldsymbol{L}^\infty(\omega_F)}{\boldsymbol{\nabla} \boldsymbol{u}^n}
    \norm{\boldsymbol{L}^2(\TF)}{\boldsymbol{e}_h^n}^2
    &\qquad& \forall F \in \Fh.
  \end{alignedat}
  \]
  Therefore, we get $\convTerm{n}_2 \le C_{\boldsymbol{u}}^\star \norm{\boldsymbol{L}^2(\Omega)}{\boldsymbol{e}_h^n}^2$.
  \medskip\\
  \underline{\textbf{Estimate of $\convTerm{n}_3$}.} We split the last term as $\convTerm{n}_3 = \sum_{T \in \Th} \convTerm{n}_{3,T}+ \sum_{F\in \Fh} \convTerm{n}_{3,F}$ and repeat here what we have done for $\convTerm{n}_{2}$.
  In particular, to estimate the term $\convTerm{n}_{3,T}$, we apply a $(2,\infty,2)$-H\"older inequality together with the $\boldsymbol{W}^{1,\infty}$-continuity of $\projRTEle{k}$ and the approximation property \eqref{eq:approx:RT:ele} with $(j,q,s)=(0,2,m+1)$. This yields, for all $T \in\ \Th$,
  \begin{equation}\label{eq:estimate:C3.T}
    \convTerm{n}_{3,T} = \int_T ((\boldsymbol{u}^n -\projRTEle{k}\boldsymbol{u}^n ) \cdot \boldsymbol{\nabla}) \projRTEle{k}\boldsymbol{u}^n \cdot \boldsymbol{e}_h^n \lesssim h_T^{m+1} \norm{\boldsymbol{L}^\infty(T)}{\boldsymbol{\nabla} \boldsymbol{u}^n}
    \seminorm{\boldsymbol{H}^{m+1}(T)}{\boldsymbol{u}^n} \norm{\boldsymbol{L}^2(T)}{\boldsymbol{e}_h^n}.
  \end{equation}
  For $\convTerm{n}_{3,F}$, we insert $\jump{\boldsymbol{u}^n} = \boldsymbol{0}$ in the integral and we use the bound $\norm{\boldsymbol{L}^{\infty}(F)}{\normalF} \leq 1$ to write, for all $F \in \Fh$,
  \begin{equation}\label{eq:estimate:C3.F}
    \begin{aligned}
      \convTerm{n}_{3,F}
      &\leq
      \norm{\boldsymbol{L}^\infty(F)}{\boldsymbol{u}^n -\projRT{k}\boldsymbol{u}^n}
      \norm{\boldsymbol{L}^2(F)}{\jump{ \boldsymbol{u}^n -\projRT{k}\boldsymbol{u}^n}}
      \norm{\boldsymbol{L}^2(F)}{ \avg{\boldsymbol{e}_h^n} }\\
      &\lesssim h_F^{m+1}
      \norm{\boldsymbol{L}^\infty(\omega_F)}{\boldsymbol{\nabla} \boldsymbol{u}^n}
      \seminorm{\boldsymbol{H}^{m+1}(\TF)}{\boldsymbol{u}^n} \norm{\boldsymbol{L}^2(\TF)}{\boldsymbol{e}_h^n},
    \end{aligned}
  \end{equation}
  where the final inequality follows from applying the approximation property \eqref{eq:approx:RT:face} twice, first with $(q,s)=(\infty,1)$ and then with $(q,s)=(2,m+1)$, along with discrete trace inequality applied to $\norm{\boldsymbol{L}^2(F)}{ \avg{\boldsymbol{e}_h^n}}$.

  Plugging \eqref{eq:estimate:C3.T} and \eqref{eq:estimate:C3.F} into the expression of $\convTerm{n}_3$ leads to the following bound, valid for any $\delta > 0$:
  \[
  \convTerm{n}_{3}
  \le \delta \norm{\boldsymbol{L}^2(\Omega)}{\boldsymbol{e}_h^n}^2
  + c(\delta) \convErr.
  \]
  \\
  \underline{\textbf{Conclusion}.}
  By plugging the estimates for $\convTerm{n}_1$, $\convTerm{n}_2$ and $\convTerm{n}_3$ into \eqref{eq:estimate:C} we get the desired conclusion.
\end{proof}

\subsection{Upwind stabilization error}

\begin{lemma}[Estimate of the upwind error]\label{lem:upw:stab:error}
  Recalling the definition~\eqref{eq:consistency:error:upwind} of $\upwTerm{n}(\cdot)$, under Assumption \ref{reg:ass:u}, for $n=1,\ldots,N$ and any $\delta >0$, it holds
  \begin{equation}\label{eq:upwTerm:estimate}
    \upwTerm{n}(\boldsymbol{e}_h^n)
    \le
    C_{\boldsymbol{u}}^\star \norm{\boldsymbol{L}^2(\Omega)}{\boldsymbol{e}_h^n}^2
    +\delta \left[
      \norm{\boldsymbol{L}^2(\Omega)}{\boldsymbol{e}_h^n}^2 
      + \seminorm{\boldsymbol{u}_h^n}{\boldsymbol{e}_h^n}^2   \right]
    + c(\delta) \upwErr,
  \end{equation}
  where the constant $C_{\boldsymbol{u}}^\star$ was introduced in \eqref{eq:Custar}, and
  \begin{equation}\label{eq:upwErr}
    \begin{aligned}
      \upwErr
      &\coloneqq
      \sum_{F\in\Fh} h_F^{2m+1}
      \max(C_F,h_F \norm{\boldsymbol{L}^\infty(\omega_F)}{\boldsymbol{\nabla} \boldsymbol{u}^n}^2)
      \seminorm{\boldsymbol{H}^{m+1}(\TF)}{\boldsymbol{u}^n}^2
      \\
      &\quad
      + \sum_{F\in\Fh} h_F^{2m+1} \norm{L^{\infty}(F)}{\boldsymbol{u}^n \cdot \normalF} \seminorm{\boldsymbol{H}^{m+1}(\TF)}{\boldsymbol{u}^n}^2.
    \end{aligned}
  \end{equation}
\end{lemma}

\begin{proof}
  Recalling the definitions \eqref{eq:consistency:error:upwind} of $\upwTerm{n}$ and \eqref{eq:jh} of $j_h$, and noticing that $\jump{\boldsymbol{u}^n} = \boldsymbol{0}$, we write $\upwTerm{n}(\boldsymbol{e}_h^n) = \sum_{F \in \Fhi} \upwTerm{n}_F$ with $\upwTerm{n}_F \coloneq \gamma_F(\boldsymbol{u}_h^n) \int_F \jump{ \boldsymbol{u}^n - \projRT{k} \boldsymbol{u}^n } \cdot \jump{\boldsymbol{e}_h^n }$.
  Let a face $F \in \Fhi$ be fixed.
  We distinguish between the cases $\gamma_F(\boldsymbol{u}_h^n) = \norm{L^\infty(F)}{\boldsymbol{u}_h^n \cdot \normalF}$ and $\gamma_F(\boldsymbol{u}_h^n) = C_F$.
  \medskip
  \\
  \underline{\textbf{Case $\gamma_F(\boldsymbol{u}_h^n) = \norm{L^\infty(F)}{\boldsymbol{u}_h^n \cdot \normalF}$}.}
  After decomposing
  $\boldsymbol{u}_h^n = \boldsymbol{e}_h^n + ( \projRT{k}\boldsymbol{u}^n-\boldsymbol{u}^n ) + \boldsymbol{u}^n$ and applying a triangle inequality, we bound $\upwTerm{n}_F$ as follows: $\upwTerm{n}_F \leq \upwTerm{n}_{1,F} + \upwTerm{n}_{2,F}+\upwTerm{n}_{3,F}$, where
  \begin{equation*}
    \begin{aligned}
      \upwTerm{n}_{1,F} &\coloneqq  \norm{L^{\infty}(F)}{\boldsymbol{e}_h^n \cdot \normalF} \int_F \jump{ \boldsymbol{u}^n - \projRT{k} \boldsymbol{u}^n } \cdot \jump{\boldsymbol{e}_h^n }, \\
      \upwTerm{n}_{2,F} &\coloneqq \norm{L^{\infty}(F)}{(\projRT{k}\boldsymbol{u}^n-\boldsymbol{u}^n) \cdot \normalF} \int_F \jump{ \boldsymbol{u}^n - \projRT{k} \boldsymbol{u}^n } \cdot \jump{\boldsymbol{e}_h^n }, \\
      \upwTerm{n}_{3,F} &\coloneqq \norm{L^{\infty}(F)}{\boldsymbol{u}^n \cdot \normalF} \int_F \jump{ \boldsymbol{u}^n - \projRT{k} \boldsymbol{u}^n } \cdot \jump{\boldsymbol{e}_h^n }.
    \end{aligned}
  \end{equation*}
  For the first term,
  we start by noticing that, using an inverse Lebesgue inequality (see for instance \cite[Lemma 1.25]{Di-Pietro.Droniou:20}) followed by a discrete trace inequality,
  \[
  \norm{L^\infty(F)}{\boldsymbol{e}_h^n \cdot \normalF}
  \lesssim |F|^{-\frac12} {\norm{L^2(F)}{\boldsymbol{e}_h^n \cdot \normalF}}
  \lesssim |F|^{-\frac12} h_F^{-\frac12} {\norm{\boldsymbol{L}^2(\TF)}{\boldsymbol{e}_h^n}}.
  \]
  We next write, using a $(2,\infty,2)$-H\"{o}lder inequality,
  \[
  \begin{aligned}
    \int_F \jump{ \boldsymbol{u}^n - \projRT{k} \boldsymbol{u}^n } \cdot \jump{\boldsymbol{e}_h^n }
    &\le |F|^{\frac12} {\norm{\boldsymbol{L}^\infty(F)}{\jump{ \boldsymbol{u}^n - \projRT{k} \boldsymbol{u}^n }}} {\norm{\boldsymbol{L}^2(F)}{\jump{\boldsymbol{e}_h^n} }}
    \\
    &\lesssim |F|^{\frac12} h_F^{\frac12} \norm{\boldsymbol{L}^\infty(\omega_F)}{\boldsymbol{\nabla} \boldsymbol{u}^n} \norm{\boldsymbol{L}^2(\TF)}{\boldsymbol{e}_h^n},
  \end{aligned}
  \]
  where the conclusion follows using a triangle inequality together with the approximation properties \eqref{eq:approx:RT:face} of $\projRTEle{k}$ with $(q,s) = (\infty,1)$ for the second factor and triangle and discrete trace inequalities for the third.
  In conclusion,
  \[
  \upwTerm{n}_{1,F} \le C_{\boldsymbol{u}}^\star \norm{\boldsymbol{L}^2(\Omega)}{\boldsymbol{e}_h^n}^2
  \]
  with $C_{\boldsymbol{u}}^\star$ introduced in \eqref{eq:Custar}.

  In a similar way, using the fact that $\norm{\boldsymbol{L}^\infty(F)}{\normalF} \le 1$ and a Cauchy--Schwarz inequality, we obtain, for the second term,
  \begin{equation*}
    \begin{aligned}
      \upwTerm{n}_{2,F}
      & \leq \norm{\boldsymbol{L}^{\infty}(F)}{\projRT{k}\boldsymbol{u}^n-\boldsymbol{u}^n} \norm{\boldsymbol{L}^2(F)}{\jump{\projRT{k}\boldsymbol{u}^n-\boldsymbol{u}^n}} \norm{\boldsymbol{L}^2(F)}{\jump{\boldsymbol{e}_h^n}} \\
      & \lesssim h_F^{m+1} \norm{\boldsymbol{L}^{\infty}(\omega_F)}{\boldsymbol{\nabla} \boldsymbol{u}^n} \seminorm{\boldsymbol{H}^{m+1}(\TF)}{\boldsymbol{u}^n}\norm{\boldsymbol{L}^2(\TF)}{\boldsymbol{e}_h^n},
    \end{aligned}
  \end{equation*}
  where the conclusion follows using \eqref{eq:approx:RT:face} with $(q,s) = (\infty,1)$ for the first factor and with $(q,s) = (2,m+1)$ for the second and triangle plus trace inequalities for the third factor.

  For the last term, applying the same argument that leads to \eqref{eq:C1F}, we obtain
  \[
  \upwTerm{n}_{3,F}
  \lesssim
  \norm{L^{\infty}(F)}{\boldsymbol{u}^n \cdot \normalF}^{\frac12} h_F^{m+\frac12}\seminorm{\boldsymbol{H}^{m+1}(\TF)}{\boldsymbol{u}^n}
  \left( \gamma_F(\boldsymbol{u}_h^n) \norm{\boldsymbol{L}^2(F)}{\jump{\boldsymbol{e}_h^n}}^2\right)^{\frac{1}{2}} \, .
  \]
  \noindent
  \underline{\textbf{Case $\gamma_F(\boldsymbol{u}_h^n) = C_F$}.}
  For this case, we simply apply the Cauchy--Schwarz inequality to $\upwTerm{n}_{F}$ followed by the approximation property \eqref{eq:approx:RT:face} to obtain
  \begin{equation*}
    \upwTerm{n}_{F} \leq C_F^{\frac12} \norm{\boldsymbol{L}^2(F)}{\jump{ \boldsymbol{u}^n - \projRT{k} \boldsymbol{u}^n }} C_F^{\frac12}\norm{\boldsymbol{L}^2(F)}{\jump{\boldsymbol{e}_h^n }}
    \lesssim C_F^{\frac12} h_F^{m+\frac12} \seminorm{\boldsymbol{H}^{m+1}(\TF)}{\boldsymbol{u}^n} C_F^{\frac12}\norm{\boldsymbol{L}^2(F)}{\jump{\boldsymbol{e}_h^n }}.
  \end{equation*}
  \underline{\textbf{Conclusion.}} The conclusion follows easily gathering the above estimates and using appropriate Young-type inequalities on the sums over faces.
  Note that the first addendum in the definition of $\upwErr$, equation \eqref{eq:upwErr}, gathers terms coming both from the bound for $\upwTerm{n}_{2,F}$ and the bound for case $\gamma_F(\boldsymbol{u}_h^n) = C_F$.
\end{proof}

\subsection{Error estimate for $r \ge 2$}

We preliminarily note that, in the case $r \ge 2$, bound \eqref{eq:ah:coercivity} immediately implies
\begin{equation}\label{eq:ah:coercivity:bis}
a_h (\boldsymbol{w}_h,\boldsymbol{e}_h) - a_h(\boldsymbol{z}_h,\boldsymbol{e}_h) 
\gtrsim \nu  \norm{1,r,h}{\boldsymbol{e}_h}^{r} \, . 
\end{equation}
\begin{theorem}[Velocity error estimate $r \geq 2 $]\label{thm:error.estimate}
  Under the regularity Assumption \ref{reg:ass:u}, and further assuming $\tstep < \frac{1}{5 C_{\boldsymbol{u}}^\star}$, with $C_{\boldsymbol{u}}^\star$ introduced in \eqref{eq:Custar}, for $ r \geq 2 $ it holds,
  \begin{equation}\label{eq:error.estimate}
    \norm{\boldsymbol{L}^2(\Omega)}{\boldsymbol{e}_h^N}^2
    + \sum_{n=1}^N \tstep \left(
    \nu \norm{1,r,h}{\boldsymbol{e}_h^n}^{r}
    + \seminorm{\boldsymbol{u}_h^n}{\boldsymbol{e}_h^n}^2
    \right)
    \lesssim
    \sum_{n=1}^N \tstep \Err,
  \end{equation}
  where, recalling \eqref{eq:timeErr}, \eqref{eq:diffErr}, \eqref{eq:convErr}, and \eqref{eq:upwErr}, the consistency error is defined by
  \[
  \Err \coloneq \timeErr + \diffErr + \convErr + \upwErr,
  \]
  and the hidden constant in \eqref{eq:error.estimate} depends linearly on $\exp\left[ \frac{5C_{\boldsymbol{u}}^\star}{1 - 5 C_{\boldsymbol{u}}^\star \tstep} \tF \right]$.
\end{theorem}
\begin{proof}
  We bound from below the left-hand side of \eqref{eq:error:equation} with $\boldsymbol{v}_h = \boldsymbol{e}_h^n$.
  Using the relation \eqref{eq:relation.unsteady} for the first term, we obtain
  \[
  \int_{\Omega} \frac{\boldsymbol{e}_h^n - \boldsymbol{e}_h^{n-1}}{\tstep} \cdot \boldsymbol{e}_h^n \geq \frac{1}{2 \tstep} \norm{\boldsymbol{L}^2(\Omega)}{\boldsymbol{e}_h^n}^2 - \frac{1}{2 \tstep} \norm{\boldsymbol{L}^2(\Omega)}{\boldsymbol{e}_h^{n-1}}^2.
  \]
  From \eqref{eq:ah:coercivity:bis}, after writing $\alpha = \frac{\alpha}{2} + \frac{\alpha}{2}$ and making the hidden constant in $\gtrsim$ explicit, we get
  \[
  a_h(\boldsymbol{u}_h^n,\boldsymbol{e}_h^n) - a_h(\projRT{k}\boldsymbol{u}^n,\boldsymbol{e}_h^n) \geq
  \frac12 \left( a_h(\boldsymbol{u}_h^n,\boldsymbol{e}_h^n) - a_h(\projRT{k}\boldsymbol{u}^n,\boldsymbol{e}_h^n) \right) + \frac{C_{a}}{2} \nu \norm{1,r,h}{\boldsymbol{e}_h^n}^{r}.
  \]
  From \eqref{eq:ch:non-dissipativity} and the definition~\eqref{eq:convective.seminorm} of the seminorm $\seminorm{\boldsymbol{u}_h^n}{\cdot}$, we deduce that the convective term vanishes and that $j_h(\boldsymbol{u}_h^n;\boldsymbol{e}_h^n,\boldsymbol{e}_h^n) = \seminorm{\boldsymbol{u}_h^n}{\boldsymbol{e}_h^n}^2$.
  Using the above facts, we obtain the following lower bound for the right-hand side of \eqref{eq:error:equation} with $\boldsymbol{v}_h = \boldsymbol{e}_h^n$:
  \begin{multline}\label{eq:error:lower:bound}
    \timeTerm{n} (\boldsymbol{e}_h^n)
    + \diffTerm{n} (\boldsymbol{e}_h^n)
    + \convTerm{n} (\boldsymbol{e}_h^n)
    + \upwTerm{n} (\boldsymbol{e}_h^n)
    \geq \frac{1}{2 \tstep} \norm{\boldsymbol{L}^2(\Omega)}{\boldsymbol{e}_h^n}^2 - \frac{1}{2 \tstep} \norm{\boldsymbol{L}^2(\Omega)}{\boldsymbol{e}_h^{n-1}}^2
    \\
    + \frac12 \left( a_h(\boldsymbol{u}_h^n,\boldsymbol{e}_h^n) - a_h(\projRT{k}\boldsymbol{u}^n,\boldsymbol{e}_h^n) \right) + \frac12 C_a \nu \norm{1,r,h}{\boldsymbol{e}_h^n}^{r} + \seminorm{\boldsymbol{u}_h^n}{\boldsymbol{e}_h^n}^2.
  \end{multline}
  Plugging our estimates \eqref{eq:timeTerm:estimate} of $\timeTerm{n}(\boldsymbol{e}_h^n)$,
  \eqref{eq:diffTerm:estimate} of $\diffTerm{n}(\boldsymbol{e}_h^n)$,
  \eqref{eq:convTerm:estimate} of $\convTerm{n}(\boldsymbol{e}_h^n)$,
  and \eqref{eq:upwTerm:estimate} of $\upwTerm{n}(\boldsymbol{e}_h^n)$ into \eqref{eq:error:lower:bound} and rearranging terms, we obtain
  \begin{multline*}
    \left( \frac{1}{2 \tstep} - (3 \delta + 2C_{\boldsymbol{u}}^\star)  \right) \norm{\boldsymbol{L}^2(\Omega)}{\boldsymbol{e}_h^n}^2
    +\left(\frac12 C_a - \delta \right) \nu \norm{1,r,h}{\boldsymbol{e}_h^n}^{r}
    +\left( 1 - 3 \delta \right) \seminorm{\boldsymbol{u}_h^n}{\boldsymbol{e}_h^n}^2 \\
    +\left( \frac12 - \delta \right) \left( a_h(\boldsymbol{u}_h^n,\boldsymbol{e}_h^n) - a_h(\projRT{k}\boldsymbol{u}^n,\boldsymbol{e}_h^n) \right)
    \leq 
     \frac{1}{2 \tstep} \norm{\boldsymbol{L}^2(\Omega)}{\boldsymbol{e}_h^{n-1}}^2
    + 4 c(\delta) \Err.
  \end{multline*}
  Taking $\delta < \frac12$, noticing that $ a_h(\boldsymbol{u}_h^n,\boldsymbol{e}_h^n) - a_h(\projRT{k}\boldsymbol{u}^n,\boldsymbol{e}_h^n) \geq C_a \nu \norm{1,r,h}{\boldsymbol{e}_h^n}^{r} \geq 0$, and setting $\tilde{C} \coloneqq \min \left\{ \frac12 C_a , 1 \right\}$, we find, after multiplying by $2 \tstep$,
  \begin{multline*}
    \left(1 - \tstep(6 \delta + 4C_{\boldsymbol{u}}^\star)  \right) \norm{\boldsymbol{L}^2(\Omega)}{\boldsymbol{e}_h^n}^2
    +\left( 2 \tilde{C} - 6 \delta \right) \tstep \left( \nu \norm{1,r,h}{\boldsymbol{e}_h^n}^{r}
   + \seminorm{\boldsymbol{u}_h^n}{\boldsymbol{e}_h^n}^2 \right)
    \leq 
    \norm{\boldsymbol{L}^2(\Omega)}{\boldsymbol{e}_h^{n-1}}^2
    + 8 c(\delta) \tstep \Err.
  \end{multline*}
  We simplify the inequality by choosing $\delta = \hat{\delta} \coloneqq \min \left\{\frac12, \frac{C_{\boldsymbol{u}}^\star}{6}, \frac{\tilde{C}}{6}\right\}$, and setting $\hat{C} \coloneqq c(\hat{\delta})$. This yields 
  \begin{equation*}
    \left(1 - 5 C_{\boldsymbol{u}}^\star \tstep \right) \norm{\boldsymbol{L}^2(\Omega)}{\boldsymbol{e}_h^n}^2
    +\tilde{C} \tstep \left( \nu \norm{1,r,h}{\boldsymbol{e}_h^n}^{r}
   + \seminorm{\boldsymbol{u}_h^n}{\boldsymbol{e}_h^n}^2 \right)
    \leq 
    \norm{\boldsymbol{L}^2(\Omega)}{\boldsymbol{e}_h^{n-1}}^2
    + 8 \hat{C} \tstep \Err.
  \end{equation*}
  Under the assumption $\tstep < \frac{1}{5 C_{\boldsymbol{u}}^\star }$, we can divide both side by $1 - 5 C_{\boldsymbol{u}}^\star \tstep  > 0$ to get
  \begin{equation*}
    \norm{\boldsymbol{L}^2(\Omega)}{\boldsymbol{e}_h^n}^2
    +\frac{\tilde{C} }{1 - 5 C_{\boldsymbol{u}}^\star \tstep} \tstep \left( \nu \norm{1,r,h}{\boldsymbol{e}_h^n}^{r}
   + \seminorm{\boldsymbol{u}_h^n}{\boldsymbol{e}_h^n}^2 \right)
    \leq 
    \frac{1}{1 - 5 C_{\boldsymbol{u}}^\star \tstep}\norm{\boldsymbol{L}^2(\Omega)}{\boldsymbol{e}_h^{n-1}}^2
    +\frac{ 8 \hat{C}}{1 - 5 C_{\boldsymbol{u}}^\star \tstep} \tstep \Err.
  \end{equation*}
  Summing the above inequality for $n=1,\ldots,N$ after subtracting $\norm{\boldsymbol{L}^2(\Omega)}{\boldsymbol{e}_h^{n-1}}^2$ from both sides, telescoping the appropriate terms, and noticing that $\boldsymbol{e}_h^0 = \boldsymbol{0}$ by choice of the discrete initial condition on the velocity, we obtain
  \begin{multline*}
    \norm{\boldsymbol{L}^2(\Omega)}{\boldsymbol{e}_h^N}^2
    +\frac{\tilde{C}}{1 - 5 C_{\boldsymbol{u}}^\star \tstep}
    \sum_{n=1}^N \tstep \left( \nu \norm{1,r,h}{\boldsymbol{e}_h^n}^{r}
   + \seminorm{\boldsymbol{u}_h^n}{\boldsymbol{e}_h^n}^2 \right) \\
    \leq 
    \frac{5 C_{\boldsymbol{u}}^\star \tstep }{1 - 5 C_{\boldsymbol{u}}^\star \tstep}
    \sum_{n=1}^N \norm{\boldsymbol{L}^2(\Omega)}{\boldsymbol{e}_h^{n-1}}^2
    +\frac{ 8 \hat{C}}{1 - 5 C_{\boldsymbol{u}}^\star \tstep} \sum_{n=1}^N \tstep \Err.
  \end{multline*}
  Using a generalized version of the discrete Gr\"onwall inequality (see \cite[Lemma 5.1]{Heywood.Rannacher:90}), we arrive at
  \begin{equation*}
    \norm{\boldsymbol{L}^2(\Omega)}{\boldsymbol{e}_h^N}^2
    +\frac{\tilde{C}}{1 - 5 C_{\boldsymbol{u}}^\star \tstep}
    \sum_{n=1}^N \tstep \left( \nu \norm{1,r,h}{\boldsymbol{e}_h^n}^{r}
   + \seminorm{\boldsymbol{u}_h^n}{\boldsymbol{e}_h^n}^2 \right)
    \leq 
    \exp\left[ \frac{5C_{\boldsymbol{u}}^\star}{1 - 5 C_{\boldsymbol{u}}^\star \tstep} \tF \right]
    \frac{ 8 \hat{C}}{1 - 5 C_{\boldsymbol{u}}^\star \tstep} \sum_{n=1}^N \tstep \Err.
  \end{equation*}
  Finally, taking the minimum of the constants on the left-hand side, we infer that
  \begin{multline*}
  \min \left\{ 1, \frac{\tilde{C}}{1 - 5 C_{\boldsymbol{u}}^\star \tstep}  \right\}
  \left[
  \norm{\boldsymbol{L}^2(\Omega)}{\boldsymbol{e}_h^N}^2 + \sum_{n=1}^N \tstep \left( \nu \norm{1,r,h}{\boldsymbol{e}_h^n}^{r}
   + \seminorm{\boldsymbol{u}_h^n}{\boldsymbol{e}_h^n}^2 \right)
  \right] \\
    \leq 
    \exp\left[ \frac{5C_{\boldsymbol{u}}^\star}{1 - 5 C_{\boldsymbol{u}}^\star \tstep} \tF \right]
    \frac{ 8 \hat{C}}{1 - 5 C_{\boldsymbol{u}}^\star \tstep} \sum_{n=1}^N \tstep \Err.
  \end{multline*}
  Dividing by $\min \left\{ 1, \frac{\tilde{C}}{1 - 5 C_{\boldsymbol{u}}^\star \tstep}  \right\} > 0 $, we conclude that
  \begin{multline*}
  \norm{\boldsymbol{L}^2(\Omega)}{\boldsymbol{e}_h^N}^2 + \sum_{n=1}^N \tstep \left( \nu \norm{1,r,h}{\boldsymbol{e}_h^n}^{r}
   + \seminorm{\boldsymbol{u}_h^n}{\boldsymbol{e}_h^n}^2 \right) \\
    \leq 
    \exp\left[ \frac{5C_{\boldsymbol{u}}^\star}{1 - 5 C_{\boldsymbol{u}}^\star \tstep} \tF \right]
    \frac{ 8 \hat{C}}{\left(\min \left\{ 1, \frac{\tilde{C}}{1 - 5 C_{\boldsymbol{u}}^\star \tstep} \right\}\right)(1 - 5 C_{\boldsymbol{u}}^\star \tstep)} \sum_{n=1}^N \tstep \Err. \qedhere
  \end{multline*}
\end{proof}
\begin{remark}[Convergence rates]\label{rem:conv-rates}
  Simplified convergence rates for the terms in the right-hand side of \eqref{eq:error.estimate} and the various contributions to the error are reported in Table~\ref{tab:convergence.rates}.
  Specifically, we consider the situation of diffusion- or convection-dominated flows, in which, respectively, $(\Thd,\Fhd) = (\Th,\Fh)$ and $(\Thc,\Fhc) = (\Th,\Fh)$ for all $n \in \{1, \ldots, N\}$.
  In the linear case $r = 2$, consistently with the findings of \cite{Han.Hou:21}, for all the square roots of the error measures we obtain convergence in $\tstep + h^m$ in the diffusion-dominated regime and $\tstep + h^{m+\frac12}$ in the convection-dominated regime.
\end{remark}
\begin{remark}[The case $r>2$]
  In the case $r>2$, we notice that the diffusive error term $\nu \| \cdot \|_{1,r,h}$ carries a different exponent when compared to the other error terms in the left hand side of \eqref{eq:error.estimate}.
  In addition, as becomes clear inspecting Table~\ref{tab:convergence.rates}, for $r>2$ the presence of an approximation of the stress $\sigma$ in the error analysis leads to sub-optimal estimates in diffusion dominated cases whenever $m > r'/(2-r')$
  (at least when compared to the expected bounds for conforming Finite Elements). This is a known drawback of non-conforming methods in this type of problems \cite{Beirao-da-Veiga.Di-Pietro.ea:24}; see \cite{Di-Pietro.Droniou.ea:21} for techniques to derive improved error estimates.
\end{remark}
\begin{remark}[Classical error]\label{rem:class-error}
  Combining the discrete error estimate in Theorem \ref{thm:error.estimate} with interpolation results for $\BDM^k$ elements it is trivial to obtain analogous error bounds for the error
  $$
  \norm{\boldsymbol{L}^2(\Omega)}{\boldsymbol{u}_h^N - \boldsymbol{u}^N}^2
  + \sum_{n=1}^N \tstep \left(
  \nu \norm{1,r,h}{\boldsymbol{u}_h^n - \boldsymbol{u}^n}^{r}
  + \seminorm{\boldsymbol{u}_h^n}{\boldsymbol{u}_h^n - \boldsymbol{u}^n}^2
  \right)
  $$
  with the same convergence rates.
\end{remark}

\subsection{Error estimate for $1< r < 2$}

Whenever $1<r<2$, bound \eqref{eq:ah:coercivity} does not immediately imply \eqref{eq:ah:coercivity:bis}. Indeed, the critical point is that, in order to obtain \eqref{eq:ah:coercivity:bis} one would need a uniform $L^\infty$-in-time type bound
\begin{equation}\label{eq:not-true}
    \norm{1,r,h}{\boldsymbol{u}_h^n} \lesssim 1  \qquad \forall n=1,2,..,N \, ,
\end{equation}
which does not seem to be guaranteed in general (compare with the stability bound in Lemma \ref{eq:stability}). 
Therefore, we will resort to prove control of a specific (weaker) diffusive jump norm and show that we are nevertheless able to derive equivalent error estimates.

\begin{lemma}[Jump control of $a_h$ for $1<r < 2$]\label{lem:mono:ah:rs2}
  If $1<r < 2$, then, for all $n=1,2,..,N$, it holds
  \[
    a_h (\boldsymbol{u}_h^n,\boldsymbol{e}_h^n) - a_h(\projRT{k}\boldsymbol{u}^n,\boldsymbol{e}_h^n)
    \gtrsim
    \nu
    \seminorm{\star}{\boldsymbol{e}_h^n}^2 . \label{eq:ah:coercivity:rs2}
  \]
  where 
  \begin{equation}\label{eq:norm:min}
  \seminorm{\star}{\boldsymbol{e}_h^n}^2 \coloneqq \sum_{F \in \Fh} \int_F \min \{ h_F^{1-r} | \jump{\boldsymbol{e}_h^n} |^{r} , h_F^{-1} | \jump{\boldsymbol{e}_h^n} |^{2} \}
  \end{equation}
\end{lemma}
\begin{proof}
From the definition of $a_h$ and by \eqref{eq:sigma:monotonicity}, after observing that $ \left( \boldsymbol{\sigma}(\Gh \boldsymbol{u}_h^n) - \boldsymbol{\sigma}(\Gh \projRT{k} \boldsymbol{u}^n)  \right) : \Gh \boldsymbol{e}_h^n \geq 0$, we obtain
\begin{equation}\label{eq:ah:mono:initial}
  \begin{aligned}
    a_h (\boldsymbol{u}_h^n,\boldsymbol{e}_h^n) - a_h(\projRT{k}\boldsymbol{u}^n,\boldsymbol{e}_h^n)
    \gtrsim 
    \sum_{F \in \Fh} h_F^{1-r} \int_{F} \left( |\jump{\boldsymbol{u}_h^n}|^{r} + |\jump{\projRT{k} \boldsymbol{u}^n}|^{r}  \right)^{\frac{r-2}{r}} |\jump{\boldsymbol{e}_h^n}|^{2}.
  \end{aligned}
\end{equation}
Recalling that $\Real^+ \ni x \mapsto x^{r-2} \in \Real$ is decreasing, we infer that
\begin{equation*}
  \begin{aligned}
    \left( |\jump{\boldsymbol{u}_h^n}|^{r} + |\jump{\projRT{k} \boldsymbol{u}^n}|^{r}  \right)^{\frac{r-2}{r}}
    &\gtrsim
    \left( |\jump{\boldsymbol{u}_h^n}| + |\jump{\projRT{k} \boldsymbol{u}^n}|  \right)^{r-2}
    \gtrsim
    \left( |\jump{\boldsymbol{e}_h^n}| + |\jump{\projRT{k} \boldsymbol{u}^n}|  \right)^{r-2} \\
    &\gtrsim
    \left( \max \{ |\jump{\boldsymbol{e}_h^n}|, |\jump{\projRT{k} \boldsymbol{u}^n}| \}  \right)^{r-2}
    \\
    &=
    \min \{ |\jump{\boldsymbol{e}_h^n}|^{r-2}, |\jump{\projRT{k} \boldsymbol{u}^n - \boldsymbol{u}^n} |^{r-2} \} \\
    &\gtrsim
    \min \{ |\jump{\boldsymbol{e}_h^n}|^{r-2}, h_F^{r-2} \},
  \end{aligned}
\end{equation*}
where the second inequality is obtained observing that $|\jump{\boldsymbol{u}_h^n}| \overset{\eqref{eq:ehn}}= |\jump{\boldsymbol{e}_h^n} - \jump{\projRT{k} \boldsymbol{u}^n}| \le |\jump{\boldsymbol{e}_h^n}| + |\jump{\projRT{k} \boldsymbol{u}^n}|$, while the equality follows from $\jump{\boldsymbol{u}^n} = \boldsymbol{0}$ and the final inequality from the approximation property \eqref{eq:approx:RT:face} of $\projRT{k}$ with $(q,s) = (\infty,1)$ together with Assumption \ref{reg:ass:u}, which ensures $\norm{\boldsymbol{L}^{\infty}(\Omega)}{\boldsymbol{\nabla} \boldsymbol{u}^n} \lesssim 1$.
Inserting this bound into \eqref{eq:ah:mono:initial} gives the desired result.
\end{proof}

The following lemma is a more specific counterpart of Lemma \ref{lem:diffusive:error}, needed in the case $1 < r < 2$ for the reason explained at the start of this section.

\begin{lemma}[Regime-dependent estimate of the diffusive error for $1 < r <2$]\label{lem:diffusive:error:rs2}
  If $1<r < 2$, then, recalling the definition~\eqref{eq:consistency:error:diffusive} of $\diffTerm{n}(\cdot)$, under Assumption \ref{reg:ass:u}, for $n=1,\ldots,N$ and any $\delta >0 $,  it holds
  \begin{equation}\label{eq:diffTerm:estimate:rs2}
    \diffTerm{n}(\boldsymbol{e}_h^n)
    \le
    \delta\left[
      a_h(\boldsymbol{u}_h^n,\boldsymbol{e}_h^n)
      - a_h(\projRT{k}\boldsymbol{u}^n,\boldsymbol{e}_h^n)
      + \nu \seminorm{\star}{\boldsymbol{e}_h^n}^{2}
      + \seminorm{\boldsymbol{u}_h^n}{\boldsymbol{e}_h^n}^2
      \right]
    + c(\delta) \widehat{\diffErr},
  \end{equation}
  where, recalling the definition \eqref{eq:diffErr} of $\diffErr$,
  \begin{equation}\label{eq:diffErr:rs2}
    \begin{aligned}
      \widehat{\diffErr}
      &\coloneqq \diffErr + \sum_{F \in \Fhd} h_F^{2(m+1)} \nu^{-1} \seminorm{\boldsymbol{H}^{m+1}(\TF)}{\boldsymbol{\sigma}(\boldsymbol{\nabla}\boldsymbol{u}^n)}^2
      \end{aligned}
  \end{equation}
  and $\seminorm{\star}{\boldsymbol{e}_h^n}^{2}$ is defined in \eqref{eq:norm:min}.
\end{lemma}
\begin{proof}
The proof follows exactly the same steps as in Lemma \ref{lem:diffusive:error}. The only difference lies in the replacement of the estimate of $\diffTerm{n}_2$ (cf. \eqref{eq:estimate:D}) in the diffusion-dominated regime. Here we proceed as follows.
For a fixed $F \in \Fhd$ we set $F_1 \coloneqq \{ \boldsymbol{x} \in F : h_F^{1-r} | \jump{\boldsymbol{e}_h^n} |^{r}(\boldsymbol{x}) \leq  h_F^{-1} | \jump{\boldsymbol{e}_h^n} |^{2} (\boldsymbol{x})  \}$ and $F_2 \coloneqq F \setminus F_1$. 
Let's recall that, for any $\delta > 0$, $q \in (1,\infty)$, and $\alpha,\,\beta > 0$ it holds that
\begin{equation}\label{eq:young:diffTerm}
\alpha \beta \leq \delta \nu h_F^{1-q} \beta^q + c(\delta) \nu^{-\frac{q'}{q}} h_F \alpha^{q'},
\end{equation} 
with $\frac{1}{q} + \frac{1}{q'}=1$. 
Applying this inequality on $F_1$ with $\alpha = \left| \avg{\projL{k} \boldsymbol{\sigma}(\boldsymbol{\nabla} \boldsymbol{u}^n) - \boldsymbol{\sigma}(\boldsymbol{\nabla} \boldsymbol{u}^n) } \normalF \right|$, $\beta = \left| \jump{\boldsymbol{e}_h^n} \right|$, and $q=r$, and noticing that $\norm{\boldsymbol{L}^{\infty}(F)}{\normalF} \leq 1$, we obtain
\begin{multline*}
  \sum_{F \in \Fhd} \int_{F_1} \left(
  \avg{\projL{k} \boldsymbol{\sigma}(\boldsymbol{\nabla} \boldsymbol{u}^n) - \boldsymbol{\sigma}(\boldsymbol{\nabla} \boldsymbol{u}^n) } \normalF
  \right) \cdot \jump{\boldsymbol{e}_h^n}
  \\
  \leq
  \delta \nu \sum_{F \in \Fhd} \int_{F_1} h_F^{1-r} \left| \jump{\boldsymbol{e}_h^n} \right|^r
  + c(\delta) \sum_{F \in \Fhd} \int_{F_1} h_F \nu^{-\frac{r'}{r}} \left| \avg{\projL{k} \boldsymbol{\sigma}(\boldsymbol{\nabla} \boldsymbol{u}^n) - \boldsymbol{\sigma}(\boldsymbol{\nabla} \boldsymbol{u}^n) } \right|^{r'}.
\end{multline*}
Using the $L^2$-projection approximation property we further obtain
\begin{multline*}
  \sum_{F \in \Fhd} \int_{F_1} \left(
  \avg{\projL{k} \boldsymbol{\sigma}(\boldsymbol{\nabla} \boldsymbol{u}^n) - \boldsymbol{\sigma}(\boldsymbol{\nabla} \boldsymbol{u}^n) } \normalF
  \right) \cdot \jump{\boldsymbol{e}_h^n}
  \\
  \leq
  \delta \nu \sum_{F \in \Fhd} \int_{F_1} h_F^{1-r} \left| \jump{\boldsymbol{e}_h^n} \right|^r
  + c(\delta) \sum_{F \in \Fhd} \int_{F} h_F^{r'(m+1)} \nu^{-\frac{r'}{r}} \seminorm{\boldsymbol{W}^{m+1,r'}(\TF)}{\boldsymbol{\sigma}(\boldsymbol{\nabla}\boldsymbol{u}^n)}^{r'}.
\end{multline*}
Similarly, applying \eqref{eq:young:diffTerm} with the same $a,b$ and $q=2$ on $F_2$, we get
\begin{multline*}
  \sum_{F \in \Fhd} \int_{F_2} \left(
  \avg{\projL{k} \boldsymbol{\sigma}(\boldsymbol{\nabla} \boldsymbol{u}^n) - \boldsymbol{\sigma}(\boldsymbol{\nabla} \boldsymbol{u}^n) } \normalF
  \right) \cdot \jump{\boldsymbol{e}_h^n}
  \\
  \leq
  \delta \nu \sum_{F \in \Fhd} \int_{F_2} h_F^{-1} \left| \jump{\boldsymbol{e}_h^n} \right|^2
  + c(\delta) \sum_{F \in \Fhd} \int_{F} h_F^{2(m+1)} \nu^{-1} \seminorm{\boldsymbol{H}^{m+1}(\TF)}{\boldsymbol{\sigma}(\boldsymbol{\nabla}\boldsymbol{u}^n)}^{2}.
\end{multline*}
Observing that $\int_F = \int_{F_1} + \int_{F_2}$ and that $ \sum_{F \in \Fhd} \left( \int_{F_1} h_F^{1-r} \left| \jump{\boldsymbol{e}_h^n} \right|^r + \int_{F_2} h_F^{-1} \left| \jump{\boldsymbol{e}_h^n} \right|^2\right) \leq   \seminorm{\star}{\boldsymbol{e}_h^n}^2$ we conclude that
\[
  \sum_{F \in \Fhd} \int_{F} \left(
      \avg{\projL{k} \boldsymbol{\sigma}(\boldsymbol{\nabla} \boldsymbol{u}^n) - \boldsymbol{\sigma}(\boldsymbol{\nabla} \boldsymbol{u}^n) } \normalF
      \right) \cdot \jump{\boldsymbol{e}_h^n} \leq \delta \nu \seminorm{\star}{\boldsymbol{e}_h^n}^2 + c(\delta) \widehat{\diffErr}.
      \qedhere
\]
\end{proof}

We are now able to show the main result of this section, which closely resembles Theorem \ref{thm:error.estimate} but differs in a single, yet critical, step of the proof.

\begin{theorem}[Velocity error estimate for $1 < r < 2$]\label{thm:error.estimate.rs2}
  Under the regularity Assumption \ref{reg:ass:u}, and further assuming $\tstep < \frac{1}{5 C_{\boldsymbol{u}}^\star}$ with $C_{\boldsymbol{u}}^\star$  introduced in \eqref{eq:Custar}, for $r \in (1,2)$ it holds,
  \begin{equation}\label{eq:error.estimate.rs2}
    \norm{\boldsymbol{L}^2(\Omega)}{\boldsymbol{e}_h^N}^2
    + \sum_{n=1}^N \tstep \left[
    \left(a_h (\boldsymbol{u}_h^n,\boldsymbol{e}_h^n) - a_h(\projRT{k}\boldsymbol{u}^n,\boldsymbol{e}_h^n) \right)
    + \seminorm{\boldsymbol{u}_h^n}{\boldsymbol{e}_h^n}^2
    \right]
    \lesssim
    \sum_{n=1}^N \tstep \widehat{\Err},
  \end{equation}
  where, recalling \eqref{eq:timeErr}, \eqref{eq:diffErr:rs2}, \eqref{eq:convErr}, and \eqref{eq:upwErr}, the consistency error is defined by
  \[
  \widehat{\Err} \coloneq \timeErr + \widehat{\diffErr} + \convErr + \upwErr.
  \]
  and the hidden constant in \eqref{eq:error.estimate.rs2} depends linearly on $\exp\left[ \frac{5C_{\boldsymbol{u}}^\star}{1 - 5 C_{\boldsymbol{u}}^\star \tstep} \tF \right]$.
\end{theorem}
\begin{proof}
The argument follows the same line of reasoning as in the proof of Theorem \ref{thm:error.estimate}. Making the hidden constant in \eqref{eq:ah:coercivity:rs2} explicit we get
\[
  a_h(\boldsymbol{u}_h^n,\boldsymbol{e}_h^n) - a_h(\projRT{k}\boldsymbol{u}^n,\boldsymbol{e}_h^n) \geq
  \frac12 \left( a_h(\boldsymbol{u}_h^n,\boldsymbol{e}_h^n) - a_h(\projRT{k}\boldsymbol{u}^n,\boldsymbol{e}_h^n) \right) + \frac{C_{a}}{2} \nu \seminorm{\star}{\boldsymbol{e}_h^n}^{2}.
  \]
Hence, by the same argument leading to \eqref{eq:error:lower:bound}, one obtains
 \begin{multline}\label{eq:error:lower:bound:rs2}
    \timeTerm{n} (\boldsymbol{e}_h^n)
    + \diffTerm{n} (\boldsymbol{e}_h^n)
    + \convTerm{n} (\boldsymbol{e}_h^n)
    + \upwTerm{n} (\boldsymbol{e}_h^n)
    \geq \frac{1}{2 \tstep} \norm{\boldsymbol{L}^2(\Omega)}{\boldsymbol{e}_h^n}^2 - \frac{1}{2 \tstep} \norm{\boldsymbol{L}^2(\Omega)}{\boldsymbol{e}_h^{n-1}}^2
    \\
    + \frac12 \left( a_h(\boldsymbol{u}_h^n,\boldsymbol{e}_h^n) - a_h(\projRT{k}\boldsymbol{u}^n,\boldsymbol{e}_h^n) \right) + \frac12 C_a \nu \seminorm{\star}{\boldsymbol{e}_h^n}^{2} + \seminorm{\boldsymbol{u}_h^n}{\boldsymbol{e}_h^n}^2.
  \end{multline}
  Plugging our estimates \eqref{eq:timeTerm:estimate} of $\timeTerm{n}(\boldsymbol{e}_h^n)$,
  \eqref{eq:diffTerm:estimate:rs2} of $\diffTerm{n}(\boldsymbol{e}_h^n)$,
  \eqref{eq:convTerm:estimate} of $\convTerm{n}(\boldsymbol{e}_h^n)$,
  and \eqref{eq:upwTerm:estimate} of $\upwTerm{n}(\boldsymbol{e}_h^n)$ into \eqref{eq:error:lower:bound:rs2} and rearranging terms, we obtain
  \begin{multline*}
    \left( \frac{1}{2 \tstep} - (3 \delta + 2C_{\boldsymbol{u}}^\star)  \right) \norm{\boldsymbol{L}^2(\Omega)}{\boldsymbol{e}_h^n}^2
    +\left(\frac12 C_a - \delta \right) \nu \seminorm{\star}{\boldsymbol{e}_h^n}^{2}
    +\left( 1 - 3 \delta \right) \seminorm{\boldsymbol{u}_h^n}{\boldsymbol{e}_h^n}^2 \\
    +\left( \frac12 - \delta \right) \left( a_h(\boldsymbol{u}_h^n,\boldsymbol{e}_h^n) - a_h(\projRT{k}\boldsymbol{u}^n,\boldsymbol{e}_h^n) \right)
    \leq 
     \frac{1}{2 \tstep} \norm{\boldsymbol{L}^2(\Omega)}{\boldsymbol{e}_h^{n-1}}^2
    + 4 c(\delta) \widehat{\Err}.
  \end{multline*}
 Setting $\tilde{C} \coloneqq \frac12 \min \{ C_a, 1 \}$ and multiplying the above inequality by $2 \tstep$, we arrive at 
  \begin{multline*}
    \left( 1 - \tstep (6 \delta + 4C_{\boldsymbol{u}}^\star)  \right) \norm{\boldsymbol{L}^2(\Omega)}{\boldsymbol{e}_h^n}^2
    \\
    +\left(2\tilde{C} - 6 \delta \right) \tstep\left[ \nu \seminorm{\star}{\boldsymbol{e}_h^n}^{2}
    +\seminorm{\boldsymbol{u}_h^n}{\boldsymbol{e}_h^n}^2 
    + \left( a_h(\boldsymbol{u}_h^n,\boldsymbol{e}_h^n) - a_h(\projRT{k}\boldsymbol{u}^n,\boldsymbol{e}_h^n) \right) \right]
    \\
    \leq 
     \norm{\boldsymbol{L}^2(\Omega)}{\boldsymbol{e}_h^{n-1}}^2
    + 8 c(\delta) \tstep \widehat{\Err}.
  \end{multline*}
  From this point, the proof proceeds exactly as in Theorem \ref{thm:error.estimate}. We choose $\hat{\delta} \coloneqq \min \left\{ \frac{C_{\boldsymbol{u}}^\star}{6}, \frac{\tilde{C}}{6} \right\}$, set $\hat{C} \coloneqq c(\hat{\delta})$ and then, under the assumption $\tstep < \frac{1}{5 C_{\boldsymbol{u}}^\star}$, we divide both sides by $1-5 C_{\boldsymbol{u}}^\star > 0$. 
  Summing over $n=1,\ldots, N$ we arrive at
  \begin{multline*}
    \norm{\boldsymbol{L}^2(\Omega)}{\boldsymbol{e}_h^N}^2
    +\frac{\tilde{C}}{1 - 5 C_{\boldsymbol{u}}^\star \tstep}
    \sum_{n=1}^N \tstep \left[ \nu \seminorm{\star}{\boldsymbol{e}_h^n}^{2}
    +\seminorm{\boldsymbol{u}_h^n}{\boldsymbol{e}_h^n}^2 
    + \left( a_h(\boldsymbol{u}_h^n,\boldsymbol{e}_h^n) - a_h(\projRT{k}\boldsymbol{u}^n,\boldsymbol{e}_h^n) \right) \right] \\
    \leq 
    \frac{5 C_{\boldsymbol{u}}^\star \tstep }{1 - 5 C_{\boldsymbol{u}}^\star \tstep}
    \sum_{n=1}^N \norm{\boldsymbol{L}^2(\Omega)}{\boldsymbol{e}_h^{n-1}}^2
    +\frac{ 8 \hat{C}}{1 - 5 C_{\boldsymbol{u}}^\star \tstep} \sum_{n=1}^N \tstep \widehat{\Err}.
  \end{multline*}
  We conclude applying the discrete Gr\"onwall inequality (see \cite[Lemma 5.1]{Heywood.Rannacher:90}), rearranging terms as in the last two steps of the proof of Theorem \ref{thm:error.estimate}, and noting that $\nu \seminorm{\star}{\boldsymbol{e}_h^n}^{2} \geq 0$.
\end{proof}
\begin{table}\centering
  \adjustbox{width=\textwidth}{%
    \begin{tabular}{cccccc}
      \toprule
      Term & Definition & Diffusion-dominated & Diffusion-dominated & Diffusion-dominated & Convection-dominated \\
      & & $r < 2$ & $r = 2$ & $r > 2$ & \\
      \midrule
      $\sum_{n=1}^N \tstep \timeErr$ & \eqref{eq:timeErr} &
      $\tstep^2 + h^{2(m+1)}$
      & $\tstep^2 + h^{2(m+1)}$
      & $\tstep^2 + h^{2(m+1)}$
      & $\tstep^2 + h^{2(m+1)}$
      \\
      $\sum_{n=1}^N \tstep \diffErr, \quad \sum_{n=1}^N \tstep \widehat{\diffErr}$ & \eqref{eq:diffErr},\eqref{eq:diffErr:rs2} & $h^{rm}$ & $h^{2m}$ & $\begin{cases}
        h^{r'(m+1)} & \text{if $m \ge \frac{r'}{2-r'}$}
        \\
        h^{2m} & \text{otherwise}
      \end{cases}$ & $h^{2m+1}$ \\
      $ \sum_{n=1}^N \tstep \big( \convErr + \upwErr \big)$
      & \eqref{eq:convErr}, \eqref{eq:upwErr}
      & $h^{2m+1}$ & $h^{2m+1}$ & $h^{2m+1}$ & $h^{2m+1}$ \\
      \midrule
      $\sum_{n=1}^N \tstep \Err, \quad \sum_{n=1}^N \tstep \widehat{\Err}$ &
      & $\tstep^2 + h^{rm}$
      & $\tstep^2 + h^{2m}$
      & $\tstep^2 + \begin{cases}
        h^{r'(m+1)} & \text{if $m \ge \frac{r'}{2-r'}$}
        \\
        h^{2m} & \text{otherwise}
      \end{cases}$
      & $\tstep^2 + h^{2m+1}$  \\
      \bottomrule
    \end{tabular}
  }
  \caption{Dominant term for each contribution to $\Err$ and $\widehat{\Err}$ and for the right-hand side of \eqref{eq:error.estimate} and \eqref{eq:error.estimate.rs2}. Here above $1 \le m \le k$.
    \label{tab:convergence.rates} }
\end{table}
\begin{remark}
  The quantity $a_h (\boldsymbol{u}_h^n,\boldsymbol{e}_h^n) - a_h(\projRT{k}\boldsymbol{u}^n,\boldsymbol{e}_h^n)$ appearing in \eqref{eq:error.estimate.rs2} represents, due to the monotonicity property of $a_h$, a diffusive error measure (discrete quasi-norm). 
  Due to \eqref{eq:ah:coercivity}, whenever \eqref{eq:not-true} holds, the above quantity will control $\norm{1,r,h}{\boldsymbol{u}_h^n}^2$, immediately leading to an analogous bound as in  \eqref{eq:error.estimate.rs2} also for 
  $$
  \sum_{n=1}^N \tstep \nu \norm{1,r,h}{\boldsymbol{e}_h^n}^{2} \, ;
  $$
  see also Remark \ref{rem:class-error}.
  Finally note that, from the (uniform in $n$) convergence of $\norm{\boldsymbol{L}^2(\Omega)}{\boldsymbol{e}_h^n}$, $n =1,2,..,N$, of \eqref{eq:error.estimate.rs2} and making use of inverse estimates, it can be easily proved that (provided the mesh sequence is quasi-uniform and $k \ge 2$) bound \eqref{eq:not-true} does indeed hold. 
\end{remark}

\section{Numerical tests}\label{sec:Numerical tests}

In the present section we investigate the performance of the method from the computational standpoint. Test 1 represents a problem with a manufactured regular solution, the objective being to compare the numerical error with the theoretical convergence rates for different values of the model parameters $\nu$ and $r$. Test 2 studies a benchmark problem from the literature \cite{Mahmood.Bilal.ea:20} and has a more practical flavor, the focus being on some qualitative characteristics of the numerical solution.

The implementation of the scheme was performed in the language C++. All performed tests are limited to the lowest order case ($k=1$) and a standard Picard iteration is used at each time step to solve the non-linear problem \eqref{eq:full.discrete.problem} (the initial value being the discrete solution at the previous time step). The ``safeguard'' parameter $C_F$ in \eqref{eq:gammaF} was set equal to $10^{-4}$, but was never activated during all performed tests (that is, the max in \eqref{eq:gammaF} was never attained by $C_F$).  

\subsection{Test 1}

We consider the square domain $\Omega=(0,1)^2$, the time interval $[0,1]$, and the exact solution
\[
\boldsymbol{u} (t,x,y) =e^{\frac{t}{10}} \begin{pmatrix}
  16y(1-y)(1-2y)\sin^2(\pi x) \\
  -8\pi y^2(1-y)^2 \sin(2\pi x)
\end{pmatrix},\qquad
p (t,x,y) = e^{\frac{t}{10}} \sin (\pi x) \cos (\pi y).
\]
The loading term $\boldsymbol{f}$, the initial data, and Dirichlet boundary condition are set in accordance to the above expressions.
We emphasize that the solution is purposefully very smooth in time (low derivatives) in order to better underline the spatial discretization error, which is the focus of this contribution.
We solve the problem numerically with the proposed method for a family of five unstructured triangular meshes of average size $h$ equal to $\{0.349,0.149,0.079,0.033,0.015\}$, respectively.
The time-step size $\tstep$ is set roughly equal to $\frac32 h$ so, that finer spatial meshes correspond to finer time grids.

We monitor the following errors for velocities and pressures:
\[
\begin{gathered}
  \mathrm{velERR}_h \coloneqq
  \norm{\boldsymbol{L}^2(\Omega)}{\boldsymbol{u}(\tF) - \boldsymbol{u}_h^{N}}^2
  + \tstep \sum_{n=1}^N \left(
  \nu \norm{1,r,h}{\boldsymbol{u} (t^n) - \boldsymbol{u}_h^{n}}^{\overline{r}}
  + \seminorm{\boldsymbol{u}_h^n}{\boldsymbol{u} (t^n) - \boldsymbol{u}_h^{n}}^2
  \right) \, , \\
  \mathrm{preERR}_h \coloneq
  \tstep \sum_{n=1}^N \norm{L^{r'}(\Omega)}{p (t^n) - p_h^{n}}^2
  \, .
\end{gathered}
\]
According to our theoretical findings, see for instance Table \ref{tab:convergence.rates}, one can easily derive the expected velocity error behaviour in $h$.
Note that, as confirmed by the results below, for the considered meshes we can ignore the influence of $\tstep$ on the error for the reason outlined above (the solution is, purposefully, very slowly changing in time).
We obtain that $\mathrm{velERR}_h$ should behave as $O(h^{3})$ in convection dominated cases and that
\begin{equation}\label{eq:expected}
  \mathrm{velERR}_h \simeq
  \left\{
  \begin{aligned}
    &  O(h^{r}) \quad \textrm{ for }  1 < r \le 2 \, , \\
    &  O(h^2) \quad \textrm{ for }  r > 2 \, .
  \end{aligned}
  \right.
\end{equation}
in diffusion dominated cases.

Although, as already noted, deriving error bounds for the pressures is outside the scope of the present work, we include also a numerical study of the  pressure error for the sake of completeness.

In Figure~\ref{fig:conv.req2} we plot the convergence lines as a function of the mesh size $h$ for a Newtonian fluid $r=2$; in this case our method corresponds to a fully implicit version of the scheme in \cite{Han.Hou:21}. We consider two choices for the parameter $\nu$, corresponding to a diffusion dominated ($\nu=1$) and a convection dominated ($\nu=10^{-5}$) regime.
We can appreciate an $O(h^2)$ convergence rate for the velocities in the first case and an $O(h^3)$ (pre-asymptotic) error reduction rate for the second one, as expected by our theory.

In Figures \ref{fig:conv.rs2} and \ref{fig:conv.rg2} we plot the errors for various choices of $r$ different from $2$. In all cases we can observe the expected higher order pre-asymptotic convergence rate in convection dominated regimes. Regarding diffusion dominated regimes, for $r<2$ (Figure \ref{fig:conv.rs2}), we can indeed note a non-optimal (i.e. less than the interpolation error) convergence rate, as expected from the theoretical results \eqref{eq:expected}.
In the case $r>2$ (Figure \ref{fig:conv.rg2}), the practical convergence rate for the proposed test seems to behave as $O(h^2)$, which is again aligned with the theoretical findings \eqref{eq:expected}.

In many among the presented cases, the pressure error behaves as the corresponding expected interpolation error but, especially for $(\nu,r) = (1,2)$, a non-negligible loss in the convergence rate is observed (in accordance, for example, to the theoretical results for the simpler stationary non-Newtonian Stokes problem).
Finally, we note that the above tests also serve the purpose (by comparing the $\nu=1$ and $\nu=10^{-5}$ plots) of underlining the Reynolds quasi-robustness of the proposed scheme.
To further investigate this aspect, we performed additional experiments in which we fix the mesh size $h=0.149$ and, for different values of $r$, we consider $\nu \in \{ 1,10^{-1},10^{-2},\ldots,10^{-7}\}$. In Figures \ref{fig:rob.rs2} and \ref{fig:rob.rg2} we plot the velocity error $\mathrm{velERR}_h$.
For each value of $r$, the error shows an initial small decrease and then stabilizes, remaining essentially constant for smaller values of $\nu$.
\begin{figure}
  \centering
  \begin{subfigure}[b]{0.45\textwidth}
    \centering
    \includegraphics[width=\linewidth]{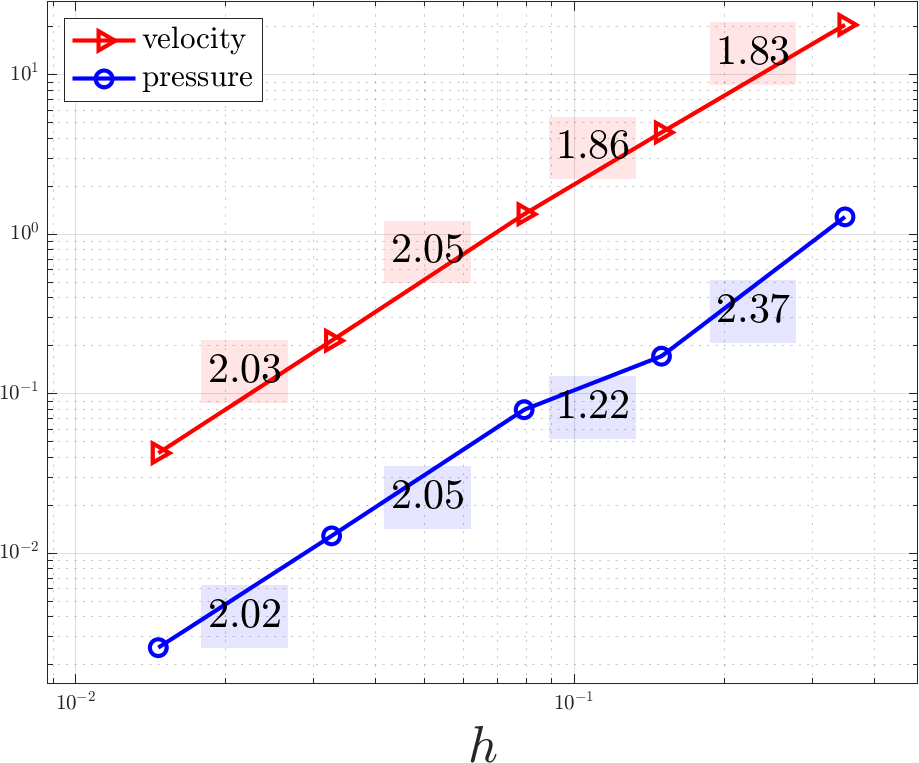}
    \caption{$ \nu = 1$}
  \end{subfigure}
  \hfill
  \begin{subfigure}[b]{0.45\textwidth}
    \centering
    \includegraphics[width=\linewidth]{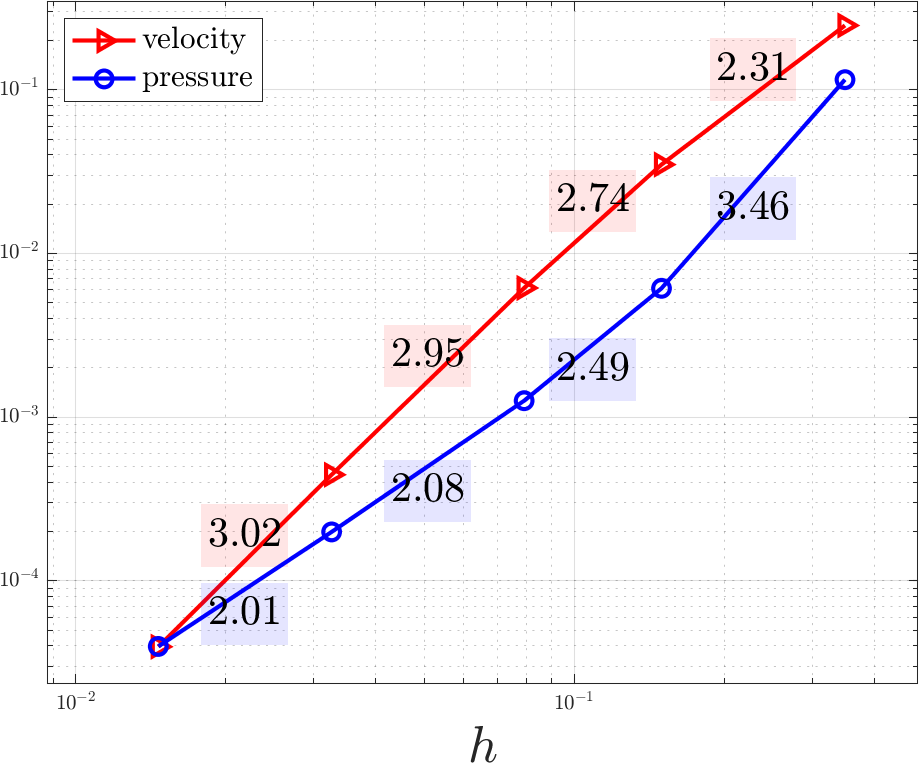}
    \caption{$ \nu = 10^{-5}$}
  \end{subfigure}
  \caption{Test 1. Error plots for $r = 2$.}
  \label{fig:conv.req2}
\end{figure}

\begin{figure}
  \centering
  \begin{subfigure}[b]{0.45\textwidth}
    \centering
    \includegraphics[width=\linewidth]{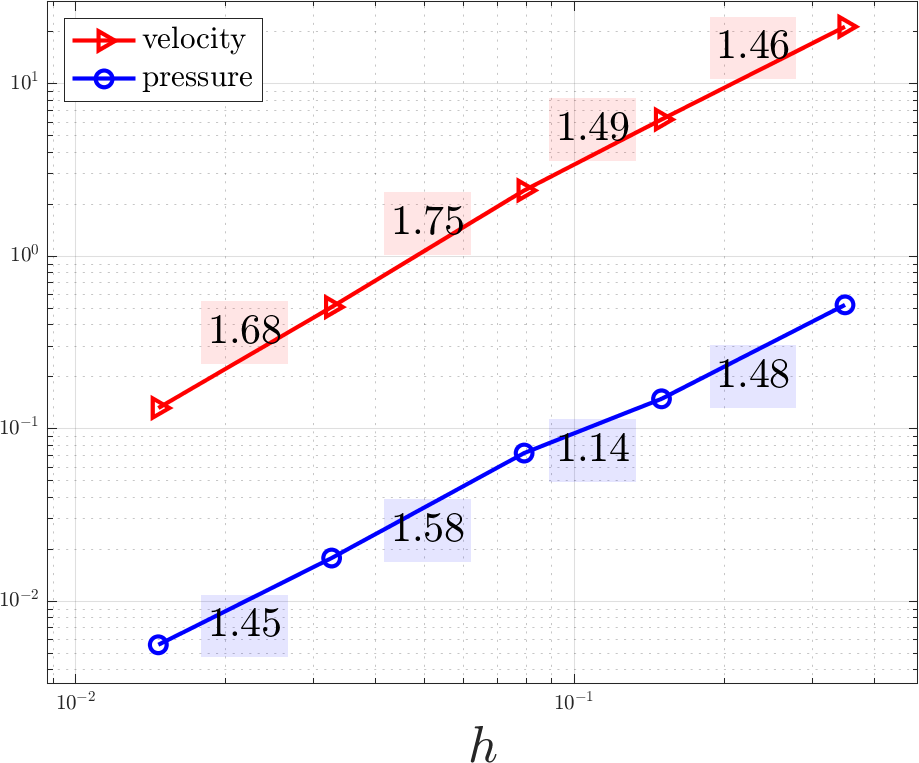}
    \caption{$r = 1.5$ and $\nu=1$}
  \end{subfigure}
  \hfill
  \begin{subfigure}[b]{0.45\textwidth}
    \centering
    \includegraphics[width=\linewidth]{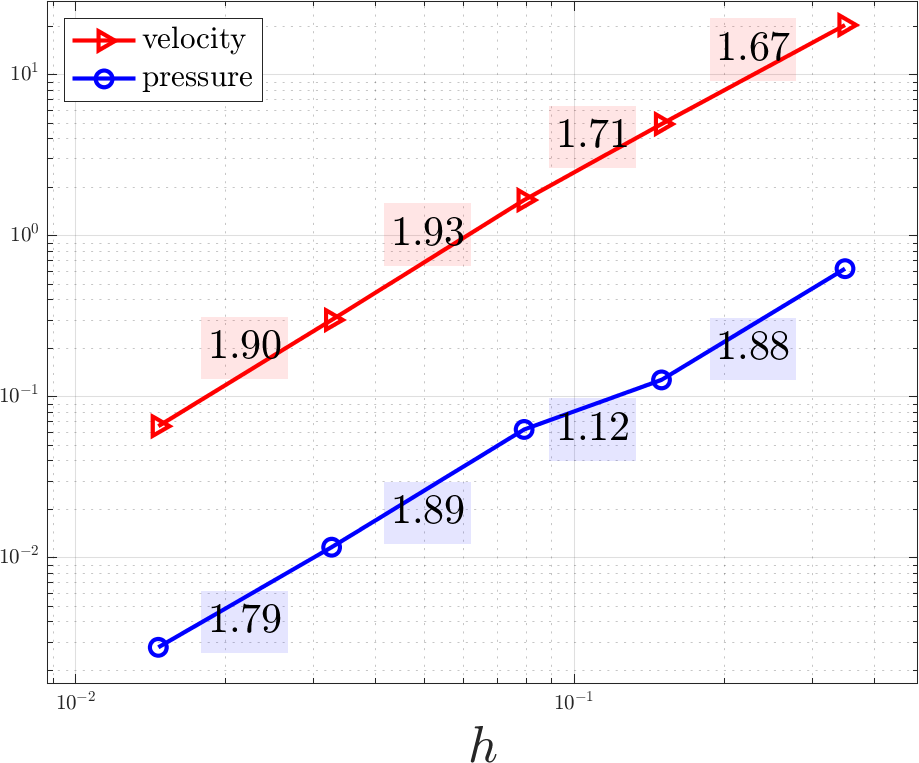}
    \caption{$r = 1.75$ and $\nu=1$}
  \end{subfigure}
  \\
  \vspace{0.5cm}
  \begin{subfigure}[b]{0.45\textwidth}
    \centering
    \includegraphics[width=\linewidth]{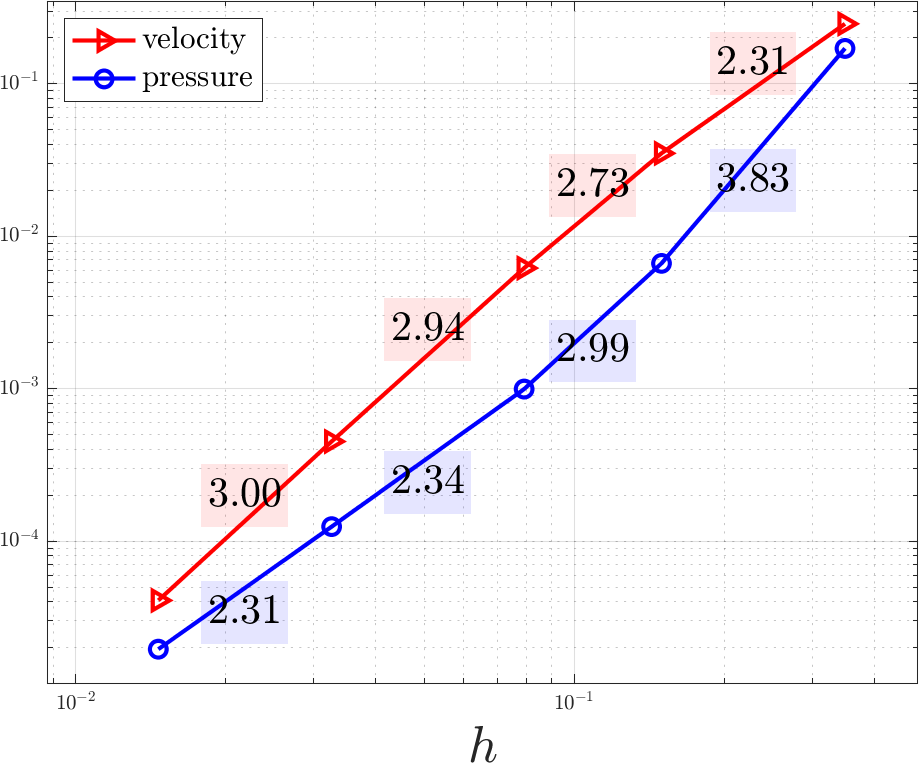}
    \caption{$r = 1.5$ and $\nu=10^{-5}$}
  \end{subfigure}
  \hfill
  \begin{subfigure}[b]{0.45\textwidth}
    \centering
    \includegraphics[width=\linewidth]{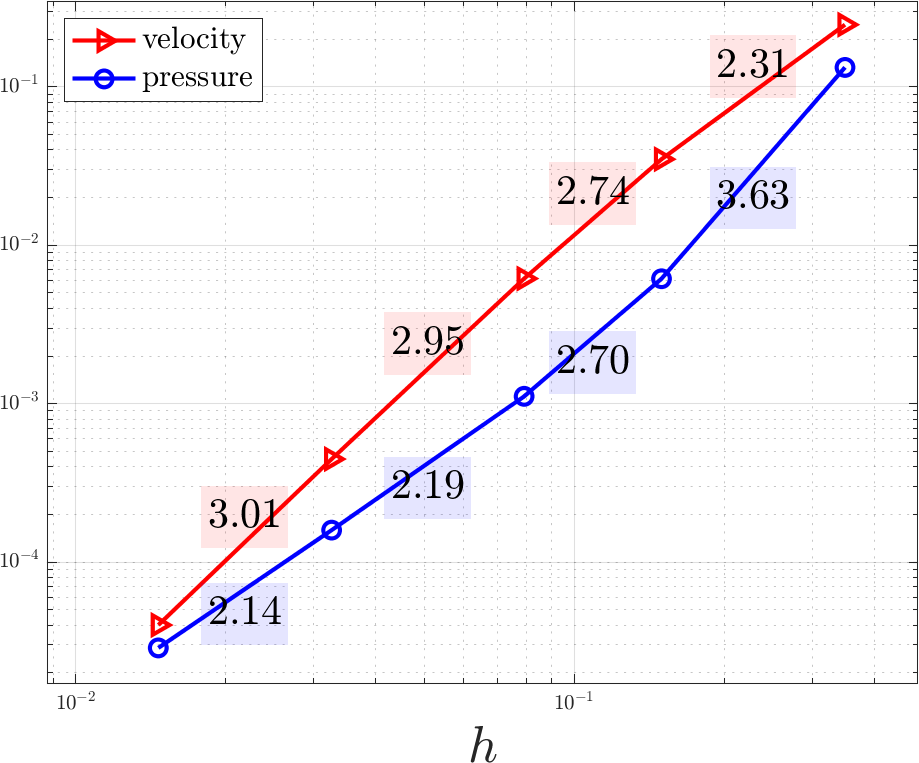}
    \caption{$r = 1.75$ and $\nu=10^{-5}$}
  \end{subfigure}
  \caption{Test 1. Error plots for $r < 2$.}
  \label{fig:conv.rs2}
\end{figure}

\begin{figure}
  \centering
  \begin{subfigure}[b]{0.45\textwidth}
    \centering
    \includegraphics[width=\linewidth]{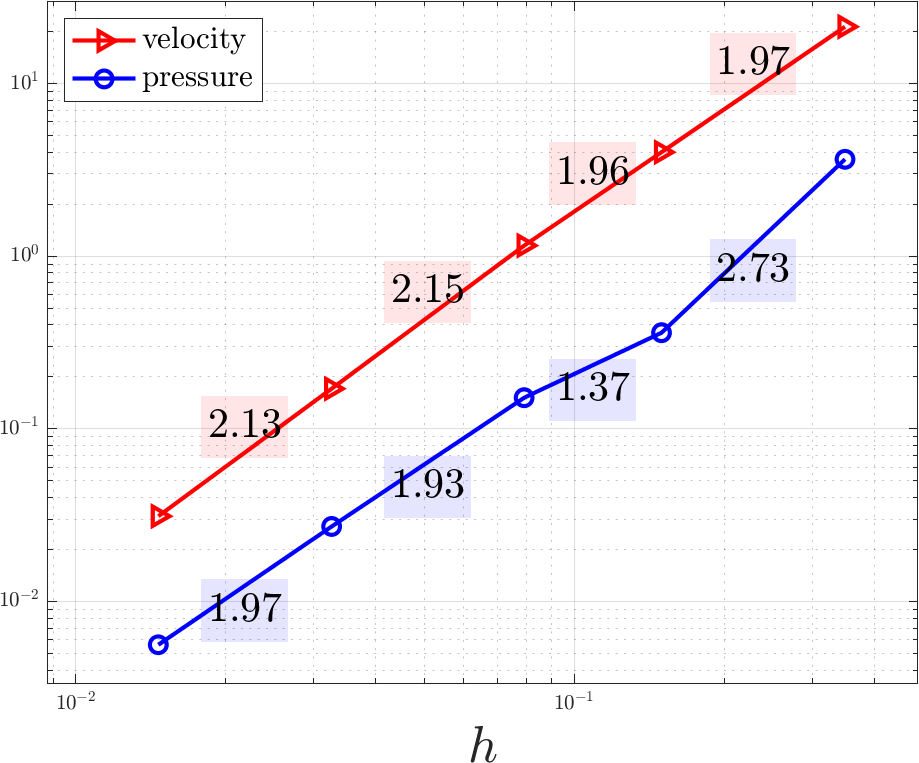}
    \caption{$r = 2.25$ and $\nu=1$}
  \end{subfigure}
  \hfill
  \begin{subfigure}[b]{0.45\textwidth}
    \centering
    \includegraphics[width=\linewidth]{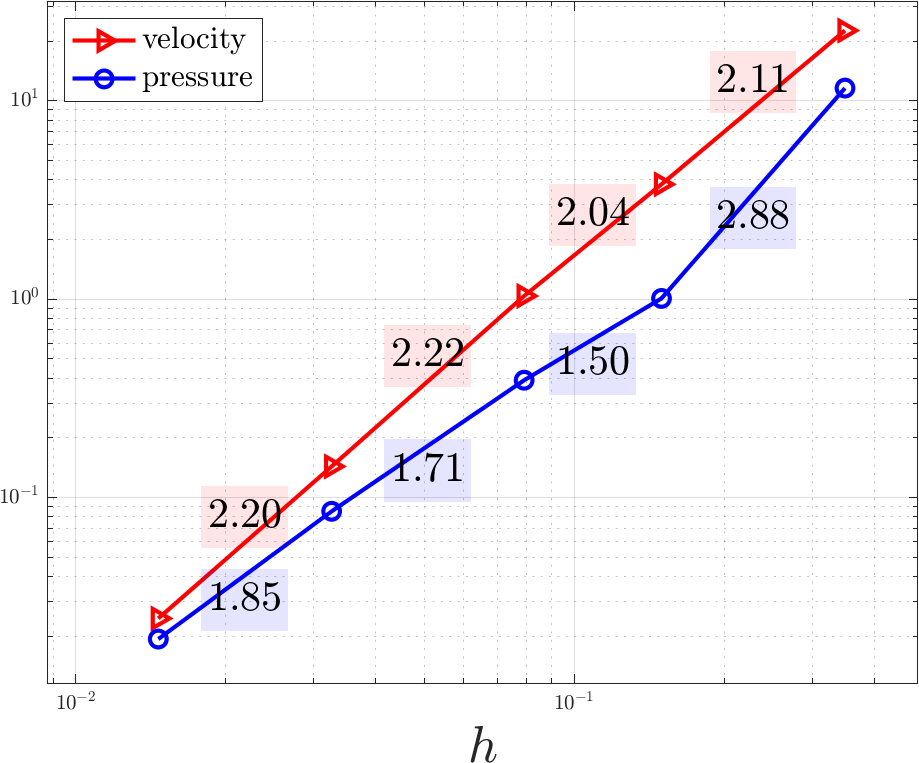}
    \caption{$r = 2.5$ and $\nu=1$}
  \end{subfigure}
  \\
  \vspace{0.5cm}
  \begin{subfigure}[b]{0.45\textwidth}
    \centering
    \includegraphics[width=\linewidth]{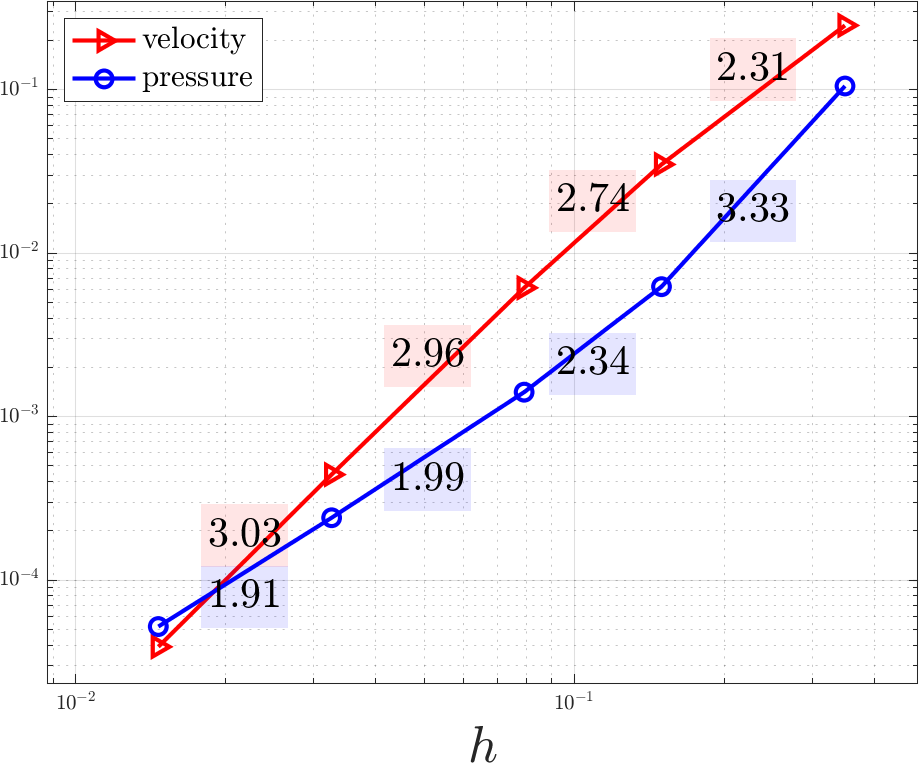}
    \caption{$r = 2.25$ and $\nu=10^{-5}$}
  \end{subfigure}
  \hfill
  \begin{subfigure}[b]{0.45\textwidth}
    \centering
    \includegraphics[width=\linewidth]{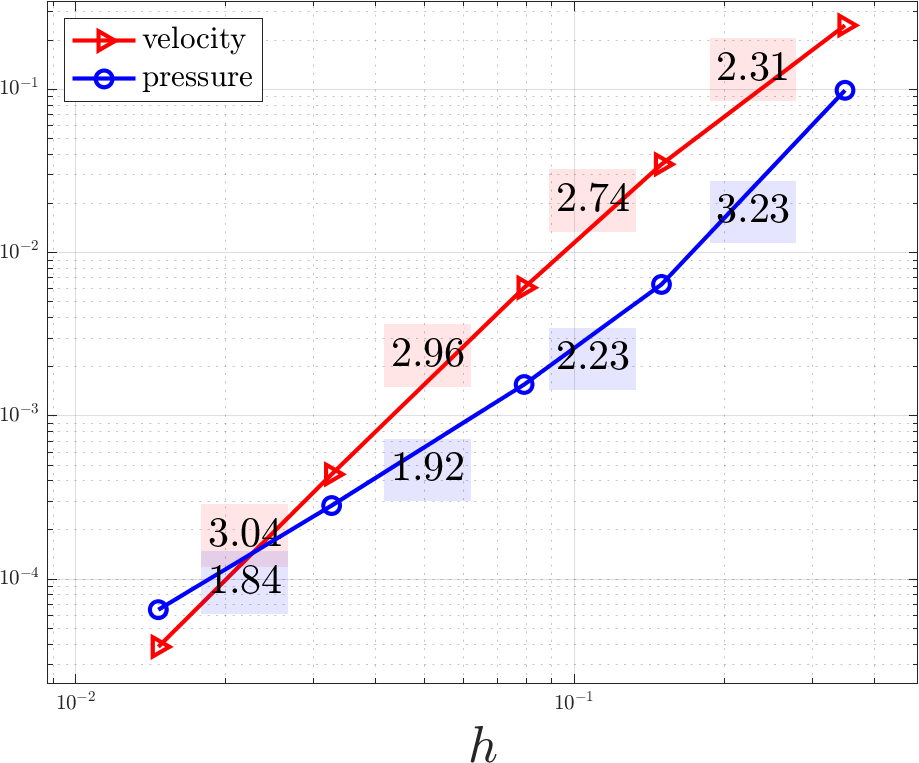}
    \caption{$r = 2.5$ and $\nu=10^{-5}$}
  \end{subfigure}
  \caption{Test 1. Error plots for $r >2$.}
  \label{fig:conv.rg2}
\end{figure}

\begin{figure}
  \centering
  \begin{subfigure}[b]{0.45\textwidth}
    \centering
    \includegraphics[width=\linewidth]{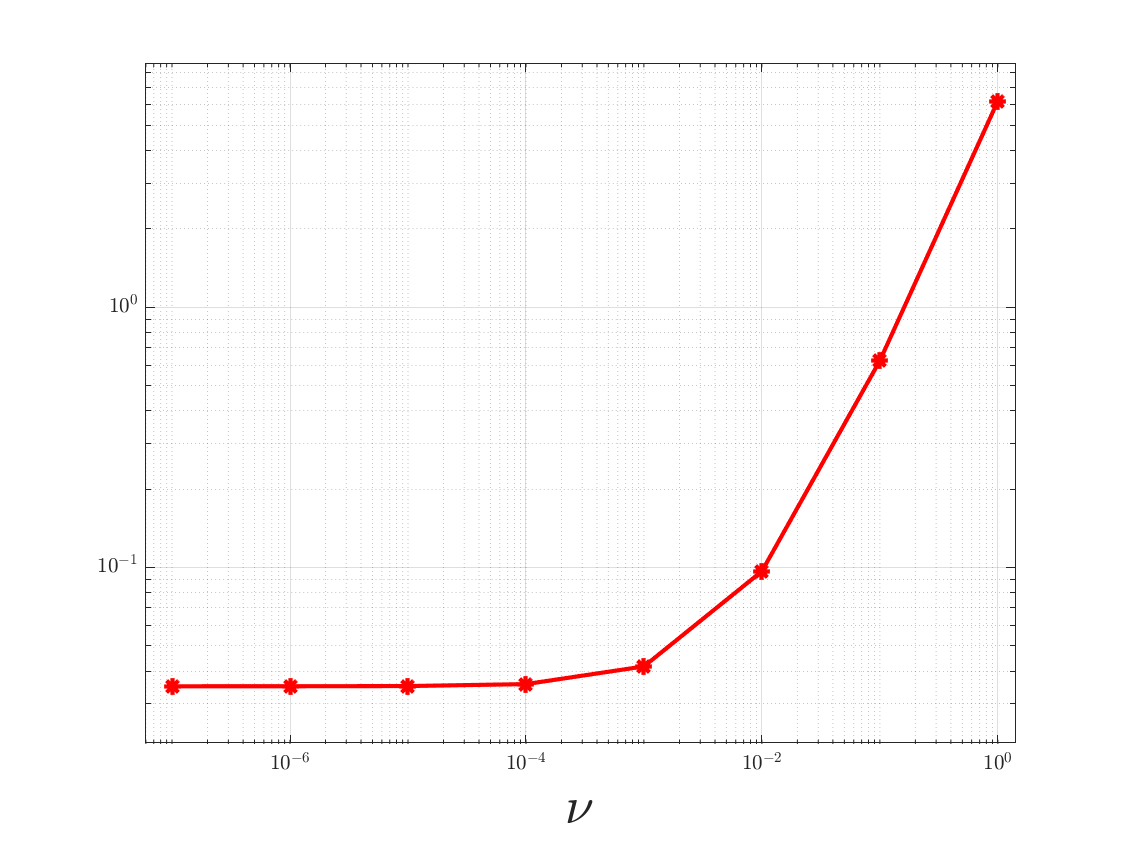}
    \caption{$ r = 1.5$}
  \end{subfigure}
  \hfill
  \begin{subfigure}[b]{0.45\textwidth}
    \centering
    \includegraphics[width=\linewidth]{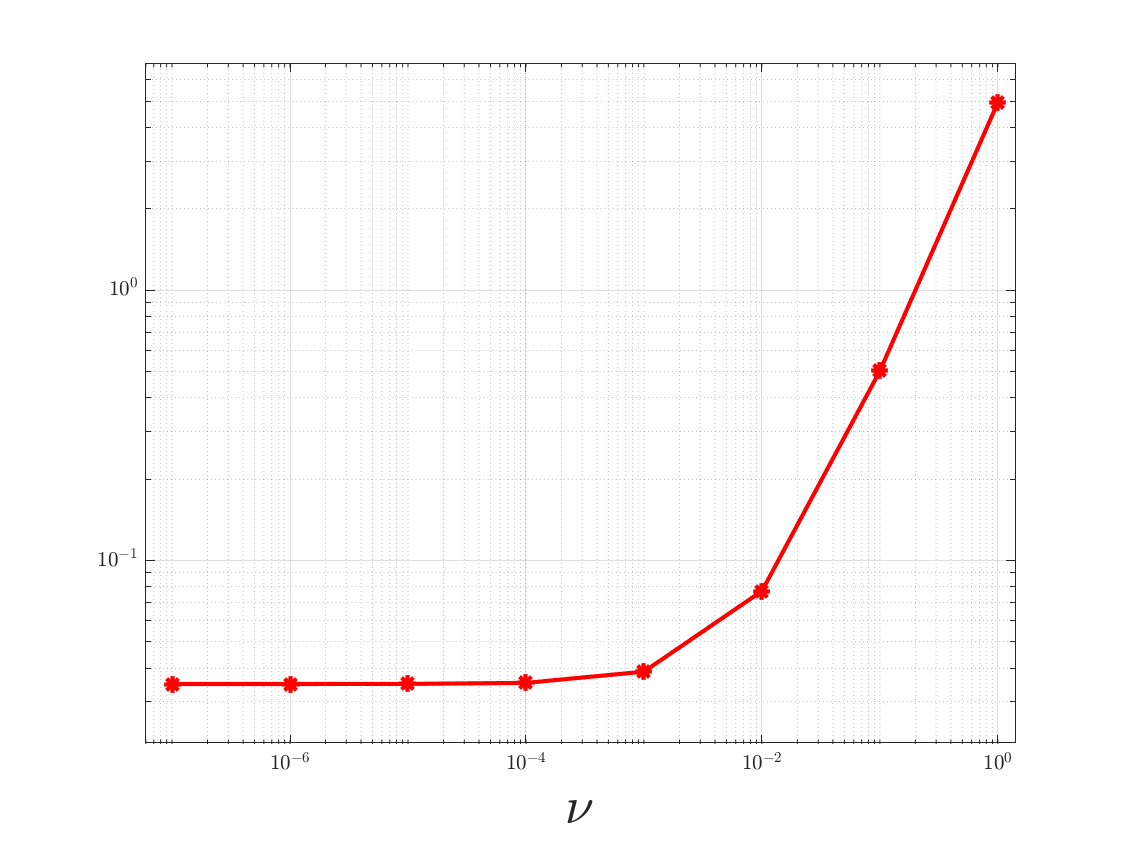}
    \caption{$ r = 1.75$}
  \end{subfigure}
  \caption{Test 1. Robustness plots for $r < 2$.}
  \label{fig:rob.rs2}
\end{figure}

\begin{figure}
  \centering
  \begin{subfigure}[b]{0.45\textwidth}
    \centering
    \includegraphics[width=\linewidth]{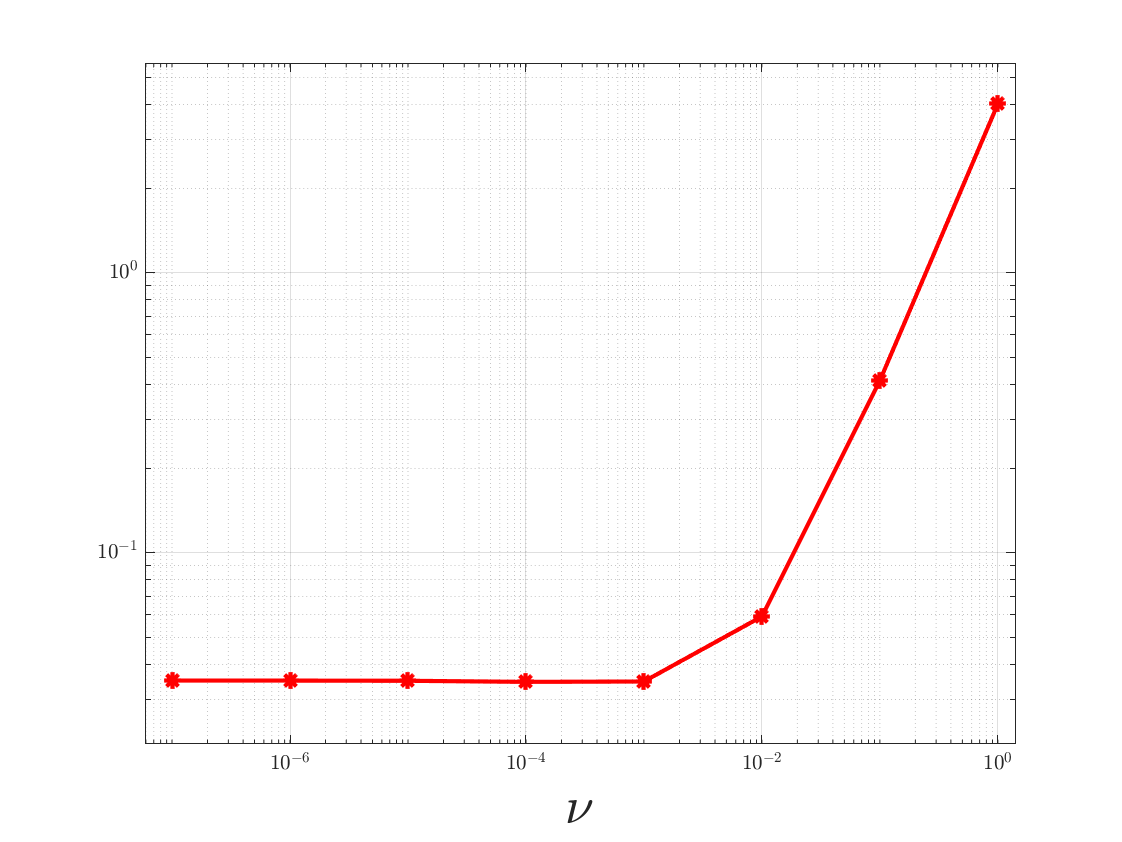}
    \caption{$ r = 2.25$}
  \end{subfigure}
  \hfill
  \begin{subfigure}[b]{0.45\textwidth}
    \centering
    \includegraphics[width=\linewidth]{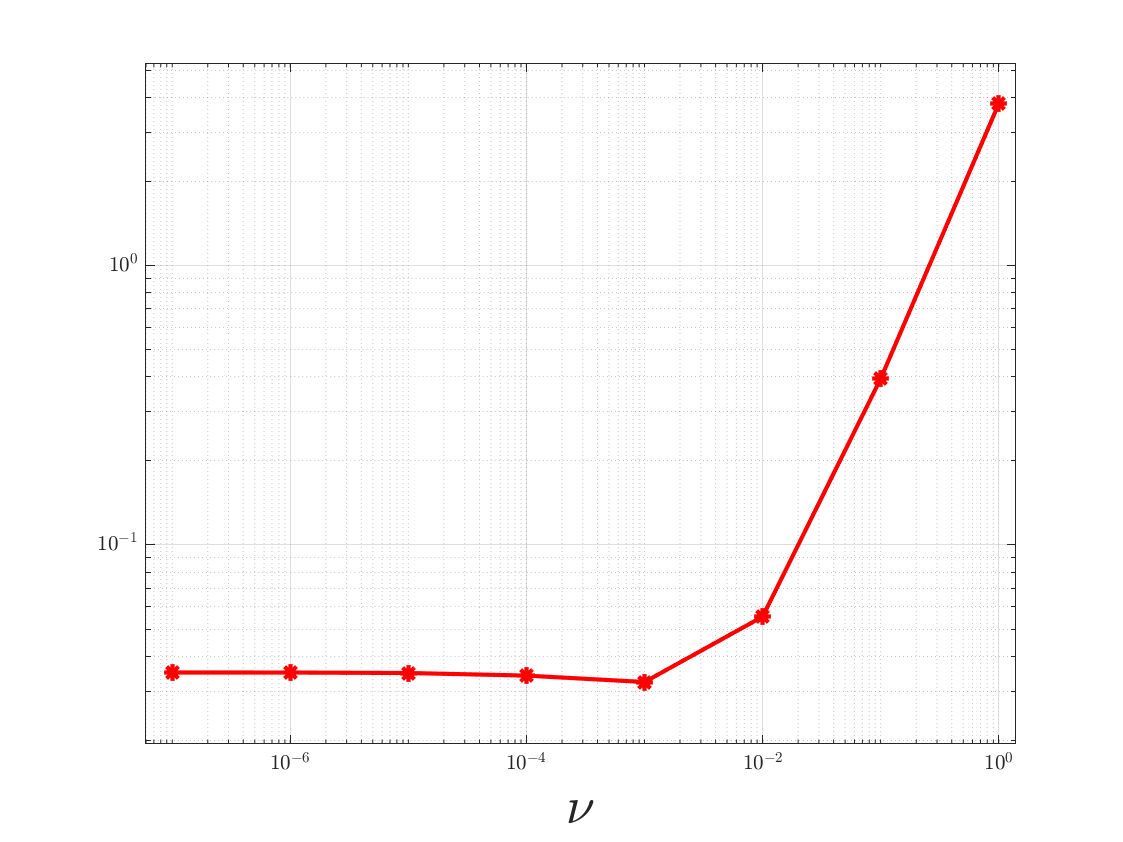}
    \caption{$ r = 2.5$}
  \end{subfigure}
  \caption{Test 1. Robustness plots for $r > 2$.}
  \label{fig:rob.rg2}
\end{figure}

\subsection{Test 2}

We now consider the benchmark test from \cite{Mahmood.Bilal.ea:20}, simulating a ``hunchback'' channel.
The adopted mesh and the domain $\Omega$ are shown in Figure \ref{fig:mesh-r2} (left).
The force term $\boldsymbol{f}$ is set equal to zero and we enforce a parabolic incoming horizontal velocity with maximum magnitude equal to $0.3$ at the inlet, that is, the leftmost vertical edge.
The velocities at the upper and lower walls of the channel are set equal to zero (no slip boundary conditions) while, at the outlet, that is, the rightmost vertical edge, we apply homogeneous Neumann boundary conditions.
Finally, the diffusive parameter is set equal to $\nu=10^{-2}$ and our investigation is focused on different significant values of $r$, ranging from $r=1.5$ to $r=2.5$.

We solve the problem up to the final time $\tF = 10$ (starting from time zero) with 20 discrete time-steps; all the plots below are depicted at the final time.
In Figure \ref{fig:mesh-r2} (right) we show the streamlines obtained for the Newtonian case $r=2$. In Figure \ref{fig:streams} we instead plot the two extreme cases $r=1.5$ (shear thinning fluid) and $r=2.5$ (shear thickening fluid).
We can, in particular, appreciate the different shapes of the cavity vortex because of the different viscosities.

In Figure \ref{fig:cutlines} we present the normal (horizontal) velocity values at different cutlines of the channel and for various values of the model parameter $r$. Higher values of $r$ lead, as expected, to higher peak velocities and slightly more acute profiles.

We finally note that we have also tested the same problem with different (triangular) meshes and no appreaciable change in the graphs was noticed.
\begin{figure}
  \includegraphics[height=3cm]{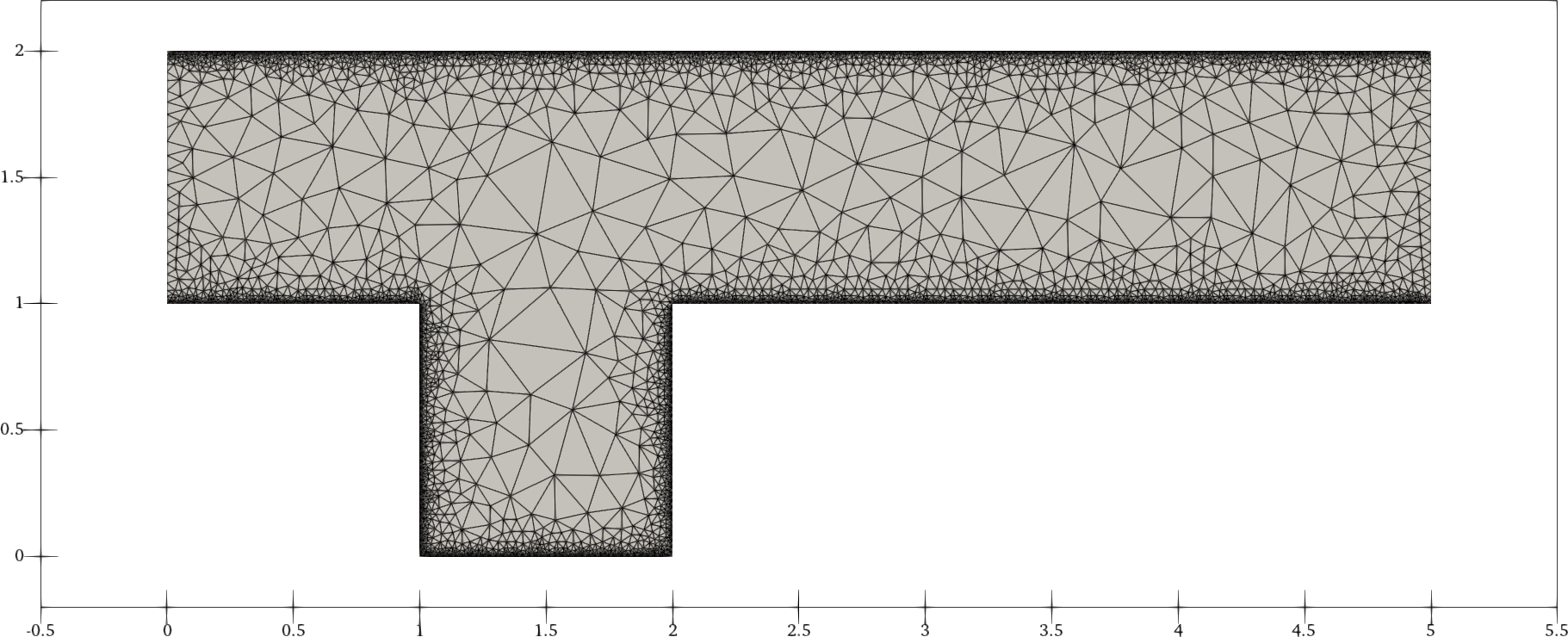}
  \hfill
  \includegraphics[height=3cm]{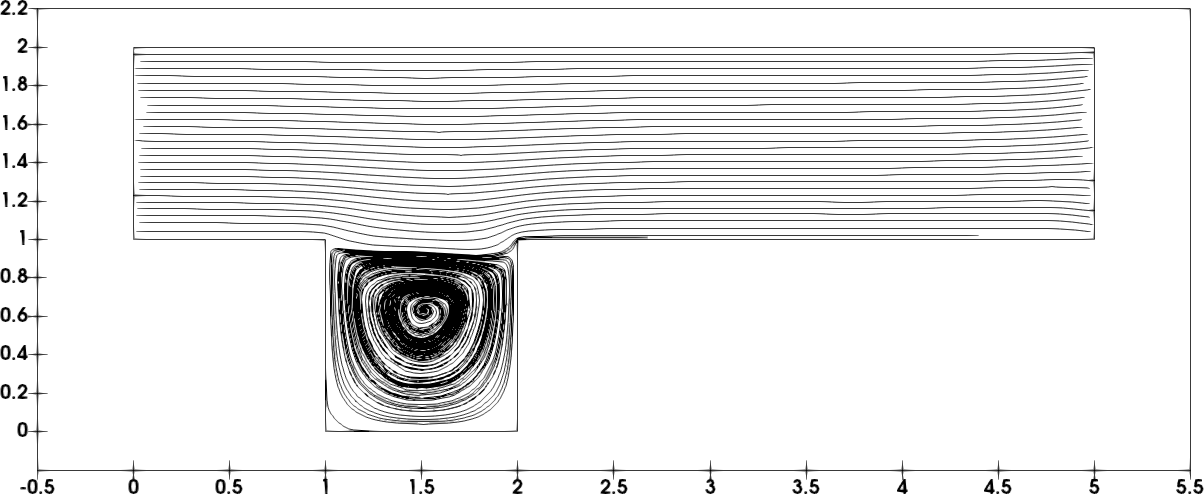}
  \caption{Test 2. Domain with adopted mesh (left) and streamlines at $T=10$ for $r=2$ (right).}
  \label{fig:mesh-r2}
\end{figure}
\begin{figure}
  \centering
  \begin{subfigure}[b]{0.48\textwidth}
    \centering
    \includegraphics[height=3cm]{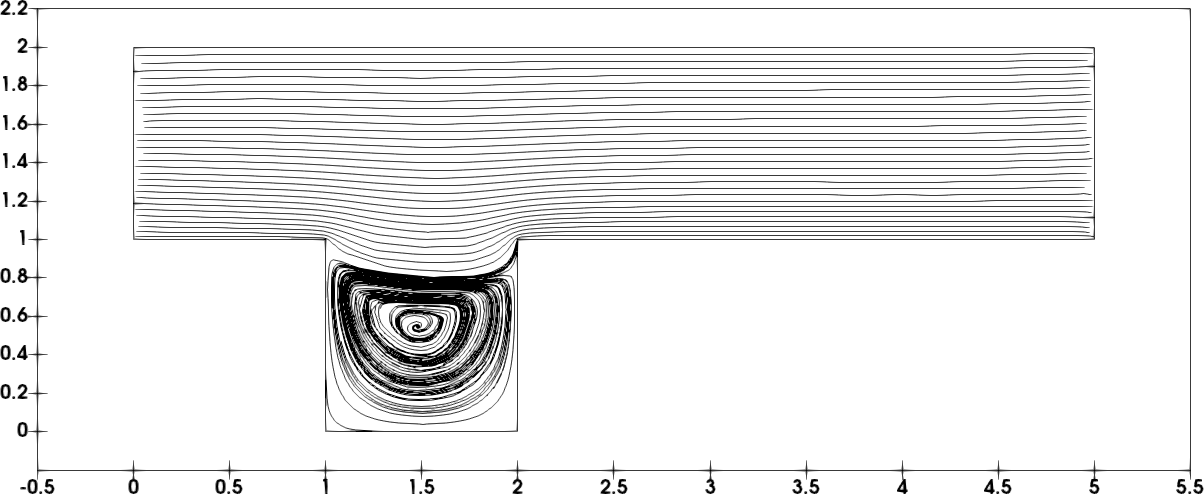}
    \caption{$r=1.5$}
  \end{subfigure}
  \hfill
  \begin{subfigure}[b]{0.48\textwidth}
    \centering
    \includegraphics[height=3cm]{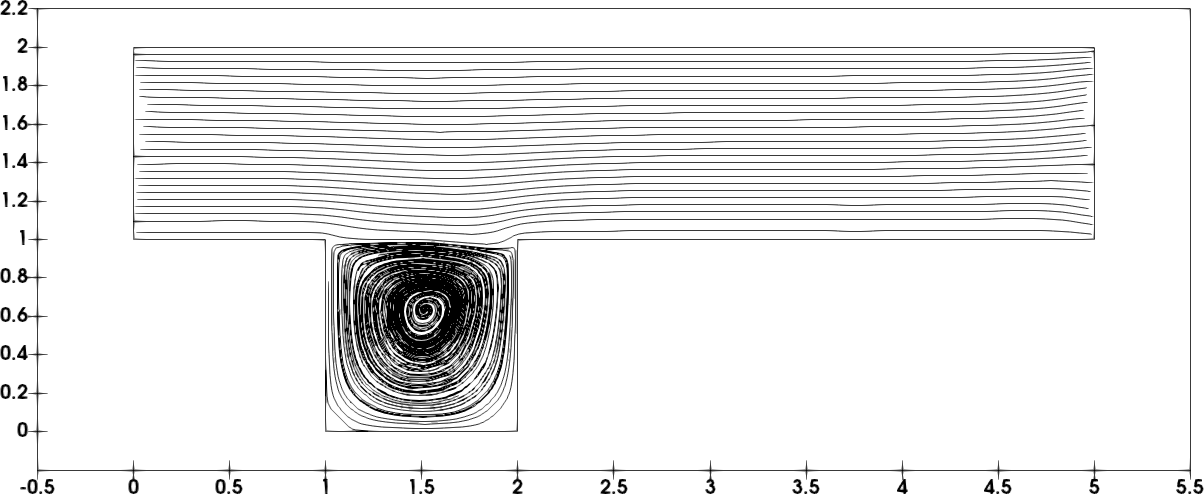}
    \caption{$r=2.5$}
  \end{subfigure}
  \caption{Test 2. Streamlines at $\tF = 10$.}
  \label{fig:streams}
\end{figure}

\begin{figure}
  \centering
  \begin{subfigure}[b]{0.45\textwidth}
    \centering
    \includegraphics[width=\linewidth]{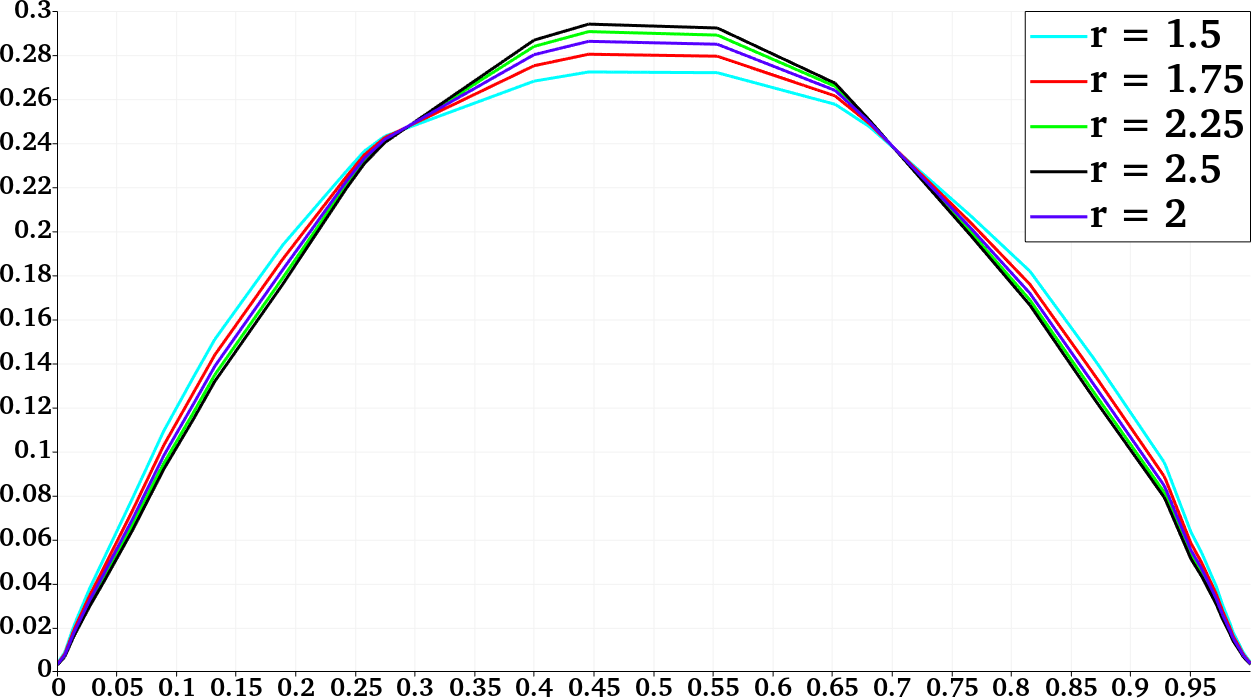}
    \caption{Cutline at $x=0.5$}
  \end{subfigure}
  \hfill
  \begin{subfigure}[b]{0.45\textwidth}
    \centering
    \includegraphics[width=\linewidth]{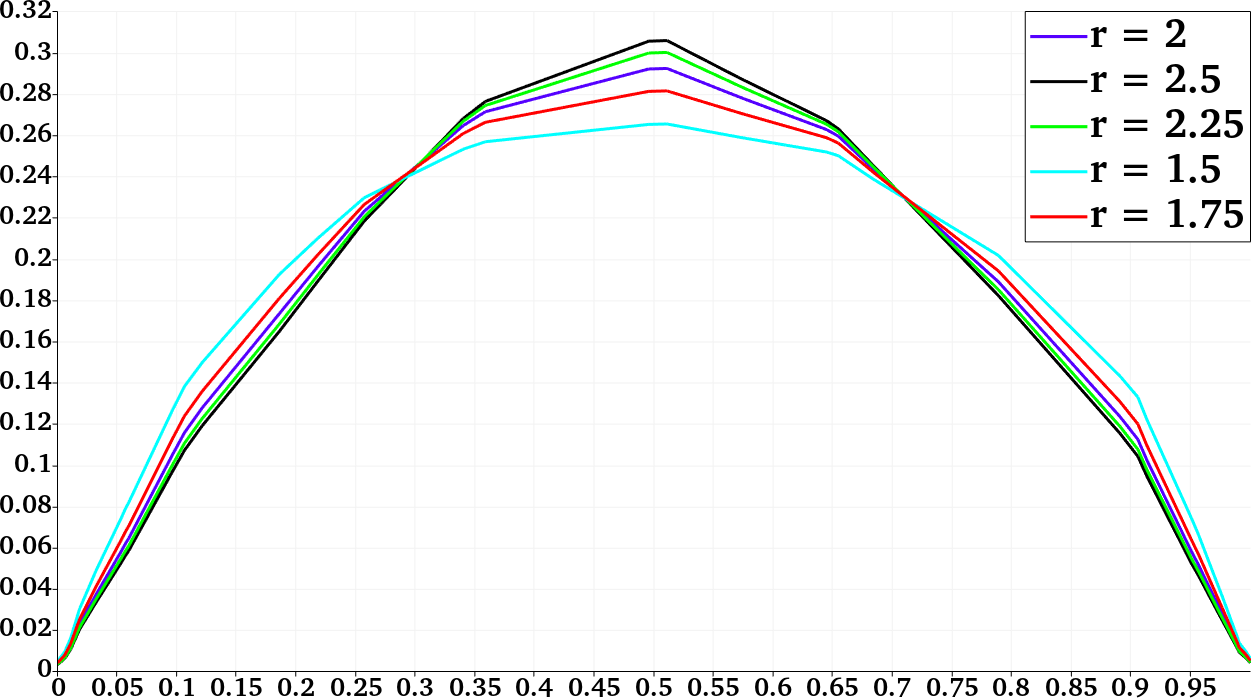}
    \caption{Cutline at $x=2.5$}
  \end{subfigure}

  \vspace{0.5cm}
  \centering
  \begin{subfigure}[b]{0.45\textwidth}
    \centering
    \includegraphics[width=\linewidth]{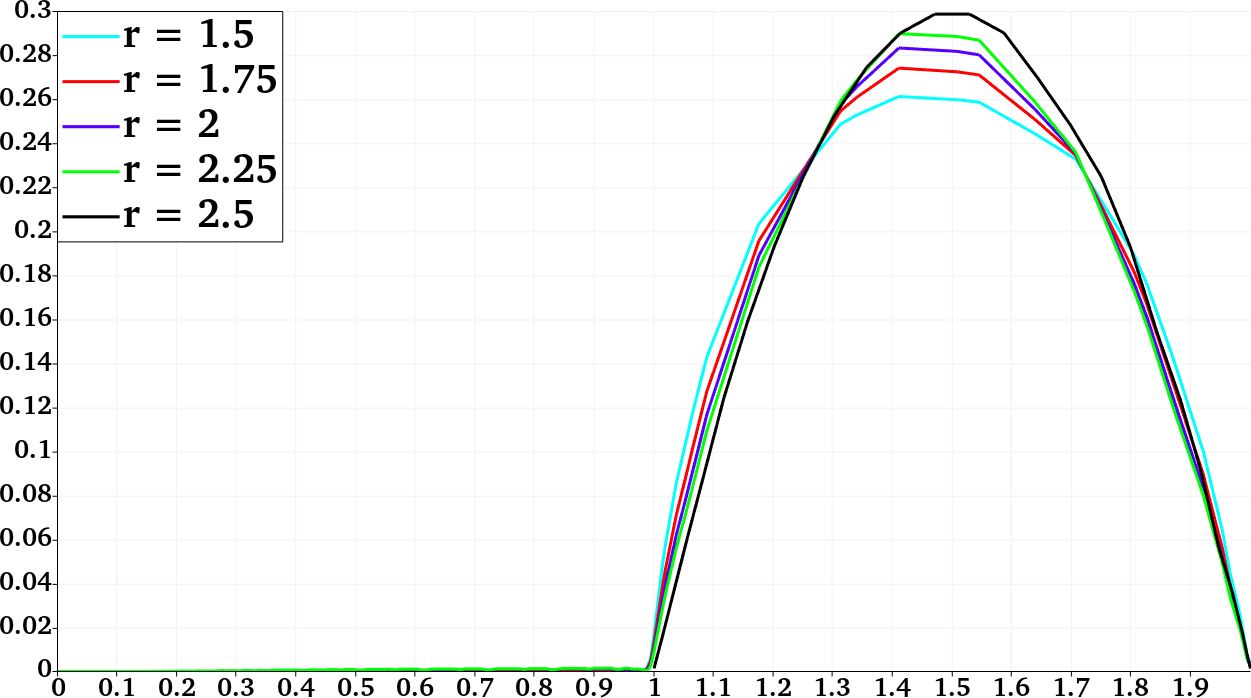}
    \caption{Cutline at $x=1$}
  \end{subfigure}
  \vspace{0.5cm}

  \caption{Test 2. Horizontal velocities at selected cutlines of the channel, T=10.}
  \label{fig:cutlines}
\end{figure}

\section*{Acknowledgements}
All the authors have been partially funded by the European Union (ERC Synergy, NEMESIS, project number 101115663).
Views and opinions expressed are however those of the authors only and do not necessarily reflect those of the European Union or the European Research Council Executive Agency.

\printbibliography

\end{document}